\RequirePackage{luatex85} % see
\documentclass[11pt]{amsart}
\usepackage{amssymb}
\usepackage{latexsym}
\usepackage{graphicx}
\usepackage{subfigure}
\usepackage{tikz}
\usepackage[linktocpage=true]{hyperref}
\usepackage[left=3.8cm, right=3.8cm, top=3cm, bottom=3cm]{geometry}
\usepackage{color}
\usepackage{here}
\usepackage[backgroundcolor = lightgray,
            linecolor = lightgray,
            textsize = tiny]{todonotes}
\usepackage{stackengine}
\usepackage{mathrsfs}
\usepackage[all]{xy}

\hypersetup{ colorlinks   = true, %Colours links instead of ugly boxes
urlcolor  = blue, %Colour for external hyperlinks
linkcolor    = blue, %Colour of internal links
citecolor   = blue %Colour of citations
}

\newcommand{\bsym}{\billiards_{\mathrm{sym}}}

\newcommand{\threevec}[3]{
  \left(\begin{array}{c}
    #1\\#2\\#3
  \end{array}\right)}

\def\mls{{\mathcal{MLS}}}

\def\I{{\mathrm I}}
\def\II{{\mathrm{II}}}
\def\III{{\mathrm{III}}}

\def\beq{\begin{equation}}
\def\eeq{\end{equation}}

\def\det{\mathrm{det}\ }

\def\ba{\bar a}
\def\bg{\bar{\gamma}}
\def\bz{\bar{\zeta}}
\def\bgg{\bar{g}}

\newcommand{\MLS}{\mathcal{MLS}}
\newcommand{\mlsinv}{$\MLS$-invariant}%
\newcommand{\ie}{i.e.\ }
\newcommand{\eg}{e.g.\ }
\newcommand{\resp}{resp.\ }
\newcommand{\billiards}{\mathbf{B}}

\newcommand{\Z}{{\mathbb Z}}
\newcommand{\R}{{\mathbb R}}

\newcommand{\C}{{\mathbb C}}

\newcommand{\T}{{\mathbb T}}

\newcommand{\N}{{\mathbb N}}

\newtheorem{theorem}{Theorem}[section]
\newtheorem{remark}[theorem]{Remark}

\newtheorem{lemma}[theorem]{Lemma}

\newtheorem{defi}[theorem]{Definition}
\newtheorem{prop}[theorem]{Proposition}
\newtheorem{corollary}[theorem]{Corollary}
\newtheorem{claim}[theorem]{Lemma}

\newtheorem*{main thm}{Main Theorem}
\newtheorem*{main thm bis}{Main Theorem (alternate version)}
\newtheorem*{thmrig}{Theorem}

\newtheorem{question}{Question}
\newcommand{\dom}{\mathcal D}
\newcommand{\obs}{\mathcal O}

\begin{document}
\numberwithin{equation}{section}

\title[Marked Length Spectral determination of analytic chaotic
	billiards]{Marked Length Spectral determination of analytic chaotic
  billiards with axial symmetries}

\author[Jacopo De Simoi]{Jacopo De Simoi$^*$}
\thanks{$^*$ J.D.S. and M.L. have been partially supported by the NSERC Discovery grant, reference number 502617-2017}
\address{$^{*,\ddagger}$ Department of Mathematics, University of Toronto, 40 St George St., Toronto, ON, Canada M5S 2E4}
\email{jacopods@math.utoronto.ca}

\author[Vadim Kaloshin]{Vadim Kaloshin$^\dagger$}
\thanks{$^\dagger$ V.K. acknowledges partial support of the
	NSF grant DMS-1402164 and ERC Grant \# 885707}
\address{$^\dagger$ Institute of Science and Technology Austria (ISTA)}
\email{vadim.kaloshin@gmail.com}

\author[Martin Leguil]{Martin Leguil$^\ddagger$}
\address{$^\ddagger$ Laboratoire Ami\'enois de Math\'ematiques Fondamentales et Appliqu\'ees (LAMFA, UMR 7352)\\
	Université de Picardie Jules Verne, 33 rue Saint Leu, 80039 Amiens, France}
\email{martin.leguil@u-picardie.fr}

\begin{abstract}
  We consider billiards obtained by removing from the plane finitely
  many strictly convex analytic obstacles satisfying the non-eclipse
  condition.  The restriction of the dynamics to the set of
  non-escaping orbits is conjugated to a subshift, which provides a
  natural labeling of periodic orbits.  We show that under suitable
  symmetry and genericity assumptions, the Marked Length Spectrum
  determines the geometry of the billiard table.
\end{abstract}

\maketitle
\tableofcontents

\section*{Introduction}
Billiard dynamics studies the behavior of a point particle which moves
freely on some domain and undergoes elastic reflections upon collision
with the domain boundary.  In this article we consider the class of
domains obtained by removing from the plane a number $m \geq 3$ of
strongly convex\footnote{ A set is strongly convex if its boundary has
  strictly positive curvature.} compact sets with analytic boundary.
Systems of this type were first considered in~\cite{GR}, where the
authors studied the classical scattering of a point particle from
three circular disks on the plane.  Strong convexity of the obstacles
allows to construct stable and unstable cones for the billiard map,
hence implies that the corresponding billiard dynamics enjoys strong
hyperbolicity properties, see \eg~\cite[Section 4.4]{CM},
\cite{Sin,W1,W2}. Hyperbolicity, together with a
standard non-eclipse condition (see Section~\ref{sec prelim}), allows
to encode the dynamics of non-escaping trajectories as a subshift of
finite type on $m$ symbols (see \eg~\cite{Mor} or~\cite[Section
2.2]{PS2}).  This observation provides, in particular, a natural
marking of each periodic orbit with the associated encoding.  The
\textit{Marked Length Spectrum} is then defined as the set of all
lengths of periodic orbits together with their marking (see
Definition~\ref{def:mls}).

Two domains that have the same Marked Length Spectrum are said to be
\emph{marked-length-isospectral}.  For instance, two isometric domains
are necessarily isospectral.  On the other hand, it is a fascinating
problem to characterize marked-length-isospectral billiards modulo
isometries.  We refer to this problem as the \emph{dynamical inverse
  spectral problem}.

The main result of this article is that, under suitable symmetry and
genericity assumptions, the Marked Length Spectrum of a domain is
sufficient to reconstruct the domain (up to isometries).

In order to describe our results in some context, let us first present
some related classical problems and corresponding results.\medskip

\textbf{The Laplace inverse spectral problem.}  The dynamical inverse
spectral problem introduced above is deeply related to the question
that M. Kac (see~\cite{K}) famously phrased as: \textit{``Can one hear
  the shape of a drum?''}, \ie is the shape of a planar domain
determined by its Laplace Spectrum (with either Dirichlet or Neumann
boundary conditions)?  The relation between the dynamical and the
Laplace problem is apparent, for instance, in the thesis of Colin de
Verdi\`ere~\cite{CdV1,CdV2}, where it is shown that the Laplace
Spectrum determines the Length Spectrum of a generic manifold, and in
the trace formula proved by Andersson--Melrose~\cite{AM}: generalizing
previous results by Chazarain~\cite{Ch},
Duistermaat--Guillemin~\cite{DG}, they showed that, for strictly
convex\footnote{ The trace formula (and its consequences) also holds
  without the convexity assumption for planar domains, and for convex
  higher dimensional domains (see \eg~\cite{PS2}).}  $C^\infty$
domains, the singular support of the wave trace is contained in the
Length Spectrum (in fact, generically, the two sets coincide) .  In
particular, in this setting, the Laplace Spectrum generically
determines the Length Spectrum.  Similarly, there is a connection
between Laplace Spectrum and Length Spectrum in hyperbolic situations:
indeed, the Selberg trace formula \cite{Se} shows that the Laplace
Spectrum determines the Length Spectrum on hyperbolic manifolds, or
for generic Riemannian metrics. In particular, in the case of closed
hyperbolic surfaces, the stronger statement that the Laplace Spectrum
determines the Length Spectrum and vice-versa holds
(see~\cite{H1,H2}).\medskip

\textbf{Spectral determination and spectral rigidity for convex
  domains with symmetries.}  It has been famously proven by Zelditch
in a series of papers (see~\cite{Z1,Z2,Z3}) that the Laplace Spectrum
completely determines (modulo isometries) the domain in a generic
class of analytic $\Z_2$-symmetric (i.e., symmetric with respect to
some axis of reflection) planar convex domains.
Hezari--Zelditch~\cite{HZ2} have obtained a higher dimensional
analogue of this result: bounded analytic domains in $\R^n$ with
reflection symmetries across all coordinate axes, and with one axis
height fixed (satisfying some generic non-degeneracy conditions) are
spectrally determined among other such domains.  Notably, the same
authors have shown in a very recent preprint (see~\cite{HZ3}) that
ellipses of small eccentricity are identified by their Laplace
spectrum among smooth domains with the symmetries of the ellipse.

Spectral determination results are, currently, far beyond reach for
smooth (non-elliptic) domains.  In the last decade, however,
interesting results have appeared for a deformation version of
spectral determination known as \emph{spectral rigidity}\footnote{
  Recall that a domain is said to be \emph{spectrally rigid} if any
  Laplace isospectral \emph{continuous deformation} is necessarily
  isometric.}.  In~\cite{HZ1}, Hezari--Zelditch have shown the
following result: given a domain bounded by an ellipse (of arbitrary
eccentricity), any one-parameter isospectral $C^\infty$ deformation
which additionally preserves the $\Z_2\times \Z_2$ symmetry group of
the ellipse is necessarily flat (i.e., all derivatives have to vanish
at the initial parameter).

The problem of dynamical spectral determination was studied by Colin
de Verdi\`ere; in~\cite{CdV3}, he has shown that, in the class of
convex analytic billiards with the symmetries of the ellipse, the
Marked Length Spectrum determines the domain geometry.  In the smooth
category, in~\cite{DKW}, the authors proved that any sufficiently
smooth $\Z_2$-symmetric strictly convex domain sufficiently close to a
circle is dynamically spectrally rigid, i.e., all deformations among
domains in the same class which preserve the length of all periodic
orbits of the associated billiard flow must necessarily be isometries;
see also the numerical exploration \cite{AD} by Ayub-De Simoi for
ellipses of eccentricity smaller than $0.30$. Moreover, recently, for
smoothly conjugate billiard maps of Birkhoff billiards,
Kaloshin-Koudjinan \cite{KK} studied rigidity in the form of
Marvizi-Melrose invariants.

\medskip \textbf{The Inverse Problem for flat billiards.}
Spectral rigidity questions have also been explored in the context of flat billiards. In \cite{DELS}, the authors study the relationship between the shape of a Euclidean polygon and the symbolic dynamics of its billiard flow. They introduce a Bounce Spectrum for polygons, and show that two simply connected Euclidean polygons with the same Bounce Spectrum are either similar or right-angled and affinely equivalent.

\medskip\textbf{Determination of the geometry of a negatively curved
  surface by the Marked Length Spectrum.} Another natural setting
where similar questions were investigated is for geodesic flows on
negatively curved surfaces. As in the present work, the dynamics of
such flows is hyperbolic.\footnote{Yet, in that case, the geodesic
  flow is a genuine \emph{Anosov} flow, while for open dispersing
  billiards, the interesting dynamics occurs only on a Cantor set.}  A
famous result due independently to Otal \cite{O} and Croke \cite{Cr}
show that if $g_0$ and $g_1$ are negatively curved metrics on a closed
surface with the same Marked Length Spectrum, then $g_1$ is isometric
to $g_0$. In higher dimension, Guillarmou--Lefeuvre in \cite{GL}
recently proved that the Marked Length Spectrum of a Riemannian
manifold $(M, g)$ with Anosov geodesic flow and non-positive curvature
locally determines the metric.

\medskip \textbf{The Laplace Inverse Resonance Problem.} %
A dual (or exterior) formulation of the inverse spectral problem is
the inverse resonance problem, in which one attempts to reconstruct an
unbounded domain (\eg~the complement of a finite number of convex
scatterers) by the resonances (i.e., the poles) of the resolvent
$(\Delta-z^2)^{-1}$ (see \eg~\cite{PS2,Zw2,Z4}).  In particular,
in~\cite{Z4}, Zelditch showed that a $\Z_2$-symmetric configuration
of two convex analytic obstacles in the plane $\R^2$ is determined by
its Dirichlet or Neumann resonance poles.  This result is the direct
analogue, for exterior domains, of the proof that a $\Z_2$-symmetric
bounded simply connected analytic plane domain is determined by the
Laplace eigenvalue spectrum. The proof is based on the fact that wave
invariants of an exterior domain are resonance invariants and on the
method of~\cite{Z2,Z3} for calculating the wave invariants explicitly
in terms of the boundary defining function. In~\cite{ISZ}, the authors
give another proof of the inverse result with two symmetries using
Birkhoff Normal Forms of the billiard map and quantum monodromy
operator rather than the Laplacian, applying some results on
semi-classical trace formulae and on quantum Birkhoff normal forms for
semi-classical Fourier integral operators to inverse
problems. Generalizing results of~\cite{G}, they showed that the
classical Birkhoff Normal Form can be recovered from semi-classical
spectral invariants, and in fact, that the full quantum Birkhoff
Normal Form of the quantum Hamiltonian near a closed
orbit, and infinitesimally with respect to the energy can be recovered.\\\

\textbf{Dynamical inverse spectral problems for billiards with
  hyperbolic dynamics.}  We finally come to the setting that will be
explored in this article.  In~\cite{BDSKL}, the authors (together with
P. Bálint) study billiards obtained by removing $m \geq 3$ strictly
convex finitely smooth obstacles from $\R^2$; it is stated (without a
complete proof) that the Marked Length Spectrum determines the
Lyapunov exponent of each periodic orbit; in
Appendix~\ref{sec:bdkl-symmetric-case} we give a proof of a slightly
weaker result that suffices for our current purposes.  The proof is
similar to the one sketched for~\cite[Corollary E]{BDSKL}.

It is a crucial observation that, unlike billiards inside convex
domains, billiards inside the domains considered in this article are
\emph{open systems}, \ie there exist initial conditions (in fact, a
full-measure set) for which the billiard ball escapes to infinity
after finitely many bounces.  As a consequence, it will be the case
that no periodic trajectory visits certain regions of the
configuration space.  This implies that periodic spectral data will
never suffice to recover the geometry of the unexplored region.
Hence, one can either attempt to recover the full geometry under
additional rigidity assumptions (\eg~analyticity of the scatterers),
or consider the question of determination restricted to the explored
region.

In this paper, we pursue the first strategy and assume that all
scatterers have real analytic curves as boundary.  Our goal is
achieved once we recover the full jet of the curvature at some point
on each scatterer.  Due to analyticity, this entirely determines the
geometry of the scatterers.  The result stated below as Main Theorem,
asserts that this is indeed possible, provided that two scatterers
have some symmetries (similar to the ``bi-symmetric'' setting
of~\cite{Z1}).

In our proof, we reconstruct from spectral data the classical
(hyperbolic) Birkhoff Normal Form of the two-periodic orbit associated
to the symmetric scatterers. This can be done, provided that a
genericity condition is satisfied, by analyzing some asymptotics in
the Marked Length Spectrum relative to periodic orbits that
approximate homoclinic orbits of the two-periodic orbit.

Once the normal form has been obtained, exploiting the symmetries of
our system, and some extra information that can be obtained by the
Marked Length Spectrum, it is possible to reconstruct the geometry of
the billiard.  A more detailed explanation of the proof will be given
in Section~\ref{sec:idea-proof}, after introducing some necessary
preliminaries.

Our results are an analogue of those presented in~\cite{CdV3} for the
class of chaotic billiards under consideration, or an analogue in
terms of the Marked Length Spectrum of~\cite{Z1} (see
also~\cite{ISZ}).

Note that due to the convexity of the obstacles, the presence of more
than two scatterers in our case is crucial to guarantee the existence
of a large set of periodic orbits.  On the other hand, billiard
trajectories in the exterior of only two strictly convex domains in
the plane were considered by Stoyanov in~\cite{Sto}.

\section{Definitions and statement of our main results}\label{sec prelim}

In the present paper, we consider billiard tables $\dom\subset \R^2$
given by $\dom = \R^{2}\setminus \bigcup^{m}_{i = 1}\obs_{i}$, for
some integer $m \geq 3$, where each $\obs_{i}$ is a convex domain with
analytic boundary $\partial\obs_{i}$.
%. In
%this paper, we assume that each $\partial\obs_{i}$ is analytic.
We
refer to each of the $\obs_{i}$'s as \emph{obstacle} or
\emph{scatterer}. We let $\ell_i:=|\partial\obs_{i}|$ be the
corresponding lengths,  set $\T_i:=\R/\ell_i\Z$, and parametrize
each $\partial\obs_{i}$ in arc-length, for some analytic map
$\Upsilon_i \in C^\omega( \T_i, \R^2)$, $s\mapsto \Upsilon_i(s)$. We
assume that the following condition holds:\\

\noindent
\textsc{Non-eclipse condition:} The convex hull  of any two scatterers
is disjoint from the other $m -2$ scatterers.\\

The set of all billiard tables obtained by removing from the plane
$m $ strictly convex analytic obstacles satisfying the non-eclipse
condition will be denoted by $\billiards(m)$.

Fix
$\dom = \R^{2}\setminus \bigcup^{m}_{i = 1}\obs_{i} \in
\billiards(m)$. We denote the collision space by
\begin{align*}
  \mathcal{M} &= \bigcup_i \mathcal{M}_i, &% \\
  \mathcal{M}_i &=\{(q,v),\ q \in \partial\obs_{i},\ v\in \R^2,\ \|v\|=1,\ \langle v,n\rangle\geq 0\},
\end{align*}
where $n$ is the unit normal vector to $\partial\obs_{i}$ pointing
inside $\mathcal{D}$. For each $x=(q,v) \in \mathcal{M}$, $q$ is
associated with the arclength parameter $s \in [0,\ell_i]$ for some
$i\in \{1,\cdots,m\}$, i.e., $q=\Upsilon_i(s)$. We let
$\varphi\in [-\frac{\pi}{2},\frac{\pi}{2}]$ be the oriented angle
between $n$ and $v$ and set $r:=\sin(\varphi)$. In other words, each
$\mathcal{M}_i$ can be seen as a cylinder $\T_i \times [-1,1]$ endowed
with coordinates $(s,r)$. \label{sec:sr-coordinates} In the following,
given a point $x=(q,v) \in \mathcal{M}$ associated with the pair
$(s,r)$, we also denote by $\Upsilon(s):=q$ the point of the table
defined as the projection of $x$ onto the $q$-coordinate. Moreover,
for each pair $(s,r), (s',r') \in \mathcal{M}$, we denote by
\begin{equation}\label{def hsspremie}
h(s,s'):=\|\Upsilon(s')-\Upsilon(s)\|
\end{equation}
the Euclidean length of the segment connecting the associated points
of the table.

Set $\Omega :=\{(q,v) \in \mathcal{D} \times \mathbb{S}^1\}$. Denote by $\Phi^t\colon \Omega\to \Omega$ the flow of the billiard
and let
\begin{align*}
  \mathcal{F}=\mathcal{F}(\mathcal{D}) \colon \mathcal{M} \to \mathcal{M},\quad x \mapsto \Phi^{\tau(x)+0}(x)
\end{align*}
be the associated billiard map, where
$\tau\colon\mathcal{M} \to \R_+\cup \{+\infty\}$ is the first return
time, and $\Phi^{\tau(x)+0}(x)$ is the image of the point $x$ right
after the collision at time $\tau(x)$.  For any point
$x=(s,r)\in \mathcal{M}$ such that $(s',r'):=\mathcal{F}(s,r)$ is
well-defined, we denote by $\mathscr{L}:=h(s,s')$ the distance between
the two points of collision, we let $\mathcal{K}:=\mathcal{K}(s)$,
$\mathcal{K}':=\mathcal{K}(s')$ be the respective curvatures, and set
$\nu:=\sqrt{1-r^2}$, $\nu':=\sqrt{1-(r')^2}$.  It follows
from~\cite[(2.26)]{CM} that
\begin{equation}\label{matrice sl deux}
D\mathcal{F}_{(s,r)}=-\begin{pmatrix}
\frac{1}{ \nu'}(\mathscr{L} \mathcal{K}+\nu) & \frac{\mathscr{L}}{\nu \nu'}\\
\mathscr{L} \mathcal{K}\mathcal{K}'+\mathcal{K} \nu'+\mathcal{K}'\nu & \frac{1}{\nu}(\mathscr{L} \mathcal{K}'+\nu')
\end{pmatrix}\in \mathrm{SL}(2,\R),
\end{equation}
and  the map $\mathcal{F}$ is symplectic for the form $ds \wedge dr$.

Due to the convexity of the obstacles, for each
$i,j \in\{1,\cdots,m\}$, with $i\ne j$, there exist
$0\leq a_i^j\leq b_i^j\leq \ell_i$, and for each parameter
$s\in [a_i^j,b_i^j]$, there exists a non-empty closed interval
$I_i^j(s) %, I_i^k(s)
\subset [-1,1]$ such that $\tau(x)<+\infty$, if
$x=(s,r) \in \widetilde{\mathcal{M}}_i:=\bigcup_{k\neq
  i}\widetilde{\mathcal{M}}_i^k$, and $\tau(x)=+\infty$, if
$x \in \mathcal{M}_i \backslash \widetilde{\mathcal{M}}_i$,
where
\begin{align*}%
\widetilde{\mathcal{M}}_i^j
:=\{(s,r)\in \mathcal{M}_i: s\in [a_i^j,b_i^j],\ r \in I_i^j(s)\}=\mathcal{M}_i \cap \mathcal{F}^{-1} (\mathcal{M}_j).
% \quad *\in\{j,k\},
\end{align*}%
In other words, %for each $i \in\{1,\cdots,m\}$, and for
%$j\in \{1,\cdots,m\}\setminus \{i\}$,
$[a_i^j,b_i^j]$ is the interval of parameters $s$ of points sitting at
$\partial \obs_i$ which can be joined to some point on
$\partial \obs_j$ for some suitable angle, namely
$r\in I_i^j(s)$.

In particular, the set of trajectories that do not escape to infinity
is given by
\begin{align*}%
  \bigcap_{k \in \Z} \mathcal{F}^{-k}(\widetilde{\mathcal{M}}),\quad \widetilde{\mathcal{M}}:=\bigcup_{j\neq i}\widetilde{\mathcal{M}}_j,
\end{align*}%
and is homeomorphic to a Cantor set.  Due to the non-eclipse
condition, the restriction of the dynamics to this set is conjugated
to a subshift of finite type associated with the transition matrix
\begin{align*}
  \begin{pmatrix}
     %1 & 1 & \cdots & 1\\
     0 & 1 & \cdots & 1\\
     1 & \ddots & \ddots & \vdots\\
     \vdots & \ddots & \ddots  & 1\\
     1 & \cdots & 1  & 0
  \end{pmatrix}.
\end{align*}
In other words, any word $(\varsigma_j)_{j}\in \{1,\cdots,m\}^\Z$ such that
$\varsigma_{j+1}\neq \varsigma_j$ for all $j \in \Z$ can be realized
by an orbit, and by hyperbolicity of the dynamics, this orbit is
unique. Such a word is called \textit{admissible}. Besides, this
marking is unique provided that we fix a starting point in the orbit
and an orientation.

In particular, any periodic orbit of period $p$ (observe that
necessarily $p\ge 2$) can be labeled by a finite admissible word
$\sigma=(\sigma_1\sigma_2\dots\sigma_p)\in \{1,\cdots,m\}^p$, such that the
infinite word $\sigma^\infty:=\dots\sigma\sigma\sigma\dots$ is
admissible (or equivalently, such that
$\sigma_{j}\neq \sigma_{j+1\mod p}$, for all $j\in \{1,\cdots,p\}$).
We denote by $\mathrm{Adm}$ the set of finite admissible words
$\sigma\in \cup_{p \geq 2} \{1,\cdots,m\}^p$.

Given any word
$\sigma = (\sigma_1\sigma_2\dots\sigma_p) \in \mathrm{Adm}$, we denote
by $\overline{\sigma}$ be the transposed word
\begin{align*}
  \overline{\sigma}:=(\sigma_p \sigma_{p-1}\dots\sigma_1).
\end{align*}
The word $\overline\sigma$ encodes the same periodic trajectory as
$\sigma$, but with opposite orientation.

As explained above, for any $j \in \{1,\cdots,p\}$, the
$j^{\text{th}}$ symbol $\sigma_j$ of $\sigma$ corresponds to a point
$x(j)$ in the trajectory, where
$x(j)=(s(j),r(j))\in\mathcal{M}_{\sigma_{j}}$ is represented by
position and angle coordinates.  For all $k \in \Z$, we also extend
the previous notation by setting $\sigma_k:=\sigma_{k \mod p}$, and similarly for $x(k)$,
$s(k)$ and $r(k)$.

\begin{defi}\label{def:mls}
  The \textit{Marked Length Spectrum} $\mathcal{MLS}(\mathcal{D})$ of
  $\mathcal{D}$ is defined as the function
  \begin{equation}\label{marked spectrum}
    \mathcal{L}\colon \mathrm{Adm} \to \R_+,\quad \sigma \mapsto \mathcal{L}(\sigma),
  \end{equation}
  where $\mathcal{L}(\sigma)$ is the length of the periodic orbit
  identified by $\sigma$, obtained as the sum of the lengths of all
  the line segments that compose it.

  An object is said to be a \emph{\mlsinv} if it can be obtained by
  the sole knowledge of the Marked Length Spectrum.
\end{defi}

For any periodic orbit $(x_1,\cdots, x_p)$ encoded by a word $\sigma$
of length $p \geq 2$, we have
$D_{x_j}\mathcal{F}^p\in \mathrm{SL}(2,\R)$, for
$j \in \{1,\cdots,p\}$.  Due to the strong convexity of the obstacles,
each of the matrices $D_{x_j}\mathcal{F}^p$ is hyperbolic, and we
denote by $\lambda(\sigma)<1<\lambda(\sigma)^{-1}$ their (common)
eigenvalues. We define the \textit{Lyapunov exponent} of this orbit as
\begin{align}\label{eq:le-definition}
\mathrm{LE}(\sigma):=-\frac 1p \log
  \lambda(\sigma) >0.
\end{align}
\begin{defi}\label{def:MLES-definition}
  The \textit{Marked Lyapunov Spectrum} of the billiard table
  $\mathcal{D}$ is defined as the function
  \begin{align}\label{maked lepnov}
    \mathrm{LE}\colon \mathrm{Adm} \to \R_+,\quad  \sigma \mapsto \mathrm{LE}(\sigma).
  \end{align}
\end{defi}
We conclude this section by recalling an important symmetry of the
billiard dynamics, which will be crucial in the following.  Let us
denote by $\mathcal{I}$ the involution
$\mathcal{I}\colon (s,r) \mapsto (s,-r)$. It conjugates the billiard
map $\mathcal{F}$ with its inverse $\mathcal{F}^{-1}$, according to
the time reversal property of the billiard dynamics:
\begin{align*}
  \mathcal{I} \circ \mathcal{F} \circ \mathcal{I}=\mathcal{F}^{-1}.
\end{align*}
A periodic orbit of period $p \geq 2$ is called \textit{palindromic}
if it can be labeled by an admissible word
$\sigma \in \{1,\cdots,m\}^p$ such that
$\sigma=(\sigma_1\dots\sigma_{q-1}\sigma_q\sigma_{q-1}\dots\sigma_1\sigma_0)$
for certain symbols
$(\sigma_0,\sigma_1,\cdots,\sigma_q)\in \{1,\cdots,m\}^{q+1}$.
Observe that $p = 2q$, hence we gather that palindromic orbit always
have even period.  As we shall see later, there is a connection
between the palindromic symmetry and the time reversal property
recalled above.  In particular, by the palindromic symmetry and by
expansiveness of the dynamics, the associated trajectory hits the
billiard table perpendicularly at the points associated to the symbols
$\sigma_0$ and $\sigma_q$.

For more details about chaotic billiards and inverse spectral
problems, we refer the reader to the books of
Chernov--Markarian~\cite{CM} and Petkov--Stoyanov~\cite{PS2}.

\subsection{Spectral determination}\label{subs spectrllllll}

Recall that $\billiards(m)$ is the set of all billiard tables
$\mathcal{D}$ formed by $m\geq 3$ convex analytic obstacles satisfying
the non-eclipse condition, that $\mathcal{F}(\mathcal{D})$ denotes the
associated billiard map, and that $\mathcal{K}$ is the curvature
function. We introduce a class of tables with two additional
symmetries, which, without loss of generality, are assumed to be
associated with the obstacles $\obs_1,\obs_2$.
\begin{defi}\label{defi sym}
We let  $\billiards_{\mathrm{sym}}(m)\subset \billiards(m)$ be the subset of all  billiard tables $\mathcal{D}= \R^{2}\setminus \bigcup^{m}_{i = 1}\obs_{i}$ which are symmetric in the following sense:
\begin{itemize}
\item %the table $\mathcal{D}$ is invariant under the reflection $\mathcal{R}$ along the   line segment bisector of the trace of the orbit $(jk)$;
the jets of   $\mathcal{K}$ are the same at the endpoints of the $2$-periodic orbit $(12)$;
\item the jets of $\mathcal{K}|_{\mathbb{T}_1}$, $\mathcal{K}|_{\mathbb{T}_2}$ are even, % at $0_j,0_k$,
assuming that $0_1\in \mathbb{T}_1$, $0_2\in\mathbb{T}_2$ are the arc-length parameters of the endpoints of the orbit $(12)$.
\end{itemize}
In particular, by analyticity, the pair of obstacles
$\obs_1,\obs_2$ has some $\Z_2\times \Z_2$-symmetry:
$\obs_1,\obs_2$ are images of each other by the reflection along the
line segment bisector of the trace of the orbit
$(12)$, and each of them is symmetric with respect to the line through
the endpoints of
$(12)$.
\end{defi}

The reason for requiring the two symmetries will be clarified in
Remark~\ref{rem:symmetries} below.  In the following, we let
$\T:=\R/(2\pi\Z)$, and let
$C^\omega(\T,\R)$ be the space of
$2\pi$-periodic real analytic
functions.  %For any positive number $h>0$, we denote by $C^\omega_h(\T,\R)$ the Banach space of real analytic functions $f \in C^\omega(\T,\R)$ admitting a holomorphic extension to $|\Im z|<h$ which are continuous up to the boundary.
%We
Given $r>0$, we denote by $C_r^\omega(\T,\R)\subset C^\omega(\T,\R)$
the subspace of (bounded) analytic functions defined on the
strip\footnote{~Given a complex number $z$, we denote its real (\resp
  imaginary) part with $\Re z$ ({\resp $\Im
    z$}).} $\{z \in \C/(2\pi \Z):|\Im z|<r\}$ and extending continuously
to the boundary, and for any $f \in C_r^\omega(\T,\R)$, we let
%endowed with the norm %$C^\omega(\T,\R)$
%$|\cdot|_r$, where
$|f|_r:=\sup_{\{|\Im z|<r\}} |f(z)|$.
%, for
%$f \in C_r^\omega(\T,\R)$.
We denote by $C^\omega(\T,\R^2)$ the
space of analytic functions
$f\colon \theta \mapsto (f_1(\theta),f_2(\theta))$, with
$f_1,f_2 \in C^\omega(\T,\R)$. Similarly, for any $r>0$, we denote by $ C_r^\omega(\T,\R^2)\subset C^\omega(\T,\R^2)$ the subspace of functions $f=(f_1,f_2) \in (C_r^\omega(\T,\R))^2$ %such that $f_1,f_2\in C_r^\omega(\T,\R)$,
endowed with the norm $\|\cdot\|_r$,
%and let %
where
$\|f\|_r:=\max(|f_1|_r,|f_2|_r)$.
\begin{defi}[Topology on
  $\billiards_{\mathrm{sym}}(m,r)$]\label{topology}
  Let $\mathbf{Conv}\subset C^\omega(\T,\R^2)$ be the set of all
  functions $f \in C^\omega(\T,\R^2)$ such that $f(\T)$ is a simple
  closed curve such that the interior region bounded by $f(\T)$ is
  strongly convex, i.e., the curvature of $f(\T)$ never vanishes.  We
  denote by $\obs(f)$ the interior region bounded by $f(\T)$. For any
  $r >0$, we let
  $\mathbf{Conv}_r:= \mathbf{Conv}\cap C_r^\omega(\T,\R^2)$.  For any
  integer $m \geq 3$, we thus get a map
\begin{equation*}
	\Phi=\Phi(m) \colon 	 \mathbf{Conv}^m\ni (f^{(i)})_{i=1,\cdots,m}  \mapsto  \mathcal{D}:=\R^2 \setminus \bigcup_{i=1}^m \obs(f^{(i)} ).
\end{equation*}
Given $r>0$, let
$\mathbf{W}_{\mathrm{sym}}(m,r):=\Phi^{-1}(\billiards_{\mathrm{sym}}(m))\cap
\mathbf{Conv}^m_r$, and endow it with the topology induced by the
product topology on $(C_r^\omega(\T,\R^2))^m$.  Then we let $\billiards_{\mathrm{sym}}(m,r):=\Phi(\mathbf{W}_{\mathrm{sym}}(m,r))$ and equip it %equip
%$\billiards_{\mathrm{sym}}(m)$
with the topology coinduced by the map
$\Phi$. In particular, for any $\mathcal{D}\in \billiards_{\mathrm{sym}}(m)$, we have $\mathcal{D}\in \billiards_{\mathrm{sym}}(m,r)$ for some $r>0$. %defined by images under $\Phi$ of open sets in $\mathbf{W}_{\mathrm{sym}}$.
\end{defi}
We are now ready to state the main result of this paper.
\begin{main thm}\label{main thm}
  For any $m \ge 3$ and for any $r>0$, there exists an open and dense
  set of billiard tables
  $\billiards_{\mathrm{sym}}^{*}(m,r)
  \subset\billiards_{\mathrm{sym}}(m,r)$ so that if
  $\mathcal{D},\mathcal{D}'\in \billiards^{*}_{\mathrm{sym}}(m,r)$
  verify $\mathcal{MLS}(\mathcal{D}) = \mathcal{MLS}(\mathcal{D}')$,
  then $\mathcal{D}$ is isometric to $\mathcal{D}'$.
\end{main thm}

\begin{remark}
  In fact, the open and dense condition we need is an explicit
  \emph{non-degeneracy condition}.  Namely: we require that after a
  change of coordinates, the first coefficient in the expansion of the
  dynamics near a certain two-periodic orbit is non-vanishing (see
  Remark~\ref{remarque non deg} and condition $\mathrm{(\star)}$ in
  Lemma~\ref{non vanish}).

  In particular, given a table
  $\mathcal{D}_0=\Phi ((f^{(i)})_{i=1,\cdots,m})\in
  \billiards_{\mathrm{sym}}(m,r)$ for some $r>0$, with
  $(f^{(i)})_{i=1,\cdots,m} \in \mathbf{Conv}_r^m$, and a
  ``non-degenerate'' family $(\mathcal{D}_x)_{x}$ of perturbations of
  $\mathcal{D}_0$ within $\billiards_{\mathrm{sym}}(m,r')$ for some
  $r'\in (0,r]$, we can give an explicit description of the set of
  tables $\mathcal{D}_x$ which are not in
  $\billiards_{\mathrm{sym}}^{*}(m,r')$. More precisely, let
  $\mathbb{D}\subset \C$ be the unit disk, let
  $\Delta := \mathbb{D}^\N$, and let
  $\varepsilon\colon \N \to \R_+\setminus \{0\}$ be an arbitrary
  function that decays exponentially fast, with
  $\sup_{n \in \N} \varepsilon (n)< \varepsilon_0$ for some small
  $\varepsilon_0>0$. For some $r'\in (0,r]$, we thus get a family
  $(\mathcal{D}_x^\varepsilon)_{x \in \Delta}$ of billiards in
  $\billiards_{\mathrm{sym}}(m,r')$, where for
  $x =(x_m)_{m\in \N} \in \Delta$, the billiard
  $\mathcal{D}_x^\varepsilon$ is given by
	\begin{align*}
      \mathcal{D}_x^\varepsilon&:=\Phi(f^{(1)}+g_x^\varepsilon,f^{(2)}+g_x^\varepsilon,f^{(3)},\cdots,f^{(m)}),\\
      g_x^\varepsilon&:= \left[\theta \mapsto \sum_{m \geq 1} \varepsilon(m)  \Re \left(x_m e^{2 \pi \mathrm{i} m \theta}\right) \right]\in C_{r'}^\omega(\T, \R^2).
	\end{align*}
	Let us consider the map
    $\mathfrak{a}_1^\varepsilon\colon \Delta \ni x \mapsto
    a_1(\mathcal{D}_x^\varepsilon)\in \R$ in Lemma~\ref{non
      vanish}. The map $\mathfrak{a}_1^\varepsilon$ is a submersion
    (see the proof of Lemma~\ref{non vanish} for more details), hence
    the set
    $\Delta_{\mathrm{bad}}^\varepsilon:=\{x\in \Delta:
    \mathcal{D}_x^\varepsilon\notin\billiards_{\mathrm{sym}}^{*}(m,r')\}=(\mathfrak{a}_1^\varepsilon)^{-1}(\{0\})$
    of parameters for which the associated table is not in
    $\billiards_{\mathrm{sym}}^{*}(m,r')$ is a codimension one
    submanifold of $\Delta$.
\end{remark}
\begin{remark}
  We have stated our Main Theorem in the case of $m\ge 3$ scatterers,
  but indeed it suffices to prove the statement for $m = 3$.  In fact,
  fix $m > 3$, $r>0$, and let
  \begin{align*}
    \mathcal{D} = \R^{2}\setminus \bigcup^{m}_{i = 1}\obs_{i}\in \billiards_{\mathrm{sym}}(m,r);
  \end{align*}
  for $2 < i\le m$, define
  \begin{align*}
    \mathcal{D}_i := \R^{2}\setminus (\obs_{1}\cup \obs_{2} \cup \obs_{i}).
  \end{align*}
  It is immediate to show that
  $\mathcal{D}_{i}\in \billiards_{\mathrm{sym}}(3,r)$ (since the
  non-eclipse condition holds automatically).  Assume that the Main
  Theorem has been proved for $m = 3$, so that
  $\billiards^{*}_{\mathrm{sym}}(3,r)$ is defined, and let
  \begin{align*}
    \billiards^{*}_{\mathrm{sym}}(m,r) :=
    \{\mathcal{D}\in\billiards_{\mathrm{sym}}(m,r)\textrm{ s.t. } \forall\, 2 < i\le m,\,\mathcal{D}_{i}\in\billiards^{*}_{\mathrm{sym}}(3,r)\}.
  \end{align*}
  It is easy to check that
  ${\billiards}^{*}_{\mathrm{sym}}(m,r)$ is open and dense.
  Since $\mls(\mathcal{D}_{i})$ is the restriction of
  $\mls(\mathcal{D})$ to the periodic orbits that only collide with
  $\obs_{1},\ \obs_{2}$ and $\obs_{i}$, we can apply our Main Theorem
  for $m = 3$ to $\mathcal{D}_{i}$ and recover the geometry of
  $\obs_{1}$, $\obs_{2}$ and $\obs_{i}$ for any $i$.  Since $i$ was
  arbitrary, we proved the Main Theorem for $m$.
\end{remark}

It is a standard observation that any continuous deformation of smooth
domains which preserves the (unmarked) Length Spectrum
$\mathcal{LS}(\mathcal{D})$ automatically preserves the Marked Length
Spectrum (see \eg~\cite[Proposition 3.2.2]{Sib2}); therefore our
result also implies a spectral rigidity result. Let us
first introduce a definition.
\begin{defi}
  Given an integer $m \ge
  3$ and $r>0$, a family $(\mathcal{D}_t)_{t \in (-1,1)}$ of billiards is called an
  iso-length-spectral deformation in $\billiards^{*}_{\mathrm{sym}}(m,r)$
  if
	\begin{itemize}
    \item each $\mathcal{D}_t$ is in $\billiards^{*}_{\mathrm{sym}}(m,r)$,
      and the map
      $(-1,1)\ni t \mapsto \mathcal{D}_t$ is continuous;
    \item $\mathcal{LS}(\mathcal{D}_t)=\mathcal{LS}(\mathcal{D}_0)$,
      for all $t \in (-1,1)$.
	\end{itemize}
\end{defi}
Our Main Theorem thus yields the following result.
\begin{thmrig}
  For $m \ge
  3$ and $r>0$, any iso-length-spectral deformation %$(\mathcal{D}_t)_{t \in (-1,1)}$
  in $\billiards^{*}_{\mathrm{sym}}(m,r)$ is isometric.
\end{thmrig}
The proof of the Main Theorem (for $m = 3$) is given in
Corollary~\ref{cororororo} and Corollary~\ref{cororororobibis} in
Section~\ref{sec recov geom}, based on the constructions provided in
detail in the next sections.  From now on, we will only consider the
case of three scatterers.  We will abbreviate
$\billiards:=\billiards(3)$ and
$\billiards_{\mathrm{sym}}:=\billiards_{\mathrm{sym}}(3)$ and
similarly for
$\billiards_{\mathrm{sym}}(r):=\billiards_{\mathrm{sym}}(3,r)$ and
$\billiards^{*}_{\mathrm{sym}}(r):=\billiards^{*}_{\mathrm{sym}}(3,r)$.\\

Let
$\mathcal{D}= \R^{2}\setminus \bigcup^{3}_{i = 1}\obs_{i}\in
\billiards$ be a billiard table, and let
$\mathcal{F}:=\mathcal{F}(\mathcal{D})$.  A key object in our study is
the so-called \textit{Birkhoff Normal Form} for saddle fixed points of
symplectic local surface diffeomorphisms, whose definition we now
recall. We introduce it for period two orbits since this is the case
we will consider in the following, but the same can be done for any
periodic orbit (given a periodic orbit of period $p \geq 2$, each
point in the orbit is a saddle fixed point of $\mathcal{F}^p$).  Let
$j\neq k\in \{1,2,3\}$, and let $(s(j,k),0)$ be the
$(s,r)$-coordinates of the point of $\obs_j$ in the orbit
$(jk)$. Recall that by~\cite{Mos,Ste}, there exists an analytic
symplectomorphism $R\colon\mathcal{U}\to \mathcal{V}$ from a
neighborhood $\mathcal{U}\subset \mathcal{M}$ of $(s(j,k),0)$ to a
neighborhood $\mathcal{V}\subset \R^2$ of $(0,0)$ and a unique
analytic map $\Delta=\Delta(\mathcal{D},j,k)\in C^\omega(\R,\R^*)$,
with $\Delta(z)= \lambda +\sum_{\ell\geq 1} a_\ell z^\ell$,
s.t.  %restricted to $\mathcal{U}(j,k)$, we have
\begin{align*}%
	R \circ \mathcal{F}^2|_{\mathcal{U}} = N \circ R|_{\mathcal{U}},
\end{align*}%
where $N$ is   the \textit{Birkhoff Normal Form} of $\mathcal{F}^2|_{\mathcal{U}}$:
\begin{align*}%
	N=N(\mathcal{D},j,k)\colon (\xi,\eta)\mapsto(\Delta(\xi\eta)\xi,\Delta(\xi\eta)^{-1}\eta).
\end{align*}%
In the following, we refer to $(a_\ell)_\ell$ as the \textit{Birkhoff invariants} or \textit{coefficients} of $N$. \\

\begin{remark}\label{rem:symmetries}
  The two symmetries described above    are needed because of two
  different issues.  Let us consider  a billiard table $\mathcal{D}=\R^{2}\setminus \bigcup_{i =1}^3\obs_{i}\in \billiards_{\mathrm{sym}}$.
  \begin{itemize}
  \item The axial symmetry between $\obs_1$ and
    $\obs_2$ we ask for is similar to the one that appears for
    instance in the work of Zelditch. It is explained by the fact that
    in order to speak about Birkhoff Normal Forms, we need a fixed
    point. %for the dynamics.
    As the billiard map has no such fixed
    point, we need at least to consider its square.  In the process,
    some information is lost, unless $\obs_1$ and $\obs_2$ are the
    mirror image of each other; otherwise, we are \textit{a priori} only able to
    recover some averaged information between $\obs_1$ and
    $\obs_2$.
  \item The second symmetry we require is due to a well known
    observation made in~\cite{CdV3}. Indeed, the reason why we ask
    each of the obstacles $\obs_1,\obs_2$ to be themselves symmetric
    follows from the fact that the Birkhoff Normal Form has some
    intrinsic symmetries (it has two axes of symmetry), and only
    conveys partial information on the billiard dynamics, which has
    \textit{a priori} only one natural symmetry, given by the time
    reversal property. Roughly speaking, we lose half of the
    information on the billiard map, unless the map itself has some
    additional symmetry: this symmetry of the map can be ensured
    provided that $\obs_1,\obs_2$ have a $\Z_2$-symmetry.
  \end{itemize}
\end{remark}

\begin{remark}\label{remarque non deg}
  In the following, we show that the Lyapunov exponents of the orbits
  $(h_n)_n$ can be expanded as a series indexed by $\Z^2$
  (see~\eqref{expansion de Lyapunoff}), each of whose coefficients is
  a \mlsinv. The expression of these coefficients combines three
  different sets of geometric data, including the Birkhoff invariants
  we want to reconstruct. We show that under some open and dense
  condition, it is possible to extract enough information from the
  first three lines of the coefficients of the series in order to
  recover separately the three sets of data. More precisely, the
  condition we need is the non-vanishing of the first Birkhoff
  invariant (see condition $\mathrm{(\star)}$ in Lemma~\ref{non vanish}), which can
  be seen dynamically as some twist condition. It guarantees that
  certain linear systems in the data we want to recover are
  invertible. Note that Lemma~\ref{corollar dofe} gives an effective
  way of checking whether a given billiard table
  $\mathcal{D} \in \billiards_{\mathrm{sym}}$ satisfies this twist
  condition; in other words, the property
  $``\mathcal{D} \in \billiards_{\mathrm{sym}}^*"$ itself is a
  \mlsinv. Besides, this condition comes from the particular subset of
  coefficients we consider (which is easiest to work with), and it is
  likely that considering other subsets of coefficients would produce
  another non-degeneracy condition involving different Birkhoff
  invariants. In particular, it seems reasonable to believe that as
  long as the Birkhoff Normal Form is not degenerate (i.e., is not
  linear), our construction can be adapted to produce invertible
  systems in the coefficients we want to reconstruct.
\end{remark}

\subsection{Idea of the proof}
\label{sec:idea-proof}
Let us give an idea of the proof of the
above results. We fix a billiard table
$\mathcal{D}= \R^{2}\setminus \bigcup^{3}_{i = 1}\obs_{i}\in
\billiards_{\mathrm{sym}}$, and let
$\mathcal{F}:=\mathcal{F}(\mathcal{D})$.
%Without loss of generality, we assume that $(j,k)=(1,2)$ in the following.
Note that it is natural to focus on $2$-periodic orbits, since we may
hope to determine $\mathcal{F}^2$ instead of $\mathcal{F}^p$ for some
higher exponent $p \geq 3$, and also because of the additional
symmetries of such orbits.  As we shall see, the Birkhoff coefficients
$(a_\ell)_{\ell \geq 1}$ above can be obtained by the asymptotics of
Lyapunov exponent for certain periodic orbits which spend a lot of
time near the periodic orbit $(12)$.  We then define a sequence of
periodic orbits $(h_n)_n$ with a certain palindromic symmetry that
accumulate\footnote{ An analogous construction can be found
  in~\cite{BDSKL}.} an orbit $h_\infty$ which is homoclinic to $(12)$.

A key step in the construction is the extension of the coordinates
given by the conjugacy $R$ between $\mathcal{F}^2$ and its Birkhoff
Normal Form, which is initially defined only in a neighborhood of the
saddle fixed point. Indeed, in order to make the connection with the
Lyapunov exponent of the orbits $(h_n)_n$, it is crucial to extend the
conjugacy to describe them globally. The construction of the extension
follows a classical procedure, by using the dynamics to propagate $R$
along the separatrices, i.e., the stable and unstable manifolds of the
origin. This is actually sufficient to describe all points in the
orbits $(h_n)_n$, since for $n$ large enough, these orbits stay in a
small neighborhood of the separatrices.  In this way, we produce
convenient coordinates to describe the dynamics in a neighborhood of
the separatrices, which can be seen as a hyperbolic analogue of the
coordinates provided by the Birkhoff Normal Form near the boundary of
the billiard table, in the elliptic case, and which were used for
instance in~\cite{CdV3}. The problem is that we are extending our
coordinates along two different directions, and at some point, since
the trajectory is periodic, these two extensions will overlap in the
collision space.  In particular, we will need to perform a ``gluing''
of the two charts obtained in this way in a neighborhood of the
homoclinic point on the third scatterer, and we will explain how to
take care of this issue in the sequel.

By the palindromic property, we can write an equation for the images
under the conjugacy map $R$ of the points in the periodic orbits
$(h_n)_n$ (see Lemma~\ref{claim eta n delta n xi n}).  This allows us
to find an implicit expression of the parameters of those points in
terms of the Birkhoff invariants and the coefficients of the arc of
points where those orbits start (as we shall see, this arc is made of
points on the second obstacle which bounce perpendicularly on the
third obstacle after one iteration of the dynamics).  As mentioned
above (see Theorem~\ref{prp:palindromic_lyapunov}), the Marked
Lyapunov Spectrum (for palindromic orbits) is a \mlsinv{}.  The key
observation is stated in Lemma~\ref{lemma exp lya detailll}: we show
that for each integer $n \geq 0$, the Lyapunov exponent of $h_n$
admits a series expansion; more precisely, it holds
\begin{equation}\label{expansion de Lyapunoff}
  2 \lambda^n\cosh\left(2(n+1)\mathrm{LE}(h_n)\right) = \sum_{p=0}^{+\infty}\sum_{q=0}^{p}L_{q,p} n^{q}\lambda^{np},
\end{equation}
for some sequence
$(L_{q,p})_{\substack{p=0,\cdots,+\infty\\ q=0,\cdots,p}}$, and where
$\lambda\in(0,1)$ is the contracting eigenvalue of $D\mathcal{F}^2$ at
the two-periodic orbit $(12)$.  In particular, each coefficient
$L_{q,p}$ is a \mlsinv{}, and by restricting ourselves to $q=0,1,2$,
this gives enough information to recover the Birkhoff
coefficients. Note that the expansion~\eqref{expansion de Lyapunoff}
obtained for the Lyapunov exponents of palindromic orbits in the
horseshoe associated to the homoclinic orbit $h_\infty$ can be seen as
a hyperbolic analogue of the Marvizi–Melrose expansion (see for
instance~\cite[(1.11)]{MM} on p.3) for the maximum lengths of periodic
orbits with a certain rotation number and period (the integer $n$
being related to the period in either case).  Let us also mention that
similar expansions were studied in \cite{FY} for a different purpose;
moreover, the relation between coefficients in Birkhoff normal forms
and spectral properties of the dynamics has been explored as well in
different settings in several works, see \eg~\cite{Be,Sib1}.

One technical issue comes from the
fact that the orbits $(h_n)_n$ bounce on the third obstacle, thus
there are additional terms which come from a certain gluing map
$\mathcal{G}$ taking this last bounce into account, and we need to
find a way to recover this data as well.  The idea is to leverage on the
“triangular” structure of the coefficients: at each step, there are
certain additive constants associated with some terms that we already
know, as well as new coefficients that we want to recover.  Then, we
derive a linear system in the new coefficients, and show that it is
invertible under a suitable twist condition (non-vanishing of the
first Birkhoff invariant).  By induction, modulo some ``homoclinic
parameter'' $\xi_\infty\in \R$, we can thus recover the Birkhoff invariants, as well as
some information on the third obstacle associated to the differential
of the gluing map $\mathcal{G}$.

More precisely, we consider some arc $\Gamma_\infty\ni(0,\xi_\infty)$
which is the image in Birkhoff coordinates of some small arc of points
associated with a perpendicular bounce on the third scatterer (see
also Figure~\ref{picture dynamique}). This arc is the graph of some
analytic function $\gamma$ that is determined by the gluing map
$\mathcal{G}=(\mathcal{G}^+,\mathcal{G}^-)$, i.e., for $\xi$ small,
$\eta=\xi_\infty+\gamma(\xi)$ satisfies the implicit equation
\begin{align*}
	\mathcal{G}^+(\xi,\eta)\mathcal{G}^-(\xi,\eta)=\xi\eta.
\end{align*}
We show that from the sequence
$(L_{q,p})_{\substack{p=0,\cdots,+\infty\\ q=0,1,2}}$, up to the parameter $\xi_\infty$, it is possible
to recover the value of the Birkhoff invariants, as well as the
function $\gamma$ and the differential
$\mathcal{DG}|_{T\Gamma_\infty}$.

Note that homotheties preserve Lyapunov exponents, hence by only
considering the Marked Lyapunov Spectrum, one cannot expect to recover
the scale of the billiard table: the missing parameter $\xi_\infty$
can indeed be interpreted as a \emph{scaling factor}.  In
Subsection~\ref{determ homco}, we prove that its value is a
\mlsinv: %(but \textit{a priori} not a Marked Lyapunov
%
%Spectrum-invariant).
%Indeed, iwe
we show (see Proposition~\ref{prop lourde un}) that for some
$\mathcal{L}^\infty\in \R$, the quantity
\begin{equation}\label{equiv longueur}
\mathcal{L}(h_n)-(n+1) \mathcal{L}(12)-\mathcal{L}^\infty
\end{equation}
decays exponentially fast as $n$ goes to infinity (see
Section~\ref{birkkkhhfo} for the notation), and can be expanded as a
series of the same form as the one obtained in~\eqref{expansion de
  Lyapunoff} -- and also similar to the expansion obtained
in~\cite{MM}. The first order term in this expansion is a \mlsinv\ and
can be written in terms of $\xi_\infty^2$ and of a certain quadratic
form. It can be then shown (see the proof of Corollary~\ref{premier cororrr})
%Remark~\ref{prp:quadratic-form-mls-invariance})
that the latter is a
\mlsinv, thus $\xi_\infty$ is a \mlsinv\ too.

It is worth emphasizing that all results proved up to Subsection~\ref{determ homco} -- in particular, the $\mathcal{MLS}$-determination of
the Birkhoff Normal Form and of the gluing map $\mathcal{G}$, up to the homoclinic parameter $\xi_\infty$ -- do not
require any symmetry assumption. Indeed, in our approach, axial
symmetries are needed only to reconstruct the geometry from the
Birkhoff Normal Form and the map $\mathcal{G}$.  Furthermore, although
we focus on Birkhoff Normal Forms associated to $2$-periodic orbits in
this paper, the procedure we describe is more general. More precisely,
under a generic condition (non-vanishing of the first Birkhoff
invariant of each periodic orbit), our construction can be adapted to
show that the Birkhoff Normal Form of any (palindromic) periodic orbit
is a \mlsinv{}. % (see Remark~\ref{rem palindod}).

Although we restrict ourselves to a specific class of dispersing
billiards in the present work, similar computations can be carried out
for other types of billiards. The general framework where this can be
done is when the billiard under consideration possesses a hyperbolic
periodic orbit with a transverse homoclinic intersection. If we stay
away from singularities (tangential collisions), it is well-known that
this intersection generates a horseshoe; we may then construct a
sequence $(h_n)_{n}$ of periodic orbits within this hyperbolic set
with a specific combinatorics, and study the asymptotics of their
Lyapunov exponents as $n \to +\infty$. The same computations as those
explained in the present paper can be performed to show that the
Birkhoff invariants of the periodic orbit can be reconstructed from
the Lyapunov exponents of the orbits $(h_n)_n$. Yet, what is
especially convenient for the class of open dispersing billiards
considered here is that the symbolic coding which is used to select
the orbits $(h_n)_n$ from the horseshoe has a precise
\emph{geometric} meaning (namely, the sequence of obstacles
corresponding to the consecutive bounces in the orbit $h_n$), which
makes it easier to relate the symbolic coding to the marking used in
our definition of the Marked Length Spectrum.\footnote{ For instance,
  convex billiards may exhibit a hyperbolic set associated to a
  hyperbolic periodic orbit; in this case, similar computations can be
  performed to show that the Birkhoff invariants can be determined
  from the variation of the Lyapunov exponents of certain periodic
  orbits $(h_n)_n$ in the horseshoe, but it is less natural to
  ``mark'' the orbits $(h_n)_n$ by some geometric information.}

Let us now consider the case of symmetric billiard
tables $\mathcal{D}\subset \billiards_{\mathrm{sym}}$. In this case,
we can introduce some flat wall between $\obs_1$ and $\obs_2$, and
``fold" the table in order to virtually create a fixed point of the
billiard dynamics. For some auxiliary billiard table $\mathcal{D}^*$
one of whose obstacles is now flat, we can extract enough information
from $\mathcal{MLS}(\mathcal{D})$ to reconstruct the Birkhoff Normal
Form $N^*$ of the square $T^*:=(\mathcal{F}^*)^2$ of the new billiard
map $\mathcal{F}^*=\mathcal{F}^*(\mathcal{D}^*)$ near the new
$2$-periodic orbit.  Besides, by the construction of $N^*$ (see
\eg~\cite{Bi,Ste,Mos}), and \textit{due to the symmetry of
  $\obs_1,\obs_2$, the jets of $N^*$ and $T^*$ %at $(s(1,2),0)$
  are in one-to-one correspondence, by some invertible triangular
  system}. %, which allows us to reconstruct $\mathcal{F}^2$. This part of the proof is similar to the
Furthermore, as Colin de Verdi\`ere~\cite{CdV3} already observed in
the elliptic setting,
%we can actually determine the map $\mathcal{F}(\mathcal{D})$ and not only its square.
%argument that can be found in  in the elliptic setting, where the author shows that
there is a bijective correspondence between the jet of $T^*$ and the jet of the graphs of $\obs_1,\obs_2$, which can thus be reconstructed.

In order to recover the geometry of the last obstacle, we analyze the information that comes from the gluing map that we were mentioning   previously. We can extract from this map some ``averaged" information between the first two obstacles $\obs_1,\obs_2$ and the third obstacle $\obs_3$, and since the geometry of $\obs_1,\obs_2$ is known, we can also reconstruct the local geometry of $\obs_3$ near a certain homoclinic point. This determines the obstacle $\obs_3$ entirely, by analyticity.
	\medskip

\subsection{Organization of the paper}
We use  the notations introduced in Subsection~\ref{sec:idea-proof}.
The proof of the Main Theorem follows different steps:

\smallskip
\begin{description}
	\item[$*$ Step 1] existence of a \textit{canonical} (which respects the time reversal symmetry) change of coordinates under some non-degeneracy condition;
	\smallskip
	\item[$*$ Step 2] extension of the system of coordinates and expression of the palindromic orbits $(h_n)_{n}$ in those coordinates;
	\smallskip
	\item[$*$ Step 3] definition of the gluing map $\mathcal{G}$;
	\smallskip
	\item[$*$ Step 4] asymptotic expansion (in $n$) of the Lyapunov exponent of $h_n$;
	\smallskip
	\item[$*$ Step 5] extracting a triangular system from the Lyapunov expansion in terms of ``scaled" Birkhoff coefficients and gluing terms;
	\smallskip
	\item[$*$ Step 6] invertibility of this system under some non-degeneracy condition (proof by induction: compute the $i^{\text{th}}$ order terms using the previous
	ones);
	\smallskip
	\item[$*$ Step 7] $\mathcal{MLS}$-determination of the missing ``scale" parameter $\xi_\infty$;
	\smallskip
	\item[$*$ Step 8] Step 5 + Step 6 $\Rightarrow$ the Birkhoff Normal Form and the differential of the gluing map $\mathcal{G}$ are \mlsinv{s};
	\smallskip
	\item[$*$ Step 9] determination of the geometry from the Birkhoff invariants + the gluing map $\mathcal{G}$ in the case of symmetric billiard tables which satisfy a non-degeneracy condition ($\star$), Lemma \ref{non vanish}.
	%\item[$*$ Step 10] determination of $\obs_3$ from the gluing + the geometry of $\obs_1,\obs_2$.
\end{description}
\medskip

Those steps are detailed respectively in: \\
\textbf{(1)} Section~\ref{birkkkhhfo}; \textbf{(2)}  Section~\ref{section extension}; \textbf{(3)} Subsection~\ref{subs prliemr}; \textbf{(4)}-\textbf{(5)} Subsection~\ref{subsection lyapunov exonent exp}; \textbf{(6)} Subsection~\ref{detertmin coeffci}; \textbf{(7)} Subsections~\ref{basic facts}-\ref{sub esti leng}; \textbf{(8)} Subsection~\ref{determ homco};
\textbf{(9)} Section~\ref{sec recov geom}.
\medskip

The central technical part of the proof is Steps 4-5; an outline of the computations carried out there is given after Remark~\ref{remarque prelimina} (see also Remark~\ref{remark partie technique}). Let us also emphasize  formula~\eqref{delta n eta n xi n} which follows from the palindromic symmetry of the orbits $(h_n)_n$ and on which the induction is based.
\medskip

Moreover, the scheme of the proof can be summarized as follows:

\begin{equation*}
	\xymatrix{
		&  \mathcal{MLS}\ar@{->}[ld]|{\mathrm{Marked\ Lyapunov\ Spectrum}} \ar@{->}[rd]& \\
		(L_{q,p})_{p,q}%\ar@/^/[dd]
		\ar@{-}[rd]%|{\text{invertible linear system}}%\ar@/_/[rr]%^-{\Big(\mathrm{Marked\ Lyapunov\ Spectrum}\Big)}
		& & %\ar@{->}[r]  & \ar@{->}[r]&
		%\left\{
		%\begin{array}{r}
		%& &
		%& &
		%\end{array}
		%\right.
		%\ar@{->}[d]  \\
		%\xi_\infty \ar@{->}[rr] & &
		\xi_\infty\ar@{-}[ld]%\ar@/_/[dd]%|{\text{invertible linear system}}%\ar@/_/[lldd]%|{invertible linear system}
		%\ar@{->}[d]
		\\
		& \text{invertible linear system}	\ar@{->}[rd]	\ar@{->}[ld] & \\
		N^* \ar@{<->}[dd]|{\Z_2\times \Z_2-\mathrm{symmetry}}   & & \Gamma_\infty + \mathcal{DG}|_{T\Gamma_\infty}\ar@{->}[dd]\\
		& \text{ geometry of } \{\obs_1,\obs_2\}\ar@{->}[dr]|{\text{canonical conjugacy $R$}}  & \\
		\text{ jet of } \mathcal{K}\text{ at } (s(1,2),0)\ar@{<->}[ur]|{\text{analyticity}}  &   & \text{ geometry of } \obs_3
	}
\end{equation*}

\medskip

\subsection{Previous results}\label{resul prec}

In this paper we will use and further develop many of the ideas that
appeared in the paper~\cite{BDSKL}, joint with P.~B\'alint.  Vaughn
Osterman found out that~\cite{BDSKL} contains a gap as-published; the
gap affects the main result of the paper, but a number of useful
technical results can be still recovered from the paper.  We list the
most important such results in this section; the reader will find in
Appendix~\ref{sec:bdkl-symmetric-case} self-contained proofs of them.
\begin{theorem}\label{corollaire lyapu}
  The Marked Lyapunov Spectrum (for palindromic orbits) is determined
  by the Marked Length Spectrum (see~\eqref{marked spectrum}
  and~\eqref{maked lepnov} for the definitions).
\end{theorem}
In particular the above theorem allows to conclude the following result.
\begin{corollary}\label{main corr}
  Assuming that $\obs_2$ is the mirror image of $\obs_1$, then the (common) radius of curvature $R$ at the bouncing points of the $(12)$
  periodic orbit is a \mlsinv.
\end{corollary}

\section{The Birkhoff Normal Form in a neighborhood of period two
  orbits}\label{birkkkhhfo}

Let us fix a billiard table
$\mathcal{D}= \R^{2}\setminus \bigcup^{3}_{i = 1}\obs_{i}\in
\billiards$ and study the local dynamics near $2$-periodic orbits.  We
focus on the $2$-periodic orbit $\sigma=(12)$ (recall that our
symmetry assumption holds for those two scatterera); it has two
perpendicular bounces on the first and the second obstacles. Let us
denote by $x(0)=(s(0),0)$ and $x(1)=(s(1),0)$ the $(s,r)$ coordinates
(recall the definition at the beginning of Section~\ref{sec prelim})
of the points in this orbit, where $s(0)$, (resp. $s(1)$) is the
position of the point on the first (resp. second) obstacle.  We extend
this notation by periodicity by setting $x(k):=x(k\ \text{mod } 2)$,
for $k \in \Z$.  We let $\tau:=(32)$ and, given any integer
$n \geq 1$, set
\begin{align*}
 h_n =h_n(\sigma,\tau)&:=(\tau
  \sigma^{n})=(32\underbrace{1212\dots12}_{2n}).
\end{align*}
The word $h_n$ encodes a periodic orbit of period $2n+2$; observe that
all such orbit are palindromic.

Let
$\mathcal{F}=\mathcal{F}(\mathcal{D})\colon x_0=(s_0,r_0)\mapsto
x_1=(s_1,r_1)$ be the billiard map.  In such coordinates,
$\mathcal{F}$ is exact symplectic, with generating function
\begin{align}\label{eq:generating-function}
  h(s,s'):=\|\Upsilon(s)-\Upsilon(s')\|,
\end{align}
where $\|\Upsilon(s)-\Upsilon(s')\|$ is the Euclidean length of the
line segment between the two points identified by parameters $s$ and
$s'$. In other words, we have
\begin{equation}\label{hamildlit}
dh(s_0,s_1)=-r_0 ds_0+r_1 ds_1,
\end{equation}
i.e., $\partial_1 h(s_0,s_1)=-r_0$, and $\partial_2 h(s_0,s_1)=r_1$.

\begin{figure}[H]
	\begin{center}
		\includegraphics [width=12cm]{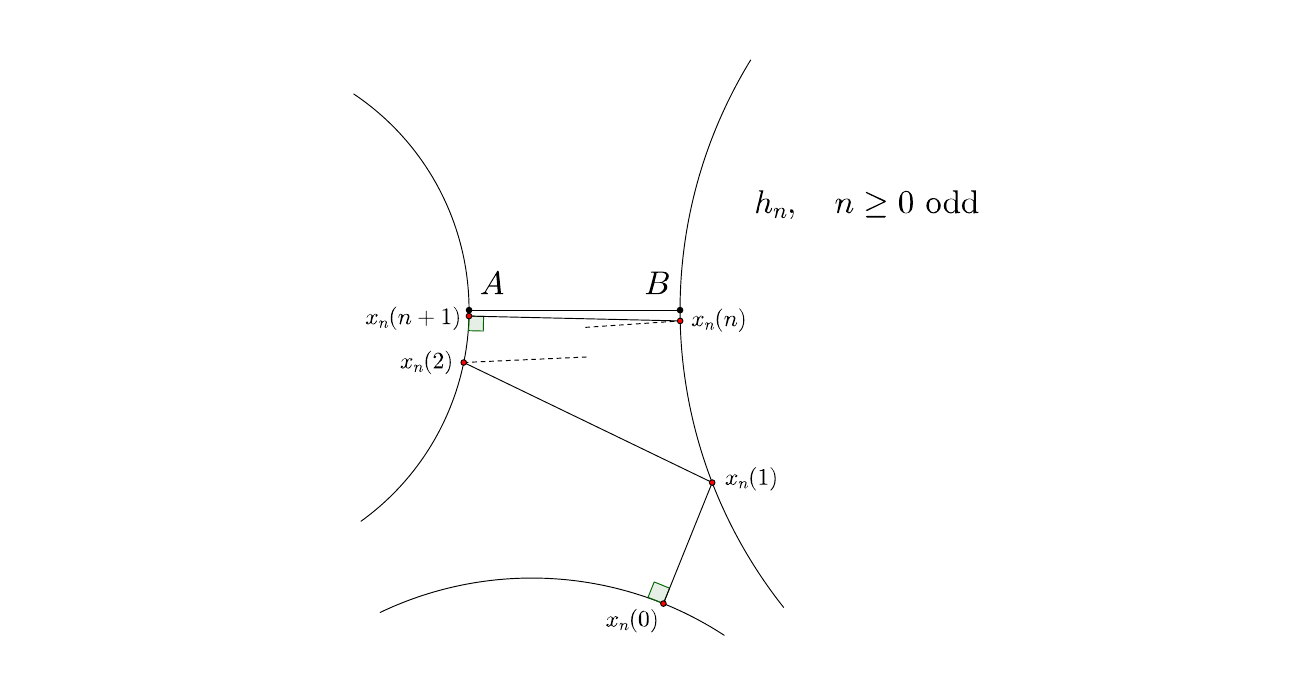}
		\caption{Trace of the orbits $h_n$ on the billiard table when
			$n \geq 0$ is odd.}
	\end{center}\label{fig:trace-orbits-h_n}
\end{figure}

Let us denote by
$\left(x_n(k) = (s_n(k),r_n(k))\right)_{k=0,\cdots,2n+1}$ the
coordinates of the points in the orbit $h_n$, where
$x_n(0)=(s_n(0),0)$ is the only collision on the third obstacle, and
$x_n(n+1)=(s_n(n+1),0)$ is the point of the orbit which is closest to
the periodic orbit $\sigma=(12)$ (see
Figure~\ref{fig:trace-orbits-h_n}). Again, thanks to the
$2n+2$-periodicity of $h_n$, we can extend those coordinates to any
$k \in
\Z$. %We also let $r_n(k):=-\cos(\varphi_n(k))$, for $k \in \Z$.
%As we have already observed in~\cite[Lemma 3.2]{BDSKL}, b
By the palindromic symmetry, for any $k\in \{0,\cdots,n+1\}$,
it holds
\begin{equation}\label{palindr symm}
  x_n(2n+2-k)=\mathcal{I} (x_n(k)),
\end{equation}
where $\mathcal{I}(s,r):=(s,-r)$.  Recall that for a periodic orbit
encoded by a finite word $\varsigma$, we denote its length by
$\mathcal{L}(\varsigma)$ . We have
\begin{align*}
\mathcal{L}(h_n)-(n+1)\mathcal{L}(\sigma)=2 \sum_{k=0}^{n}\Big(
  h(s_n(k),s_n(k+1))-h(s(k),s(k+1))\Big).
\end{align*}

Let $h_\infty=h_\infty(\sigma,\tau)$ be the homoclinic trajectory
encoded by the infinite word
$(\sigma^\infty \tau \sigma^\infty)=(\dots 21212321212\dots)$. We
denote by $(x_\infty(k))_{k \in \Z}$ its coordinates, with
$x_\infty(k)=(s_\infty(k),r_\infty(k))$, for $k \in
\Z$. %, and we let $r_\infty(k):=-\cos(\varphi_\infty(k))$.
We label them in such a way that $x_\infty(0)$ is associated with the
unique bounce on the third obstacle, and $r_\infty(k)r_n(k) \geq 0$
for all $k\in \Z$.  Since $h_\infty$ is homoclinic to $\sigma$, the
quantity $h(s_\infty(k),s_\infty(k+1))-h(s(k),s(k+1))$ decays
exponentially fast as $k\to+\infty$ (see also Proposition~\ref{prop
  bdskl} below), and the following limit remains well
defined:
\begin{align}\label{eq:definition-Linfty}
  \mathcal{L}^\infty = \mathcal{L}^\infty(\sigma) &= \lim_{n\to \infty}(\mathcal{L}(h_n)-(n+1)\mathcal{L}(\sigma))\\\notag
                     &=2 \sum_{k=0}^{+\infty}\Big(h(s_\infty(k),s_\infty(k+1))-h(s(k),s(k+1))\Big).
\end{align}
Let us now write
\begin{align*}
  \mathcal{L}(h_n)-(n+1)\mathcal{L}(\sigma)-\mathcal{L}^\infty= \Sigma_n^1+\Sigma_n^2,
\end{align*}
where
\begin{subequations}
  \begin{align}
    \Sigma_n^1&:= 2\sum_{k=0}^{n}\Big( h(s_n(k),s_n(k+1))-h(s_\infty(k),s_\infty(k+1))\Big)\label{ligne 1},\\
    \Sigma_n^2&:=2\sum_{k=n+1}^{+\infty}\Big(h(s(k),s(k+1))-h(s_\infty(k),s_\infty(k+1))\Big).\label{ligne 2}
  \end{align}
\end{subequations}
\begin{figure}[H]
	\begin{center}
		\includegraphics [trim = 0cm 2cm 8cm 1cm, width=14cm]{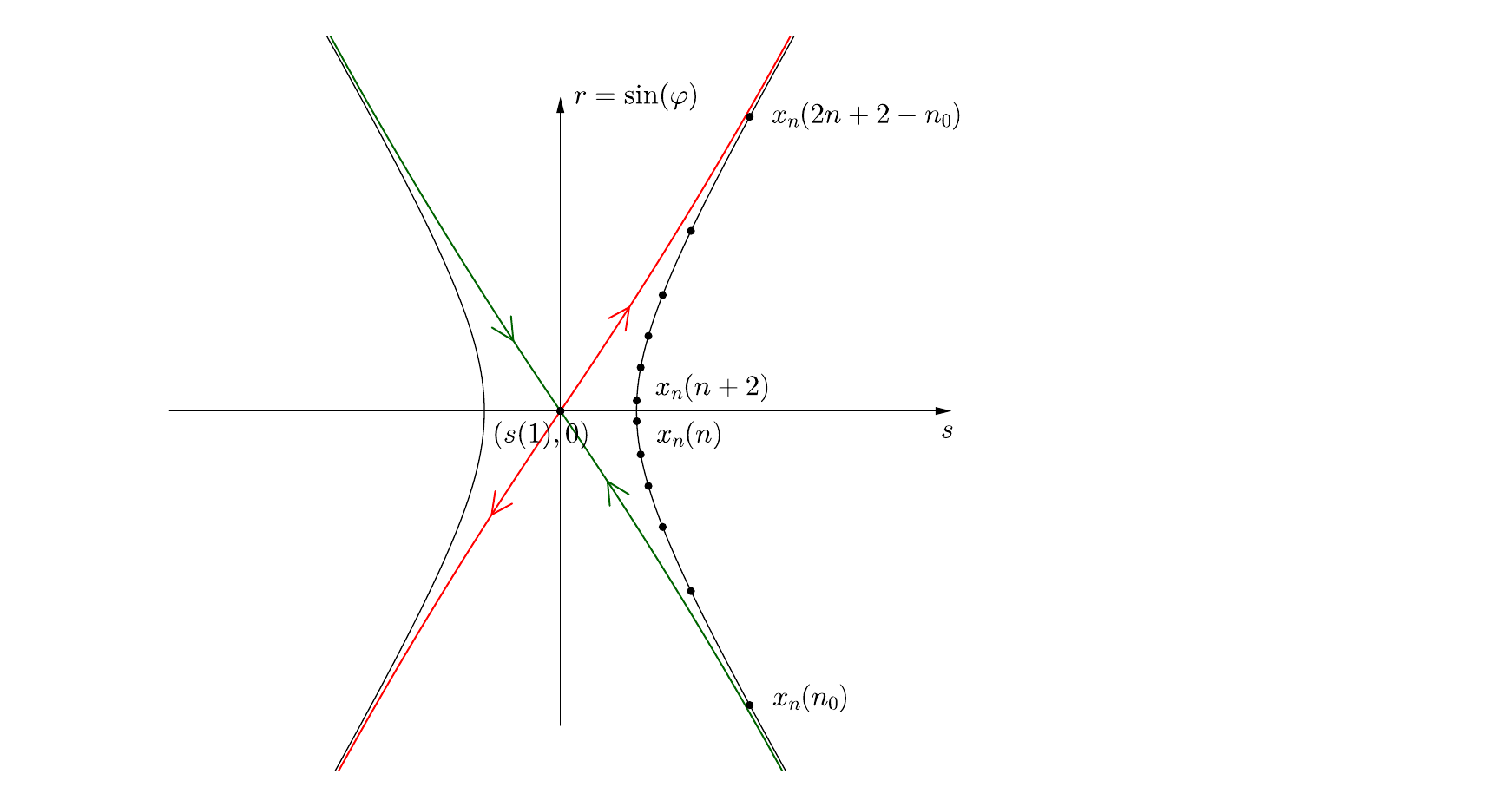}
		\caption{$(s,r)$-representation of the points in $h_n$ near the   orbit $(12)$.}
	\end{center}
\end{figure}
The point $x(1)=(s(1),0)$ is a saddle fixed point of $\mathcal{F}^2$
with eigenvalues $\lambda<1<\lambda^{-1}$.  The following proposition
was shown in~\cite{BDSKL}, it will also be proved later as
Proposition~\ref{coro proximite points}.
\begin{prop}%[{\cite[Proposition 4.1]{BDSKL}}]
	\label{prop bdskl}
    There exists an integer $n_{0} \geq 1$ such that
\begin{equation*}%\label{estimedjjt}
\begin{array}{rcll}
\|x_\infty(k)-x(k)\|&=&O(\lambda^{\frac{k}{2}}), &\text{for all } k \in \N,\nonumber\\
\|x_n(k)-x_\infty(k)\|&=&O(\lambda^{n-\frac {k}{2}}), &\text{for all }
n \geq n_0 \text{ and } k\in \{0,\cdots,n+1\}.
\end{array}
\end{equation*}
\end{prop}
The first estimate comes from the fact that the points in the homoclinic orbit
are on the stable manifold of the point $x(1)$. Moreover, as
$n\to\infty$, everything happens in a neighborhood of the unstable and
stable manifolds of the periodic orbit $\sigma$. Indeed, for the first
half of the orbit $h_n$, i.e., for $k\in \{0,\cdots,n+1\}$, the second
estimate above tells us that the points $x_n(k)$ shadow closely the
associated points $x_\infty(k)$ in the homoclinic orbit $h_\infty$,
and thus, stay close to the stable manifold of $x(1)$. On the other
hand, for the second half of $h_n$, i.e., for
$k\in \{n+1,\cdots,2n+2\}$, then by the palindromic
symmetry~\eqref{palindr symm}, the points $x_n(k)$ shadow closely the
points $\mathcal{I}(x_\infty(k))$, and thus, stay close to the
unstable manifold of $x(1)$.

Let us consider the case where $n$ is odd, i.e., $n=2m-1$ for some
integer $m \geq 1$, and let us study the dynamics of
$T:=\mathcal{F}^2$. Here, the period of $h_n=h_{2m-1}$ is equal to
$2n+2=4m$.  For simplicity, we assume in the following that
$s(1)=0$.\\

By~\eqref{hamildlit}, the map $\mathcal{F}$ is symplectic for the form $ds \wedge dr$, where $r:=\sin(\varphi)$. It follows that $T$ is symplectic too, i.e., $T^*(ds \wedge dr)=ds \wedge dr$.
Then, by~\cite{Mos} (see also~\cite{Ste}), there exists a neighborhood $\mathcal{U}$ of $(0,0)$ in the $(s,r)$-plane, and an analytic symplectic change of coordinates
\begin{align*}
  R\colon \left\{
  \begin{array}{rcl}
    \mathcal{U} & \to & \R^2,\\
    (s,r) & \mapsto & (\xi,\eta)
  \end{array}
 \right.
\end{align*}
with $R^{*}(d\xi \wedge d \eta)=ds \wedge dr$, which conjugates $T$ to
its Birkhoff Normal Form $N=R \circ T\circ R^{-1}$:
\begin{align*}
  N=N_\Delta\colon (\xi,\eta)\mapsto(\Delta(\xi\eta)\cdot\xi,\Delta(\xi\eta)^{-1}\cdot\eta),
\end{align*}
where $\Delta$ is an analytic function
\begin{align*}
  \Delta\colon z \mapsto \lambda + \sum_{k=1}^{+\infty} a_k z^k.
\end{align*}
The
numbers $(a_k)_{k \geq 1 }$ are called the \textit{Birkhoff
  invariants} (or \textit{Birkhoff coefficients}) of $T$ at $(0,0)$.

In the following, given two open sets
$\mathcal{V},\mathcal{V}'\subset \R^2$, we will denote by
$\mathrm{Sympl}^\omega(\mathcal{V},\mathcal{V}')$ the set of real
analytic symplectomorphisms from $\mathcal{V}$ to $\mathcal{V}'$ which
preserve the form $d\xi \wedge d\eta$.

\subsection{Canonical choice of the conjugacy map $R$}

The map $R$ mentioned in the previous section is not uniquely defined;
in this section we proceed to identify a canonical choice of
coordinates which respects additional symmetries of the periodic
orbit.

The following lemma ensures that the map $N$ is well defined in a
neighborhood of the vertical and horizontal axes.
\begin{claim}\label{easy claim}
  For all $(\xi,\eta) \in R(\mathcal{U})$, and for each $k \in \Z$,
  the point $(\Delta(\xi\eta)^{-k}\xi,\Delta(\xi\eta)^k \eta)$ is in
  the domain of definition of $N$. In particular, $N$ can be extended
  to each point in the orbit of $(\xi,\eta)$.
\end{claim}

\begin{proof}
  Let $k \in \Z$; clearly, $N$ is well defined at a point
  $(\Delta(\xi\eta)^{k}\xi,\Delta(\xi\eta)^{-k} \eta)$ if and only if
  $\Delta$ is well defined at the point
  $(\Delta(\xi\eta)^{k}\xi)\cdot (\Delta(\xi\eta)^{-k} \eta)=\xi\eta$,
  which is true, provided that $(\xi,\eta) \in R(\mathcal{U})$.
\end{proof}

As a consequence of this observation, there exists some $\epsilon_0>0$
such that the Birkhoff Normal Form $N$ is well defined on the
neighborhood
$\mathcal{V}_{\epsilon_{0}}=\{(\xi,\eta)\in \R^2:
|\xi\eta|<\epsilon_0\}$ of the coordinate axes $\{\xi=0\}$ and
$\{\eta=0\}$. In particular, $N$ preserves
$\mathcal{V}_{\epsilon_{0}}$ and acts merely as a translation along
hyperbolas..

Next, we study the symmetries of Birkhoff Normal Forms.
\begin{defi}[Centralizer of $N$]\label{defi centr}
  Let $0 < \epsilon_1 \leq \epsilon_{0}$; we define the symplectic
  centralizer $\mathcal{C}_N=\mathcal{C}_N(\epsilon_1)$ of $N$ as the
  set of maps
  $F \in
  \mathrm{Sympl}^\omega(\mathcal{V}_{\epsilon_{1}},\mathcal{V}_{\epsilon_{0}})$,
  such that
  \begin{equation*}
	F \circ N|_{\mathcal{V}_{\epsilon_1}} =N \circ F|_{\mathcal{V}_{\epsilon_1}}.
  \end{equation*}
  Notice that, since $N$ preserves $\mathcal V_{\epsilon}$, for any
  $\epsilon \leq \epsilon_0$, we can iterate the above expression and
  obtain
  \begin{equation}\label{iteration de N}
	F \circ N^j|_{\mathcal{V}_{\epsilon_1}} =N^j \circ F|_{\mathcal{V}_{\epsilon_1}},\quad\text{for all }j \in \Z.
  \end{equation}
\end{defi}
\begin{remark}\label{rmk:extension-centralizer}
  Observe that an analytic symplectic map $F$ is defined in a
  neighborhood $\mathcal V$ of $(0,0)$ and commutes with $N$, then it
  can always be analytically and symplectically extended to
  $\tilde{\mathcal V}= \bigcup_{k\in\mathbb Z}N^{k}\mathcal V$ using
  $N$ as follows. For $x\in\tilde{\mathcal V}$, let $k$ be so that
  $N^{k}x\in\mathcal V$, then $F(x) = N^{-k}\circ F\circ N^{k}$.  It
  is then immediate to show that this is an analytic, symplectic
  extension and it commutes with $N$ on $\tilde{\mathcal V}$.
\end{remark}
\begin{lemma}\label{generique centralisateur}
  Let $N$ be so that the Birkhoff invariants $(a_k)_{k \geq 1 }$ are
  not all equal to zero.  Then, provided that $\epsilon_1>0$ is chosen
  sufficiently small (depending on the sequence $(a_{k})_{k\geq 1}$),
  the centralizer $\mathcal{C}_N=\mathcal{C}_N(\epsilon_1)$ of $N$ is
  given by the set of maps of the form
\begin{align}\label{eq centrlaizer}
  F
  \colon
  (\xi,\eta)\mapsto(\widetilde{\Delta}(\xi\eta)\xi,\widetilde{\Delta}(\xi\eta)^{-1}\eta),\
  \text{for some }\widetilde{\Delta}\in
  C^\omega((-\epsilon_{1},\epsilon_{1}),\R^*).
\end{align}
\end{lemma}

\begin{proof}
  Clearly, any map of the above form commutes with $N$.  The fact that
  $\mathcal{C}_N$ consists only of such maps follows from the fact
  that any map in $\mathcal{C}_N$ has to respect the symmetries of
  $N$; in particular, it has to map hyperbolas to hyperbolas, and
  preserve the rate of contraction/expansion along each of them.

  More precisely, let us fix a sufficiently small $\epsilon_1>0$ to be
  chosen later and introduce the shorthand notation
  $\mathcal{V} = \mathcal{V}_{\epsilon_{1}}$.  Take
  $F \in \mathcal{C}_N$, and denote
  $F\colon
  (\xi,\eta)\mapsto(u(\xi,\eta),v(\xi,\eta))$.

  We first show that $F$ must necessarily fix the axes $\{\xi=0\}$ and
  $\{\eta=0\}$, hence the origin $(0,0)$.  For any $j\in\Z$,
  by~\eqref{iteration de N}, and after projection on the two
  coordinates, for each $(\xi,\eta)\in \mathcal{V}$, we get
\begin{align*}
  u(\Delta(\xi\eta)^{j}
  \xi,\Delta(\xi\eta)^{-j}\eta) &= \Delta(u(\xi,\eta)v(\xi,\eta))^{j} u(\xi,\eta),\\
  v(\Delta(\xi\eta)^{j}
  \xi,\Delta(\xi\eta)^{-j}\eta) &=  \Delta(u(\xi,\eta)v(\xi,\eta))^{-j} v(\xi,\eta).
\end{align*}
For $\eta=0$, we have $\Delta(\xi\eta)=\Delta(0)=\lambda$, so that
\begin{align*}
  \Delta(u(\xi,0)v(\xi,0))^{j} v(\lambda^j \xi,0)= v(\xi,0).
\end{align*}
By letting $j \to +\infty$, we deduce that $v(\xi,0)=0$.  Similarly,
$u(0,\eta)=0$ for all $\eta\in \R$. In particular, $F$ fixes the
coordinate axes $\{\xi=0\}$ and $\{\eta=0\}$, and therefore
$F(0,0)=(0,0)$.

Let us now write
\begin{align*}
  u(\xi,\eta)=\sum_{k,\ell \geq 0} u_{k,\ell} \xi^k \eta^\ell.
\end{align*}
Fix $(\xi,\eta)\in \mathcal{V}\setminus \{(0,0)\}$. The above equation
yields:
\begin{align*}
\Delta^j(u(\xi,\eta)v(\xi,\eta))u(\xi,\eta)&=u(\Delta^j(\xi\eta) \xi,\Delta^{-j}(\xi\eta)\eta)\\
&=\sum_{k,\ell\geq 0} u_{k,\ell} \Delta^{j(k-\ell)}(\xi\eta) \xi^k \eta^\ell.
\end{align*}
The left hand side goes to zero as $j$ goes to infinity, hence
$u_{k,\ell}=0$ for $\ell \geq k$. We deduce that for
$j \gg 1$,\footnote{ By symplecticity, $u_{1,0}v_{0,1}=1$ hence the
  right hand side is different from zero for $\xi\eta\ll 1$.}
\begin{align*}
  \Delta^j(u(\xi,\eta)v(\xi,\eta))u(\xi,\eta)=\Delta^j(\xi\eta)
  \cdot\eta^{-1}\sum_{k\geq 1} u_{k,k-1}(\xi\eta)^k+o(\lambda^j),
\end{align*}
and thus, as $j\to +\infty$,
\begin{align*}
  \left(\frac{\Delta(u(\xi,\eta)v(\xi,\eta))}{\Delta(\xi\eta)}\right)^j=\frac{\eta^{-1}\sum_{k\geq 1} u_{k,k-1}(\xi\eta)^{k}}{u(\xi,\eta)}+o(1).
\end{align*}
Since the right hand side does not depend on $j$, we obtain
\begin{equation}\label{eq delta}
\Delta(u(\xi,\eta)v(\xi,\eta))=\Delta(\xi\eta),
\end{equation}
and also
\begin{equation}\label{eq uxieta}
u(\xi,\eta)=\eta^{-1}\sum_{k\geq 1} u_{k,k-1}(\xi\eta)^k. %=\sum_{k\geq 1} u_{k,k-1}\xi^k\eta^{k-1}.
\end{equation}
%Let us consider the case when the point $(\xi,\eta)$ lies on the hyperbola $\mathcal{H}_0=\{\xi\eta=c_0\}$, $c_0 \in \R^*$. Then
Let $c_1$ be the analytic function
$z \mapsto \sum_{k\geq 1} u_{k,k-1}z^k$, so that
$u(\xi,\eta)\eta=c_1(\xi\eta)$. Similarly, we have
$v(\xi,\eta)\xi=:c_2(\xi\eta)$ for some analytic function $c_2$. For
any $z \neq 0$, we set $c(z):= \frac{c_1(z)c_2(z)}{z}$.  Then, for
each $(\xi,\eta)\in \mathcal{V}\setminus\{(0,0)\}$, it holds
\begin{align*}
  u(\xi,\eta)v(\xi,\eta)=\frac{c_1(\xi\eta)c_2(\xi\eta)}{\xi\eta}=c(\xi\eta).
\end{align*}
Therefore, for $|c_\flat|$ sufficiently small, $F$ maps the hyperbola
$\{\xi\eta=c_\flat\}$ to the hyperbola $\{uv=c_\sharp\}$, with
$c_\sharp:=c(c_\flat)$.  By~\eqref{eq delta}, we also have
$\Delta(uv)=\Delta(\xi\eta)$, i.e.,
$\Delta(c(\xi\eta))=\Delta(\xi\eta)$. Note that
$\lim_{\xi\eta\to 0} c(\xi\eta)=\lim_{\xi\eta\to 0}
u(\xi,\eta)v(\xi,\eta)=0$.

By assumption, the Birkhoff invariants $(a_k)_{k \geq 1 }$ are not all
equal to zero.  Let $k_0 \geq 1$ be the smallest positive integer such
that $a_{k_0} \neq 0$. Then, $\Delta(z)-\lambda = a_{k_0}z^{k_0}+o(|z|^{k_0})$
for $|z|\ll 1$, and thus, assuming that $\epsilon_1>0$ was chosen sufficiently small,  $\Delta|_{(0,\epsilon_1)}$ is strictly
monotonic. It follows from the previous discussion
that for all $(\xi,\eta)\in \mathcal{V}$,
\begin{align*}
  \Delta(c(\xi\eta))-\lambda=\sum_{k \geq k_0} a_k(c(\xi\eta))^k=\sum_{k \geq k_0} a_k (\xi\eta)^k=\Delta(\xi\eta)-\lambda,
\end{align*}
and then, $c(\xi\eta)=\xi\eta$, by the strict monotonicity of $\Delta|_{(0,\epsilon_1)}$. In other words,
%some monotonic change of the Lyapunov exponent with respect to $\xi\eta$ near the origin;
since $F$ maps hyperbolas to hyperbolas, the local non-degeneracy of $\Delta$
together with
%and preserves the Lyapunov exponent, by
\eqref{eq delta} compel  $F$  to fix each hyperbola near the origin, i.e.,
\begin{align*}
  c(\xi\eta)=u(\xi,\eta)v(\xi,\eta)=\xi\eta.
\end{align*}
For any $(\xi,\eta) \in \mathcal{V}$ such that $\xi\eta\neq 0$, let us set
\begin{align*}
  \widetilde{\Delta}(\xi\eta):=\frac{u(\xi,\eta)}{\xi}=\frac{c(\xi\eta)}{v(\xi,\eta)\xi}=\frac{\xi\eta}{v(\xi,\eta)\xi}=\frac{\eta}{v(\xi,\eta)}.
\end{align*}
For any $(\xi,\eta)\in \mathcal{V}\setminus \{(0,0)\}$, we also have
$\widetilde{\Delta}(\xi\eta)=\frac{u(\xi,\eta)}{\xi}=\sum_{j\geq 0}
u_{j+1,j}(\xi\eta)^{j}$, thus we set
$\widetilde{\Delta}(0):=\lim_{(\xi,\eta)\to
  (0,0)}\frac{u(\xi,\eta)}{\xi}=u_{1,0}=v_{0,1}^{-1}$.  We conclude
that $\widetilde{\Delta} \in C^\omega((-\epsilon_1,\epsilon_1),\R^*)$,
and for each $(\xi,\eta)\in \mathcal{V}$:
\begin{align*}
  F(\xi,\eta)=(u(\xi,\eta),v(\xi,\eta))=(\widetilde{\Delta}(\xi\eta)\xi,\widetilde{\Delta}(\xi\eta)^{-1}\eta).&\qedhere
\end{align*}
\end{proof}

\begin{remark}\label{remark deux point quatre}
In $(s,r)$-coordinates, the horizontal axis $\{r=0\}=\{\varphi=0\}$ plays a special role, because of the reflection symmetry of the billiard map:
\begin{align*}
  \mathcal{F}(s,r)=(s',r')\quad \Longleftrightarrow\quad \mathcal{F}(s',-r')=(s,-r).
\end{align*}
This time-reserval symmetry also exchanges the stable and unstable spaces.

In $(\xi,\eta)$-coordinates, the stable space is the horizontal axis
$\{\eta=0\}$, while the unstable space is the vertical axis
$\{\xi=0\}$. Moreover, $2$-periodic points are on the axis of symmetry
$\{r=0\}$ -- and more generally, all the points associated to
perpendicular bounces in palindromic orbits -- hence their stable and
unstable manifolds are symmetric with respect to $\{r=0\}$. It is thus
natural to require the new axis of symmetry to be $\{\xi=\eta\}$. By
the previous study, under some non-degeneracy condition, maps in the
centralizer of $N$ translate points along hyperbolas
$\{\xi\eta=\mathrm{cst}\}$, hence typically, they do not preserve the
axis $\{\xi=\eta\}$. As a consequence, there is a canonical choice for
the conjugacy map $R$ defined above, which preserves this symmetry.
\end{remark}

In the following, we assume that the Birkhoff invariants
$(a_k)_{k \geq 1 }$ are not all equal to zero, and that the
neighborhood $\mathcal{U}$ in the definition of the change of
coordinates $R$ introduced at the beginning of
Section~\ref{birkkkhhfo} is sufficiently small such that the
neighborhood $\mathcal{V}:=R(\mathcal{U})\subset \R^2$ of $(0,0)$
satisfies $\mathcal{V} \subset \mathcal{V}_{\epsilon_{1}}$, where
$\epsilon_1$ satisfies the conclusion of Lemma~\ref{generique
  centralisateur}.

\begin{corollary}\label{conjug canonique}
%Let $\mathcal{D}\in \billiards$ be such
  Assume that the Birkhoff invariants $(a_k)_{k \geq 1 }$ are not all
  equal to zero.  Let
  $\mathcal U^{*} = \mathcal U \cap \mathcal F^{-2}\mathcal U$. Then,
  there exists a unique map $R_0%=R_0(\mathcal{D})
  \in\mathrm{Sympl}^\omega(\mathcal U,\R^2)$ such that
 \begin{align*}
   R_0\circ\mathcal{F}^2|_{\mathcal{U}^{*}}=N\circ R_0|_{\mathcal{U}^{*}},\qquad \text{and}\qquad R_0(\{(s,0)\ s \geq 0\})\subset\{(\xi,\xi),\ \xi\geq 0\}.
 \end{align*}
\end{corollary}

\begin{proof}
  Let us start by showing the uniqueness of $R_0$. Let
  $R,\widetilde{R}$ be two such maps. Then
  $R^{-1}\circ N \circ R=\widetilde{R}^{-1} \circ N \circ
  \widetilde{R}=\mathcal{F}^2$, hence, letting
  $F:=\widetilde{R}\circ R^{-1}$, it holds
\begin{equation}\label{centr r r tilde}
  F\circ N =N\circ F.
\end{equation}
By construction of $R$ and $\widetilde{R}$, the map $F$ fixes the
horizontal and vertical axes $\{\xi=0\}$ and $\{\eta=0\}$; as noted in
Remark~\ref{rmk:extension-centralizer}, it can be extended to a
neighborhood $\mathcal V_{\epsilon}$ for some
$\epsilon < \epsilon_{1}$.  By \eqref{centr r r tilde},
$F\in \mathcal{C}_N(\epsilon)$ (recall Definition \ref{defi
  centr}). By Lemma \ref{generique centralisateur}, the centralizer
$\mathcal{C}_N(\epsilon)$ is reduced to the set of maps which
translate points along hyperbolas $\{\xi\eta=\text{cst}\}$, and then
\begin{align*}
  F(\xi,\eta)=(\widetilde{\Delta}(\xi\eta)\xi,\widetilde{\Delta}(\xi\eta)^{-1}\eta),
\end{align*}
for some real analytic map $\widetilde{\Delta}\in C^\omega((-\epsilon,\epsilon),\R^*)$, $\epsilon>0$. Since both $R$ and $\widetilde{R}$ fix the positive axis $\{(\xi,\xi),\ \xi\geq 0\}$, so does $N_{\widetilde{\Delta}}$, and then,
\begin{align*}
  (\widetilde{\Delta}(\xi^2)\xi,\widetilde{\Delta}(\xi^2)^{-1}\xi)=(\tilde\xi(\xi),\tilde\xi(\xi)),\qquad \forall\, \xi \in \R,
\end{align*}
for some function $\tilde\xi \colon \R\to \R$. By taking the product
of the two coordinates, we deduce that $\xi^2=(\tilde\xi(\xi))^2$, and
then $\tilde\xi(\xi)=\xi$, since $\xi,\tilde\xi\geq 0$.  We deduce
that $\widetilde{\Delta}\equiv 1$, and then
$F=\widetilde{R}\circ R^{-1}=\mathrm{id}$, which concludes the proof of
uniqueness.\\

To show the existence of such a map $R_0$, let us fix an analytic
symplectomorphism $R(s,r)=(\xi(s,r),\eta(s,r))$ such that
$R \circ T \circ R^{-1}=N$.  After possibly composing $R$ with
$-\mathrm{id}$,\footnote{ By symplecticity, we do not need to consider
  the reflections $(\xi,\eta)\mapsto(\xi,-\eta)$ or
  $(\xi,\eta)\mapsto(-\xi,\eta)$.}  we may assume that
$\xi(s,0),\eta(s,0) \geq 0$ for all $s \geq
0$. %For $s > 0$, we set $\delta(s):=\sqrt{\frac{\xi(s,0)}{\eta(s,0)}}$
Let
$R^{-1}\colon
(\xi,\eta)\mapsto(\mathcal{S}(\xi,\eta),\mathcal{R}(\xi,\eta))$, and
let $\pi\colon(\xi,\eta)\mapsto (\pi_1(\xi\eta),\pi_2(\xi\eta))$ be
the projection along hyperbolas $\{\xi\eta=\mathrm{cst}\}$ onto the
set $\{\mathcal{R}(\xi,\eta)=0\}$.  We denote by
$\theta\in [0,\frac\pi 2]$ the angle between the positive parts of the
horizontal axis and of the unstable space of $T$. Since the coordinate
$\eta$ vanishes only on the stable space $\{r=\tan(\theta)s\}$, we may
define
$\delta(\xi\eta):=\sqrt{\frac{\pi_2(\xi\eta)}{\pi_1(\xi\eta)}}$, and
we set
$N_\delta\colon (\xi,\eta)\mapsto
(\delta(\xi\eta)\xi,\delta^{-1}(\xi\eta)\eta)$.  Clearly,
$N_{\delta}\in \mathrm{Sympl}^\omega(\R^2,\R^2)$, and
$N_{\delta} (\{\mathcal{R}(\xi,\eta)=0,\ \xi,\eta \geq
0\})=\{(\tilde{\xi},\tilde{\xi}),\ \tilde{\xi} \geq 0\}$. Then, the
map $R_0\colon(s,r)\mapsto N_{\delta}\circ R(s,r)$ satisfies the
required conditions:
\begin{align*}
  R_0(s,0)=(\delta(\xi(s,0)\eta(s,0))\xi(s,0),\delta(\xi(s,0)\eta(s,0))^{-1}\eta(s,0))\in \{(\tilde{\xi},\tilde{\xi}),\ \tilde{\xi}\geq 0\},
\end{align*}
and $R_0 T R_0^{-1}=N_{\delta}RTR^{-1}N_{\delta}^{-1}=N_{\delta}NN_{\delta}^{-1}=N$.
\end{proof}
We call \emph{Birkhoff coordinates} the coordinates $(\xi,\eta)$
obtained via the change of coordinates $R_{0}$.

\subsection{The time reversal involution in Birkhoff coordinates}

By the time reversal property, the map
$\mathcal{I}\colon (s,r)\mapsto (s,-r)$ conjugates the
billard map $\mathcal{F}$ to its inverse $\mathcal{F}^{-1}$, and thus,
$\mathcal{I} \circ T \circ \mathcal{I}=T^{-1}$.  Assume that the Birkhoff invariants $(a_k)_{k \geq 1 }$ are not all
equal to zero, and let $R_0\colon \mathcal{U} \to \R^2$ be the
canonical symplectic change of coordinates given by Lemma~\ref{conjug
	canonique}. Since
$R_0\circ T \circ R_0^{-1}=N$, we get
\begin{align*}
	(R_0\circ \mathcal{I}\circ R_0^{-1})\circ N \circ (R_0\circ \mathcal{I} \circ R_0^{-1})=R_0 \circ  T^{-1} \circ R_0^{-1}=N^{-1}.
\end{align*}
Set $\mathcal{I}^*:=R_0\circ \mathcal{I} \circ R_0^{-1}$. We thus have
\begin{equation}\label{anti commm}
\mathcal{I}^* \circ N \circ \mathcal{I}^* = N^{-1}.
\end{equation}

\begin{lemma}\label{lemma reflexion map}
	The map $\mathcal{I}^*$ is the reflection along the bisectrix
	$\{\xi=\eta\}$:
	\begin{align}\label{prop istar}
		\mathcal{I}^*=\mathcal{I}_0\colon (\xi,\eta)\mapsto(\eta,\xi).
	\end{align}
\end{lemma}
\begin{proof}
	Let us write $\mathcal{I}^*(\xi,\eta)=(u,v)$, with $u=u(\xi,\eta)$
	and $v=v(\xi,\eta)$. For every $(\xi,\eta)\in \R^2$, we have
	\begin{align}\label{first rem}
		u(\xi,0)=0  ,\qquad v(0,\eta)=0.
	\end{align}
	In other words, $\mathcal{I}^*$ maps the horizontal axis
	$\{\eta=0\}=\{(\xi,0): \xi \in \R\}$ to the vertical axis
	$\{\xi=0\}=\{(0,\eta): \eta \in \R\}$, and vice versa. Indeed it
	follows fom the definition of the map $N$ that $\{\eta=0\}$ is the
	stable manifold of $(0,0)$, since $N^j(\xi,0)=(\lambda^j \xi,0)$,
	for $j\geq 0$, and similarly, $\{\xi=0\}$ is the unstable manifold
	of the origin. Moreover,~\eqref{anti commm} implies that $N$
	exchanges the stable manifold with the unstable manifold: given
	$p \in \R^2$ such that $\lim_{j \to +\infty} N^j(p)=(0,0)$, then its
	image $p^*:=\mathcal{I}^*(p)$ satisfies
	\begin{align*}
		\lim_{j \to+\infty} N^{-j} (p^*)=\mathcal{I}^* (\lim_{j \to+\infty} N^j(p))=\mathcal{I}^* (0,0)=(0,0).
	\end{align*}
	Here, we have used that $\mathcal{I}^*(0,0)=(0,0)$ (by~\eqref{first
		rem}).

	Moreover, by~\eqref{anti commm}, we know that
	$\mathcal{I}^*\circ N^{-1}=N\circ \mathcal{I}^*$.
	Therefore, %let us set $u^*(\eta):=u(0,\eta)$, and $v^*(\xi):=v(\xi,0)$. For any $\eta\in \R$, we have
	given any $(\xi,\eta)\in \R^2$, we obtain
    \begin{align*}
      \mathcal{I}^*(\Delta(\xi\eta)^{-1}\xi,\Delta(\xi\eta)\eta)=(\Delta(u(\xi,\eta)v(\xi,\eta))u(\xi,\eta),\Delta(u(\xi,\eta)v(\xi,\eta))^{-1}v(\xi,\eta)).
    \end{align*}
	In particular, by considering the projection on the first coordinate, we get
	\begin{equation}\label{rel for u}
	u(\Delta(\xi\eta)^{-1}\xi,\Delta(\xi\eta)\eta)=\Delta(u(\xi,\eta)v(\xi,\eta))u(\xi,\eta).
	\end{equation}
	For $\xi=0$, by the power series expansion of $\Delta$, and by~\eqref{first rem}, we have $\Delta(\xi\eta)=\Delta(0\cdot \eta)=\Delta(0)=\lambda$ and $\Delta(u(\xi,\eta)v(\xi,\eta))=\Delta(u(0,\eta)v(0,\eta))=\Delta(u(0,\eta)\cdot 0)=\Delta(0)=\lambda$. We deduce from~\eqref{rel for u} that for any $\eta \in \R$,
    \begin{align*}
      u(0,\lambda \eta)=\lambda u(0,\eta).
    \end{align*}
	By considering the power series expansion $u(\xi,\eta)=\sum_{k,\ell \geq 0}u_{k,\ell} \xi^k \eta^\ell$, this relation implies that $u_{0,\ell}=0$ for all $\ell \neq 1$, and then,
    \begin{align*}
      u(0,\eta)=u_{0,1} \eta.
    \end{align*}
	Besides, for any $(\xi,\eta)\in \R^2$, and any $j \geq 0$, we have $\mathcal{I}^*\circ N^j=N^{-j}\circ \mathcal{I}^*$. Similarly, by projecting on the first coordinate, we obtain
	\begin{equation}\label{rel for u bis}
	u(\Delta(\xi\eta)^{j}\xi,\Delta(\xi\eta)^{-j}\eta)=\Delta(u(\xi,\eta)v(\xi,\eta))^{-j}u(\xi,\eta).
	\end{equation}
	For any $j \geq 0$, we have%\marginpar{not rigorous: find an argument to show that the derivatives satisfy $|\partial_\xi^{(k)} u(0, \eta)|=O(|\eta|^k)$}
	\begin{align*}
		u(\xi,\eta)&=\Delta(u(\xi,\eta)v(\xi,\eta))^j\cdot u(\Delta(\xi\eta)^j \xi,\Delta(\xi\eta)^{-j} \eta)\\
		&=\left(\frac{\Delta(u(\xi,\eta)v(\xi,\eta))}{\Delta(\xi\eta)}\right)^{j}\cdot \left(u_{0,1} \eta+\Delta(\xi\eta)^{j}\sum_{k=1}^{+\infty} \sum_{\ell=0}^{+\infty} u_{k,\ell}\Delta(\xi\eta)^{j(k-\ell)} \cdot \xi^k \eta^\ell\right),
	\end{align*}
	where we have used that $u(0,\Delta(\xi\eta)^{-j} \eta)=u_{0,1} \Delta(\xi\eta)^{-j} \eta$.
	Since the left hand side is bounded independently of $j$, then, arguing as  in Lemma~\ref{generique centralisateur}, we get $\Delta(u(\xi,\eta)v(\xi,\eta))=\Delta(\xi\eta)$, and
    \begin{align*}
      u(\xi,\eta)=u_{0,1} \eta.
    \end{align*}
	Similarly, there exists $v_{1,0} \in \R$ such that $v(\xi,\eta)=v_{1,0} \xi$. Since $\mathcal{I}^*$ is anti-symplectic ($R_0$ is symplectic and $\mathcal{I}$ is anti-symplectic), we have
    \begin{align*}      d\xi \wedge d\eta=dv \wedge du=(u_{0,1}v_{1,0}) d\xi \wedge d\eta,
    \end{align*}
	and then $u_{0,1},v_{1,0} \in \R^*$, and $v_{1,0}=u_{0,1}^{-1}$. Besides, $R_0^{-1}=(S,\Phi)$ maps $\{(\xi,\xi),\ \xi \geq 0\}$ to $\{(s,0),\ s \geq 0\}$, hence for any $\xi\geq 0$, we have
    \begin{align*}
      (u_{0,1}\xi,u_{0,1}^{-1}\xi)=\mathcal{I}^*(\xi,\xi)=R_0\circ \mathcal{I}(S(\xi,\xi),0)=R_0(S(\xi,\xi),0)\in \{(\xi,\xi),\ \xi \geq 0\},
    \end{align*}
	and then $u_{0,1}=v_{1,0}=1$. We conclude that
	\begin{align*}
		\mathcal{I}^*(\xi,\eta)&=(\eta,\xi).\qedhere
	\end{align*}
\end{proof}

\begin{remark}
	Note that~\eqref{prop istar} can also be obtained as follows:  by~\eqref{anti commm}, both $\mathcal{I}^*=R_0\mathcal{I}R_0^{-1}$ and $\mathcal{I}_0$ conjugate $N$ with $N^{-1}$, hence $\mathcal{I}^*\circ \mathcal{I}_0^{-1}$ is in the centralizer of $N$. By Lemma~\ref{generique centralisateur} and since $\mathcal{I}^*,\mathcal{I}_0$ preserve the bisectrix $\{\xi=\eta\}$ (as $R_0$ does), we conclude that $\mathcal{I}^*=\mathcal{I}_0$.
\end{remark}

\section{Extension of the Birkhoff coordinates along the separatrices}\label{section extension}

Let us fix a billiard table $\mathcal{D}\in \billiards$.  In this
section, we consider the Birkhoff Normal Form $N$ introduced above for the $2$-periodic $(12)$ and
we assume that the Birkhoff invariants $(a_k)_{k \geq 1 }$ are not all
equal to zero.  We   denote by $R_0\colon \mathcal{U} \to \R^2$ the
canonical symplectic change of coordinates given by Lemma~\ref{conjug
  canonique} and we set $\mathcal{V}:=R_0(\mathcal{U})$.  We will also
use the notation introduced at the beginning of
Section~\ref{birkkkhhfo}.

Up to this point, the model for the dynamics of $T$ given by its
Birkhoff Normal Form $N$ only accounts for the dynamics in a
neighborhood of the $2$-periodic orbit.  In this section, we explain how
to extend this model in such a way that it also describes the global
dynamics of the palindromic orbits $(h_n)_{n \geq 1}$ introduced
earlier. In the $(\xi,\eta)$-coordinates, the only non-wandering point
of the map $N$ is the origin $(0,0)$; to describe recurrence
properties of the dynamics of $T$, we explain a \emph{gluing}
construction for some points in this model, for which we have more
information due to additional symmetries.  This is, in particular, the
case for the palindromic orbits $(h_n)_{n \geq 1}$, which have two
symmetries, and for which we have a good control on the gluing
map. Moreover, for $n$ large enough, those orbits always stay in a
neighborhood of the separatrices, and the local dynamics of $N$ near
the fixed point is sufficient to describe them, based on the relation
$NR_0=R_0T$ which can be used to extend the system of coordinates by
the dynamics. Although this relation is only true locally (some points
escape in the billiard dynamics, so the map $T$ is not everywhere
defined), it is sufficient for our purpose, which consists in
determining explicitly a link between the Birkhoff invariants and the
Lyapunov exponents of the palindromic orbits.  The extension of the
coordinates to a neighborhood of the separatrices that we describe in
the following can be seen as a hyperbolic analogue of the local
coordinates in a neighborhood of the boundary given by the Birkhoff
Normal Form in the elliptic setting, which was used, for instance,
in~\cite{CdV3}.\\

After possibly replacing $\mathcal{U}$ with
$\mathcal{U}\cap \mathcal{I} (\mathcal{U})$, where
$\mathcal{I}\colon (s,r)\mapsto(s,-r)$, we can assume that
the neighborhood $\mathcal{U}$ is symmetric with respect to the axis
$\{r=0\}$.  For any sufficiently large odd integer $n=2m-1$, and
after a certain time, Proposition~\ref{prop bdskl} implies that the
iterates under $T=\mathcal{F}^2$ of the point $x_n(1)$ in the
palindromic orbit $h_n$ are contained in the neighborhood
$\mathcal{U}$.  More precisely, there exists $m_0\geq 0$ such that if
$n=2m-1\geq n_0:=2m_0-1$, we have $x_{n}(2k+1) \in \mathcal{U}$, for
all $k \in
\{m_0,m_0+1,\cdots,2m-m_0-1\}$. %(by the palindromic symmetry, it is natural to take this collection to be symmetric with respect to the point $x_n(2m)$).
We denote by $(\xi_{n}(2k+1),\eta_{n}(2k+1))$ the coordinates of the
point $R_0(x_{n}(2k+1))$.  The contraction rate $\Delta$ is constant
along this orbit segment: for any integer
$k \in \{m_0,m_0+1,\cdots,2m-m_0-1\}$, we have
\begin{align*}
  \Delta(\xi_{n}(2k+1)\eta_{n}(2k+1))=\Delta(\xi_{n}(2m-1)\eta_{n}(2m-1))=:\Delta_{n}.%\quad k \in \{2m_0,2(m_0+1),\cdots,4m-2m_0\}.
\end{align*}

By Lemma~\ref{easy claim}, the map $N$ is well defined in a
neighborhood of the coordinate axes.  In particular, due to the
relations $R_0=NR_0T^{-1}$ and $R_0=N^{-1}R_0T$, %\footnote{The second relation follows from the first one and Lemma \ref{lemma reflexion map} below about the image of the time reversal involution in Birkhoff coordinates.}
it is possible to
extend the system of coordinates given by $R_0$ to a neighborhood of
the separatrices as follows.

Let $\mathcal{NE}^{- 1}$ be the set of all parameters
$(s,r)\in \mathcal{M}$ in the collision space such that
$T^{-1}(s,r)$ is well defined, i.e., such that both
$\mathcal{F}^{-1}(s,r)$ and $\mathcal{F}^{-2}(s,r)$ are well defined.  For
any $(s,r) \in T^{- 1}(\mathcal{U}\cap \mathcal{NE}^{-1})$,
we have $T (s,r) \in \mathcal{U}$, thus we can set
$R_{-1}(s,r):=N^{-1} R_0 T(s,r)$. By
induction, for each integer $\ell \geq 2$, we define
\begin{align*}
  \mathcal{NE}^{- \ell}:=\{(s,r)\in \mathcal{NE}^{-1}: T^{-1}(s,r) \in \mathcal{NE}^{-(\ell-1)}\},
\end{align*}
and for each
$(s,r)\in \mathcal{U}^{-\ell}:=T^{-
  \ell}(\mathcal{U}\cap\mathcal{NE}^{- \ell})$, we set
\begin{align*}
  R_{- \ell}(s,r):=N^{-1} R_{- (\ell-1)} T(s,r)=N^{- \ell} R_0 T^{ \ell}(s,r).
\end{align*}

On
$\mathcal{U}^{- (\ell-1)}\cap\mathcal{U}^{-
  \ell}$, it holds $N^{-1} R_{-(\ell-1)} T=R_{-(\ell-1)}$,
hence $R_{- \ell}$ coincides with $R_{-(\ell-1)}$ on this set. We
define $R_{-}$ as the map obtained in this way by extending the
conjugacy $R_0$ to a neighborhood of the arc of the stable manifold
between the points $(0,0)$ and
$x_\infty(1)$. %using the dynamics of $N$.
More precisely, $R_-$ is defined on
$\mathcal{U}^{-}:=\bigcup_{\ell=0}^{m_0}\mathcal{U}^{-
  \ell}$ as follows: for any $\ell \in \{0,\cdots,m_0 \}$
and
$(s,r) \in \mathcal{U}^{- \ell}\backslash\bigcup_
{k=0}^{\ell-1}\mathcal{U}^{- k}$, we set
$R_{-}(s,r):=R_{- \ell}(s,r)$. In a symmetric way, we
define $\mathcal{U}^{+}$ and extend $R_0$ to a map
$R_+$ defined on a neighborhood of the arc of the unstable manifold
between the points $(0,0)$ and $x_\infty(-1)$.

By the above remark, for any integer $n=2m-1 \geq n_0$, every point $h_n(k)$  labeled with some odd integer $k$ belongs to the set
$ \mathcal{U}^{+ }\cup\mathcal{U}^{-}$, thus
it has an image either by $R_+$ or $R_-$. We let $R$ be the map
defined on
$ \mathcal{U}^{+ }\cup\mathcal{U}^{-}$ by
$R|_{\mathcal{U}^{\pm}}:=R_\pm$. For each
$k \in \{0,\cdots,m-m_0-1\}$, we have
$R(x_n(2m\pm(2 k+1)))=R_{0}(x_n(2m\pm(2 k+1)))$, while for
$k \in \{m-m_0,\cdots,m-1\}$, the point
$R(x_n(2m\pm(2 k+1)))=R_{\pm (m-k)}(x_n(2m\pm(2 k+1)))$ is well
defined.

Moreover,  %for all $k \in \{0,\cdots,m-1\}$, and
for some neighborhood $\mathcal{U}_n \subset \mathcal{U}$ of the point $x_n(2m)=x_n(n+1)$, it follows from the above definitions that
\begin{equation}\label{eq extenenen R pal}
R \circ T^{\pm k}|_{\mathcal{U}_n}=N^{\pm k} \circ R|_{\mathcal{U}_n}=N^{\pm k} \circ R_0|_{\mathcal{U}_n},\qquad \forall\, k \in \{0,\cdots,m-1\}.
\end{equation}
More generally, %\supset \cup_{m\geq m_0} \mathcal{U}_{2m-1}$
%for any neighborhood $\mathcal{U}''$ of $\{\varphi=0\} $ near $(s_\infty(0),0)$,
% there exists a symmetric neighborhood $\mathcal{U}'$ of $\{\varphi=0\} \cap \mathcal{U}$  %of $\{\varphi=0\}$
%such that
for each $x \in \mathcal{U}^{+ }\cup\mathcal{U}^{-}$ and each integer $k$ such that $x,\cdots, T^k(x)\in \mathcal{U}^{+ }\cup\mathcal{U}^{-}$, %\text{ or }T^{-\ell'}(s,\varphi) \notin \mathcal{U}\text{ and }T^{-\ell}(s,\varphi) \in \mathcal{U}\}$,
we have
\begin{equation}\label{eq extenenen R}
R \circ T^{k}(x)=N^{k} \circ R(x).
\end{equation}

The next lemma says that after the extension, the image of the time reversal involution is still given by the map $\mathcal{I}_0\colon (\xi,\eta)\mapsto (\eta,\xi)$:
\begin{lemma}
	The extended system of coordinates $R$ satisfies
	\begin{equation}\label{plus general}
	R\circ \mathcal{I} \circ R^{-1} = \mathcal{I}_0.
	\end{equation}
\end{lemma}

\begin{proof}
	It follows directly from~\eqref{eq extenenen R} and
	Lemma~\ref{lemma reflexion map}.
\end{proof}

In particular,  by \eqref{eq extenenen R}, for each $n=2m-1$ with $m \geq m_0$, and

for $k\in \{0,\cdots,n\}$, it holds
\begin{align*}
  R(x_n(2k+1))=(\xi_{n}(2k+1),\eta_{n}(2k+1)):=((\Delta_{n}^{m-1-k})^{-1}\xi_{n}(n),\Delta_{n}^{m-1-k}\eta_{n}(n)).
\end{align*}
%and analogously for $k\in \{2m-m_0+1,\cdots,2m-1\}$.
Let us abbreviate
\begin{align*}
  R(x_n(1))=(\xi_{n},\eta_{n}):=((\Delta_{n}^{m-1})^{-1}\xi_{n}(n),\Delta_{n}^{m-1}\eta_{n}(n)).
\end{align*}
\begin{claim}\label{claim eta n delta n xi n}
	It holds
	\begin{align*}
		\eta_n=\Delta_n^{n} \xi_n.
	\end{align*}
\end{claim}
\begin{proof}
	By \eqref{plus general}, we have $R(x_n(-1))=R\circ \mathcal{I}(x_n(1))=\mathcal{I}_0\circ R(x_n(1))=\mathcal{I}_0(\xi_n,\eta_n)=(\eta_n,\xi_n)$. Therefore,
	\begin{align*}
		(\eta_n,\xi_n)=R(x_n(-1))=R(x_n(2n+1))=R \circ T^{n}(x_n(1))=N^{n} (\xi_n,\eta_n),
	\end{align*}
	which gives the required identity.
\end{proof}
Then, we have
\begin{equation}\label{eq new parmaetere}
(\xi_{n}(2k+1),\eta_{n}(2k+1))=(\Delta_{n}^{k} \xi_n, \Delta_{n}^{-k} \eta_n)=\xi_n(\Delta_{n}^{k}, \Delta_{n}^{n-k}),\quad \forall\, k \in \{0,\cdots,n\}.
\end{equation}

In the same way as above, we can extend our system of coordinates such that the images of the forward iterates of the point $x_\infty(1)$ in the homoclinic orbit $h_\infty$  are
\begin{align*}
  R(x_\infty(2k+1))=R_-(x_\infty(2k+1))=(\xi_\infty(2k+1),0)=(\lambda^k \xi_\infty,0), \quad \forall\, k =0,1,\cdots
\end{align*}

for some $\xi_\infty \in \R\setminus \{0\}$ (the second coordinate
vanishes since we are on the stable manifold $\{\eta=0\}$ of the
origin). Recall that $\mathcal{I}\colon (s,r)\mapsto (s,-r)$. By
\eqref{palindr symm}, %Proposition~\ref{prop bdskl},
for all $n=2m-1\geq n_0$, $k \geq 0$, we have
$x_n(-2k-1)=x_n(2n+2-(2k+1))=%x_n(2(2m-1-k)+1)=
\mathcal{I}(x_n(2k+1))$.  Thus, we extend analogously the coordinates
in the past, such that the preimages of $x_\infty(-1)$ have
coordinates $R(x_\infty(-2k-1))=%\mathcal{I}_0\circ
R_+(x_\infty(-2k-1))$, i.e.,
\begin{align*}
  R(x_\infty(-2k-1))=(0,\xi_\infty(-2k-1))=(0,\lambda^k \xi_\infty), \quad\forall\,  k=1,2,\dots
\end{align*}

\begin{remark}
  In the previous construction, we stop the extension after the time
  $\pm m_0$ where we reach a neighborhood of the point
  $x_\infty(\pm 1)$. Indeed, after that time, in the initial
  $(s,r)$-collision space, the neighborhoods of the separatrices start
  to overlap; in particular, they both contain a neighborhood of the
  point $x_\infty(0)$ on the third obstacle. Besides, the point of
  this construction is to study the dynamics of the map $T$ through
  its Birkhoff Normal Form $N$. Note that the latter only depends on
  the obstacles $\obs_1,\obs_2$. By analyticity, as long as the points
  bounce between the first two obstacles, it is legitimate to replace
  the billiard dynamics with that of $N$, but it does not carry any
  meaningful information once the points reach the third obstacle.
\end{remark}

\begin{figure}[H]
  \begin{center}
    \includegraphics [width=11cm]{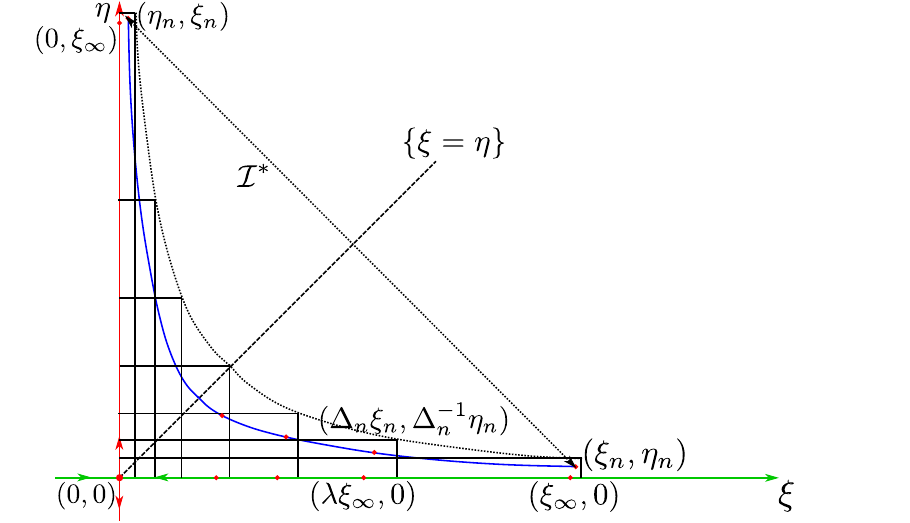}
        \caption{$(\xi,\eta)$-representation of the points in the palindromic orbit $h_n$.}
\end{center}
\end{figure}

\section{Marked Lyapunov Spectrum and Birkhoff invariants}\label{section marked lyapuno}

\subsection{Preliminary estimates on the parameters}\label{subs prliemr}

Recall that $T=\mathcal{F}^2$ where $\mathcal{F}$ is the billiard map
and that $\mathcal{M}_i$ denotes the set of $(s,r)$-coordinates
of collisions emanating from the $i^{\text{th}}$ scatterer.

We let $\mathcal{U}\subset \R^2$ be a small neighborhood of $(0,0)$ as
defined in Section~\ref{birkkkhhfo} and denote
$\mathcal V = R_{0}\mathcal U\subset \R^{2}$; recall that by
construction of $R_{0}$ For any $x=(s,0)\in \mathcal{U}\cap \{r=0\}$,
we have $R_0(x)\in \mathcal{V}\cap \{\xi=\eta\}$.

Let $\mathscr{O}_\infty\subset \mathcal{M}_2$ be a small neighborhood
of the point $x_\infty(-1)$. We denote by
$\Omega_\infty:=R(\mathscr{O}_\infty)$ the image of
$\mathscr{O}_\infty$ in Birkhoff coordinates, and we let
\begin{align*}
  \mathcal{G}:=R\circ T \circ R^{-1}|_{\Omega_\infty}=R_-\circ T \circ R_+^{-1}|_{\Omega_\infty}
\end{align*}
be the gluing map between $R_+$ and $R_-$. It satisfies the
time reversal property
\begin{align}\label{renve temps}
  \mathcal{G}^{-1}=\mathcal{I}_0\circ \mathcal{G} \circ \mathcal{I}_0,\qquad \mathcal{I}_0\colon (\xi,\eta)\mapsto (\eta,\xi).
\end{align}

\begin{figure}[H]
	\begin{center}
		\includegraphics [width=11cm]{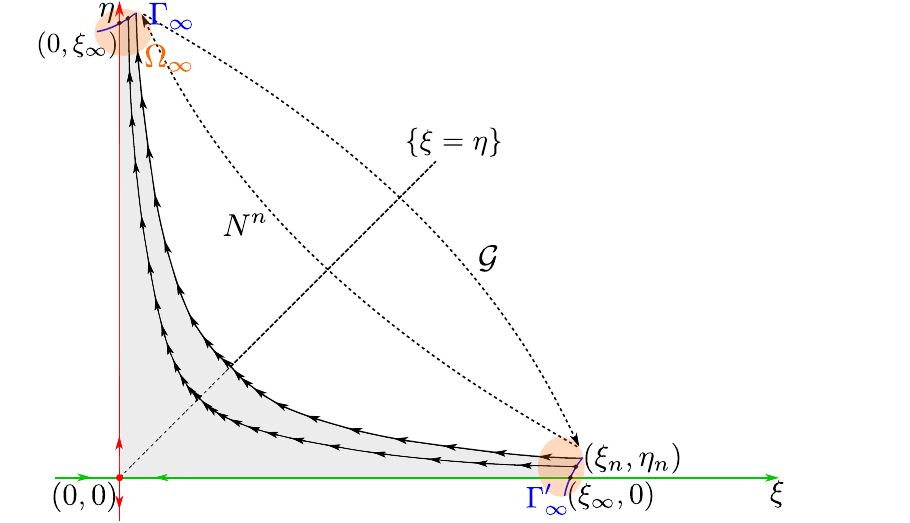}
		\caption{Gluing map between the extensions $R_+,R_-$.}\label{picture dynamique}
	\end{center}
\end{figure}

Let
$\mathscr{A}_\infty:= \mathscr{O}_\infty \cap
\mathcal{F}^{-1}(\{r=0\})$ be the curve in $\mathscr{O}_\infty$
containing $x_\infty(-1)$ and made of points whose image under the
billiard map $\mathcal{F}$ is associated to an orthogonal collision on
$\mathcal{O}_{3}$; let
$\Gamma_\infty:=R(\mathscr{A}_\infty)\subset \Omega_\infty$ be the
image of $\mathscr{A}_\infty$ in Birkhoff coordinates, and set
$\Gamma_\infty':=\mathcal{G}(\Gamma_\infty)\subset
\mathcal{G}(\Omega_\infty)$.

\begin{claim}\label{claim petitt}
  For any $x=(s,r)\in \mathscr{A}_\infty$, we have
  $T(s,r)=\mathcal{I}(x)=(s,-r)$, and
  $DT_{(s,r)}\in \mathrm{SL}(2,\R)$. Analogously, for any
  $(\xi,\eta)\in \Gamma_\infty$, it holds
  $\mathcal{G}(\xi,\eta)=\mathcal{I}_0(\xi,\eta)=(\eta,\xi)$, and
  $D\mathcal{G}_{(\xi,\eta)} \in \mathrm{SL}(2,\R)$. In particular, we
  have $\Gamma_\infty'=\mathcal{I}_0(\Gamma_\infty)$.
\end{claim}
\begin{proof}
  Let $x=(s,r)\in \mathscr{A}_\infty$. The point
  $\mathcal{F}(x)\in \{r=0\}$ is invariant under the involution
  $\mathcal{I}\colon (s,r)\mapsto (s,-r)$, hence, using
  Remark~\ref{remark deux point quatre},
  \begin{align*}
    T(x)=\mathcal{F} (\mathcal{F}(x))=\mathcal{F} \circ
    \mathcal{I}(\mathcal{F}(x))= \mathcal{I}\circ \mathcal{F}^{-1}
    (\mathcal{F}(x))=\mathcal{I} (x).
  \end{align*}
  Moreover, by~\eqref{matrice sl deux}, %the formula (2.29) in~\cite{CM},
  we have
  %\begin{align*}
    $\det D T_x=\det D\mathcal{F}^2_{(s,r)}  %=\frac{\cos(\varphi)}{\cos(-\varphi)}=
    =1$.
 % \end{align*}
  By~\eqref{plus general}, $R\circ \mathcal{I}=\mathcal{I}_0 \circ R$,
  hence for any $(\xi,\eta)=R(x)\in \Gamma_\infty$, we also have
  \begin{align*}
    \mathcal{G}(\xi,\eta)=R T R^{-1} (R(x))=R T (x)=R\circ \mathcal{I} (x) =\mathcal{I}_0 \circ R (x)=\mathcal{I}_0(\xi,\eta),
  \end{align*}
  and $\det D\mathcal{G}_{(\xi,\eta)}=\det DT_x=1$.
%hence $R (x_n(-2))=\mathcal{I}_0(R(x_n(2)))=(\eta_n,\xi_n)$. Now, $(\eta_n,\xi_n)=R (x_n(-2))\in \Gamma_\infty'$, so that $\eta_n=\gamma(\xi_n-\xi_\infty)$.
\end{proof}

As in Section~\ref{section extension}, for any large integer $n\geq n_0$,
we let $(\xi_n,\eta_n):=R_-(x_n(1))$, $\Delta_n:=\Delta(\xi_n\eta_n)$,
and we let
$(\xi_\infty,0):=R_-(x_\infty(1))=\lim_{n \to +\infty}
(\xi_n,\eta_n)$.

We let
$\gamma\colon \xi \mapsto \sum_{j = 1}^{\infty} \gamma_j \xi^j $ be
the analytic function such that $\Gamma_\infty$ is the graph of
$\xi_\infty+\gamma(\cdot)$, i.e., for any
$(\xi,\eta)\in \Gamma_\infty$, we have $\eta=\xi_\infty+
\gamma(\xi)$. As we have seen in Lemma~\ref{claim petitt},
$\Gamma_\infty'=\mathcal{I}_0(\Gamma_\infty)$, hence for any
$(\xi,\eta)\in \Gamma_\infty'$, we also have
$\xi=\xi_\infty+ \gamma(\eta)$.

\begin{claim}\label{petitt lemmmma}
For each  integer $n \geq n_0$, it holds
\begin{align*}
  R(x_n(-1))=(\eta_n,\xi_n),\qquad R(x_\infty(-1))=(0,\xi_\infty),\qquad \xi_n=\xi_\infty+\gamma(\eta_n).
\end{align*}
In particular, $(\eta_n,\xi_n),(0,\xi_\infty)\in \Gamma_\infty$, while $(\xi_n,\eta_n),(\xi_\infty,0)\in \Gamma_\infty'$.
\end{claim}

\begin{proof}
  Let $n\geq n_0$.  We have
  $x_n(0)=\mathcal{F}(x_n(-1))\in \{r=0\}$, i.e.,
  $x_n(-1)\in \mathscr{A}_\infty$, and $R(x_n(-1))\in
  \Gamma_\infty$. It follows from Lemma~\ref{claim petitt} that
  $(\xi_n,\eta_n)=R(T(x_n(-1)))=\mathcal{G}(R(x_n(-1)))=\mathcal{I}_0(R(x_n(-1)))$,
  which gives the first identity.  Besides,
  $(\eta_n,\xi_n)\in \Gamma_\infty$, and
  $(\xi_n,\eta_n)=R (x_n(1))\in \Gamma_\infty'$, so that
  $\xi_n=\xi_\infty+\gamma(\eta_n)$.

  Similarly, $x_\infty(-1)\in \mathscr{A}_\infty$,
  $R(x_\infty(-1))\in \Gamma_\infty$, and by Lemma~\ref{claim petitt},
  we have
  \begin{align*}
    R(x_\infty(-1)) &= \mathcal{G}^{-1}(\xi_\infty,0)=\mathcal{I}_0(\xi_\infty,0)=(0,\xi_\infty).\qedhere
  \end{align*}
\end{proof}

Let us denote by $\mathcal{G}^\pm$ the coordinate functions of
$\mathcal{G}$, i.e.,
$\mathcal{G}\colon(\xi,\eta)\mapsto
(\mathcal{G}^+(\xi,\eta),\mathcal{G}^-(\xi,\eta))$.  Since
$\mathcal{G}(0,\xi_\infty)=(\xi_\infty,0)$, for any
$(\xi,\eta)\in \Omega_\infty$, we may write
\begin{align*}
  \mathcal{G}(\xi,\eta)=(\mathcal{G}^+(\xi,\eta),\mathcal{G}^-(\xi,\eta))=(\xi_\infty+%\alpha_{1,1}\xi+\alpha_{1,2}(\eta-\xi_\infty)+
  G^+(\xi,\eta-\xi_\infty),%\alpha_{2,1}\xi+\alpha_{2,2}(\eta-\xi_\infty)+
  G^-(\xi,\eta-\xi_\infty)),
\end{align*}
for two analytic functions
$G^\pm\colon (\xi,\eta)\mapsto \sum_{j+k \geq 1}G_{j,k}^\pm \xi^j\eta^k$.
Note that by the time reversal property~\eqref{renve temps}, for any
$(\xi,\eta)\in \mathcal{G}(\Omega_\infty)$, we have
\begin{equation}\label{cons temp rev}
\mathcal{G}^{-1}(\xi,\eta)=(G^-(\eta,\xi-\xi_\infty),\xi_\infty+G^+(\eta,\xi-\xi_\infty)).
\end{equation}
As a consequence of Lemma~\ref{claim petitt}, we get, for $|\eta|$
sufficiently small:
\begin{align}\label{lemma int fkl}
  G^-(\eta,\gamma(\eta))&=\eta,
  &G^+(\eta,\gamma(\eta))&=\gamma(\eta).
\end{align}
For $i=1,2$, we set
$G_i^\pm\colon \eta \mapsto \partial_i G^\pm(\eta,\gamma(\eta))$.
\begin{lemma}\label{lemma hidden sym}
  The following relations hold:
  \begin{align*}
    G_1^-&=-G_2^+=1-\gamma'G_2^-,\\
    G_1^+&=\gamma'(2-\gamma'G_2^-).
	\end{align*}
\end{lemma}
\begin{proof}
  By differentiating~\eqref{lemma int fkl}, for $|\eta|$ sufficiently
  small, we obtain
  \begin{align*}
    \partial_1 G^-(\eta,\gamma(\eta))+\gamma'(\eta)\partial_2 G^-(\eta,\gamma(\eta))&=G_1^-(\eta)+\gamma'(\eta)G_2^-(\eta)=1,\\
    \partial_1 G^+(\eta,\gamma(\eta))+\gamma'(\eta)\partial_2 G^+(\eta,\gamma(\eta))&=G_1^+(\eta)+\gamma'(\eta)G_2^+(\eta)=\gamma'(\eta).
  \end{align*}
  Now, by Lemma~\ref{claim petitt}, the differential
  $D\mathcal{G}_{(\eta,\xi_\infty+\gamma(\eta))}$ of the gluing
  map %at the periodic point $(\eta_n,\xi_n)=(\eta_n,\xi_\infty+\gamma(\eta_n))$ has to be in
  is in $\mathrm{SL}(2,\R)$. We have
  \begin{align*}
	D\mathcal{G}_{(\eta,\xi_\infty+\gamma(\eta))}
	&= \begin{pmatrix}
		\partial_1 G^+(\eta,\gamma(\eta)) & \partial_2 G^+(\eta,\gamma(\eta))\\
		\partial_1 G^-(\eta,\gamma(\eta))& \partial_2 G^-(\eta,\gamma(\eta))
	\end{pmatrix}
	=\begin{pmatrix}
	G_1^+(\eta) & G_2^+(\eta)\\
	G_1^-(\eta) & G_2^-(\eta)
      \end{pmatrix},
  \end{align*}
  and thus,
  \begin{align*}
    G_1^-(\eta) G_2^+(\eta)=G_1^+(\eta)G_2^-(\eta)-1.
  \end{align*}
  We deduce from the relations obtained previously that
  \begin{align*}
    (G_1^+(\eta)-\gamma'(\eta))(\gamma'(\eta)G_2^-(\eta)-1)&=\gamma'(\eta)G_1^-(\eta)G_2^+(\eta)\\
                                                           &=\gamma'(\eta)(G_1^+(\eta)G_2^-(\eta)-1),
  \end{align*}
  which yields
  \begin{align*}
	G_1^+=\gamma'(2-\gamma'G_2^-).
  \end{align*}
  Combining this with the relations obtained above, we conclude that
  $G_1^-=1-\gamma' G_2^-$ and
  $G_2^+=1-(\gamma')^{-1}G_1^+=\gamma'G_2^--1=-G_1^-$.
\end{proof}

The above lemma leads us to define the analytic function
\begin{align*}
  g\colon \eta\mapsto \sum_{k=0}^{+\infty} g_{k} \eta^k:=G_2^- (\eta)=\partial_2 G^-(\eta,\gamma(\eta)),
\end{align*}
so that for $|\eta|$ sufficiently small, we have
\begin{align}\label{expre differentielllle}
D\mathcal{G}_{(\eta,\xi_\infty+\gamma(\eta))}
= \begin{pmatrix}
\gamma'(\eta)(2-\gamma'(\eta)g(\eta)) & \gamma'(\eta)g(\eta)-1\\
1-\gamma'(\eta)g(\eta) & g(\eta)
\end{pmatrix}.
\end{align}

\begin{remark}\label{remark inv man}
    Let
  $\mathcal{W}_\infty^+\subset\Omega_\infty$ be the image in Birkhoff coordinates of the arc of the stable manifold of $(0,0)$ containing the homoclinic point $(0,\xi_\infty)$,
  and let $\mathcal{W}_\infty^-\subset \mathcal{G}(\Omega_\infty)$ be
  the arc of the unstable manifold of $(0,0)$ containing
  $(\xi_\infty,0)$. We write $\mathcal{W}_\infty^+=\{(\eta,\xi_\infty+w(\eta)):|\eta|\text{
  	small\,}\}$, for
  some analytic function $w$. By the time reversal symmetry,
  $T\circ R_+^{-1}(\mathcal{W}_\infty^+)=\mathcal{I}(T^{-1}\circ
  R_-^{-1}(\mathcal{W}_\infty^-))$, hence
	$$
	\mathcal{W}_\infty^-=\mathcal{I}_0(\mathcal{W}_\infty^+),
	$$ i.e., $\mathcal{W}_\infty^-=\{(\xi_\infty+w(\eta),\eta):|\eta|\text{ small\,}\}\subset\mathcal{G}(\Omega_\infty)$.
    By~\eqref{cons temp rev} and the fact that the gluing map
    $\mathcal{G}$ preserves the invariant subspaces of the saddle
    fixed point $(0,0)$, we can write analogous relations between
    $G^+$ and $G^-$, but involving the function $w$ instead of
    $\gamma$, i.e., for $|\eta|$ sufficiently small, it holds
\begin{equation*}
G^-(\eta,w(\eta))=0, \qquad G^+(0,\eta)=w(G^-(0,\eta)).
\end{equation*}
%which also gives
%$$
%G^-(G^-(0,\eta),G^+(0,\eta))=0.
%$$
Let us denote by $\sum_{j = 1}^{\infty} \gamma_j \xi^j$, resp. $\sum_{j = 1}^{\infty} w_j \xi^j$ the expansion of $\gamma$, resp. $w$.
Differentiating $G^-(\eta,w(\eta))=0$ and evaluating at $0$, we get $G^-_1(0)=-w_1 G^-_2(0)$. On the other hand, the identities in Lemma \ref{lemma hidden sym} yield $G_1^-(0)=1-\gamma_1 G_2^-(0)$. In particular, % it follows that
\begin{equation}\label{g zero gamma un}
g_0=G_2^-(0)=(\gamma_1-w_1)^{-1}.
\end{equation}
The arc $\mathcal{W}_\infty^+$ of the stable manifold is transverse to
the unstable manifold of $N$ at the homoclinic point $(0,\xi_\infty)$,
which is vertical in those coordinates, hence $w_1 \neq \infty$.
Besides, $\mathscr{A}_\infty \subset T^{-1}(\{r=0\})$ is the
image under $T^{-1}$ of some arc on the third scatterer, and then, its
image $\Gamma_\infty$ under $R$ is also transverse to the unstable
manifold of $N$ at $(0,\xi_\infty)$, i.e., $\gamma_1 \neq \infty$.
Since the gluing map
$\mathcal{G}=R\circ T \circ R^{-1}|_{\Omega_\infty}$ is defined
dynamically, we deduce that
\begin{align*} \Gamma_\infty'&=\mathcal{G}(\Gamma_\infty)=\mathcal{I}_0(\Gamma_\infty)=\{(\xi_\infty+\gamma(\eta),\eta):|\eta|\text{ small\,}\}\\
and\  \mathcal{W}_\infty^-&=\mathcal{G}(\{\xi=0\})=\mathcal{I}_0(\mathcal{W}_\infty^+)=\{(\xi_\infty+w(\eta),\eta):|\eta|\text{ small\,}\}
\end{align*} are still transverse at $(\xi_\infty,0)$, i.e., $\gamma_1\neq w_1$; as $\gamma_1,w_1\neq \infty$, and by \eqref{g zero gamma un}, we deduce
\begin{equation}\label{transversalite}
	|g_0|=|G_2^-(0)|=|\gamma_1-w_1|^{-1}\in (0,+\infty).%\label{linge 1} \\
	%G_{1,0}^-&=1-\gamma_1 G_{0,1}^-=-w_1 G_{0,1}^-.%\label{linge 2}\\
	%G_{0,1}^+&=-G_{1,0}^-=\gamma_1 G_{0,1}^- -1,\\%\label{linge 3}\\
	%G_{1,0}^+ &=\gamma_1 (1-G_{0,1}^+)=\gamma_1(2-\gamma_1 G_{0,1}^-).%\label{linge 4}
\end{equation}
\end{remark}

Let us recall that for any integer $n \geq n_0$, we let
$\Delta_n:=\Delta(\zeta_n)$, with $\zeta_n:=\xi_n\eta_n$. Let us also recall that by Lemma \ref{claim eta n delta n xi n}, it holds
\begin{equation}\label{eta n xi n bis}
	\eta_n=\Delta_n^{n} \xi_n.
\end{equation}
\begin{claim}\label{claim deuxis}
	For any integer $n \geq n_0$, we have
	\begin{equation}\label{delta n eta n xi n}
      \eta_n=\Delta(\eta_n(\xi_\infty+ \gamma(\eta_n)))^{n}(\xi_\infty+\gamma(\eta_n)).
	\end{equation}
\end{claim}
\begin{proof}
	Let $n \geq n_0$. By \eqref{eta n xi n bis}, $\eta_n=\Delta_n^n \xi_n$, and by definition,
    $\Delta_n=\Delta(\zeta_n)=\Delta(\xi_n\eta_n)$; moreover, by
    Lemma~\ref{petitt lemmmma}, we have
    $\xi_n=\xi_\infty+\gamma(\eta_n)$, which concludes the proof.
\end{proof}

In other words, Lemma~\ref{claim deuxis} tells us that for each
integer $n \geq n_0$, the coordinates
$(\xi_n,\eta_n)=(\xi_\infty+\gamma(\eta_n),\eta_n)$ of the image under
$R$ of the periodic point $x_n(1)$ are defined implicitely in terms of
the coefficients of $\Delta$ and $\gamma$, according to the previous
equation.

Let us now give the proof of Proposition \ref{prop bdskl} stated
earlier in the paper\footnote{ This proposition can be proved using
  several coordinate systems; in~\cite{BDSKL} it was proved using a
  linearization near the periodic point; here it will be proved using
  the Birkhoff Normal Form.}.
\begin{prop}\label{coro proximite points}
	For $n \geq n_0$, it holds
\begin{equation*}%\label{estimedjjt}
	\begin{array}{rcll}
		\|x_\infty(k)-x(k)\|&=&O(\lambda^{\frac{k}{2}}), &\text{for all } k \in \N,\nonumber\\
		\|x_n(k)-x_\infty(k)\|&=&O(\lambda^{n-\frac {k}{2}}), &\text{for all }
		k\in \{0,\cdots,n+1\}.
		%\|x_n(-k)-x_\infty(-k)\|&=&O(\lambda^{n-\frac k2}), &\text{for all }
		%n \geq n_0 \text{ and } k\in \{0,\cdots,n+1\}.
	\end{array}
\end{equation*}
\end{prop}
\begin{proof}
	The first point follows from the fact that for each $k \geq 0$,
	\begin{equation*}%\label
		R(x_\infty(2k+1))=(\lambda^k \xi_\infty,0),
	\end{equation*}
and when $k\gg 1$, $x_\infty(2k)-x(2k)= D\mathcal{F}_{x(2k-1)} (x_\infty(2k-1)-x(2k-1))+H.O.T.$

	Let us deal with the second point. As $n \to +\infty$, the symbolic codings of the points $x_n(1)$ and $x_\infty(1)$ match on longer and longer chunks, hence, by expansiveness, $\lim_{n \to +\infty}  x_n(1)=x_\infty(1)$. We deduce that $R(x_n(1))=R(\xi_n,\eta_n)\to_{n} R(x_\infty(1))$, i.e.,
	$$
	\lim_{n \to +\infty} \xi_n=\xi_\infty,\quad \lim_{n \to +\infty} \eta_n=0,
	$$
	and $\Delta_n=\Delta(\xi_n \eta_n)=\lambda+O(\xi_n \eta_n)=\lambda+o(1)$. Then, by Lemma \ref{claim eta n delta n xi n}, $\eta_n=\lambda^n \xi_\infty +o(\lambda^n)$, and thus, $\Delta_n=\Delta(\xi_n \eta_n)=\lambda+O(\lambda^n)$. By \eqref{eq new parmaetere}, we have
	\begin{equation*}%\label
		(\xi_{n}(2k+1),\eta_{n}(2k+1))=\xi_n(\Delta_{n}^{k}, \Delta_{n}^{n-k}),
	\end{equation*}
	for all $k \in \{0,\cdots,n\}$, hence, for $k\in \{0,\cdots,\frac{n+1}{2}\}$,
	\begin{equation*}%\label
		R(x_n(2k+1))-R(x_\infty(2k+1))=(\Delta_{n}^{k}\xi_n-\lambda^k \xi_\infty, \Delta_{n}^{n-k}\xi_n).
	\end{equation*}
We have seen that $|\Delta_n-\lambda|=O(\lambda^{n})$, and by Lemma~\ref{petitt lemmmma},
\begin{equation*}
	|\xi_n-\xi_\infty|=|\gamma(\eta_n)|=O(\lambda^{n}),
\end{equation*}
thus for any for $k\in \{0,\cdots,\frac{n+1}{2}\}$,
$|\Delta_n^{k}\xi_n-\lambda^k\xi_\infty|=O(n\lambda^n)=O(\lambda^{n-k})$, while
$ |\Delta_n^{n-k}\xi_n|%\right)\right\| =
=O(\lambda^{n-k})$, hence
$$
\|x_n(2k+1)-x_\infty(2k+1)\|=O(\|(\Delta_{n}^{k}\xi_n-\lambda^k \xi_\infty, \Delta_{n}^{n-k}\xi_n)\|)=O(\lambda^{n-k}).
$$
The estimate for even indices follows from the estimate for odd indices and the fact that $x_n(2k)-x_\infty(2k)= D\mathcal{F}_{x_\infty(2k-1)} (x_n(2k-1)-x_\infty(2k-1))+H.O.T.$
\end{proof}

\subsection{Lyapunov exponents and asymptotic expansions of the parameters}\label{subsection lyapunov exonent exp}

In this part, we use the same notation as in the previous subsection,
and show the relation between the above formulas and the Marked
Lyapunov Spectrum of the billiard table.

\begin{remark}\label{remark approche}
  Let us make a few comments and introduce some notation.
  \begin{enumerate}
  \item Given $\xi_\infty \in \R$ and the pair of functions
    $(\gamma,g)$, then by~\eqref{expre differentielllle}, it is
    possible to reconstruct the restriction of the gluing map
    $\mathcal{G}|_{\Gamma_\infty}$. Conversely, given
    $\xi_\infty \in \R$ and the coordinate functions
    $(\mathcal{G}^+,\mathcal{G}^-)$ of $\mathcal{G}$, then the
    function $\gamma$ can be recovered. Indeed, the gluing map
    $\mathcal{G}$ is dynamically defined, hence it maps some unstable
    cone at $(0,\xi_\infty)$ into some unstable cone at
    $(\xi_\infty,0)$. In particular, for $\xi$ small,
    $\eta=\xi_\infty+\gamma(\xi)$ is determined by the implicit
    equation
    \begin{align*}
      \mathcal{G}^+(\xi,\eta)\mathcal{G}^-(\xi,\eta)=\xi\eta.
    \end{align*}
  \item The homoclinic parameter $\xi_\infty \in \R$ can be regarded
    as a scaling factor: we will show in Subsection~\ref{sub esti
      leng} how its value is determined by the Marked Length Spectrum.
    For any integer $j \geq 0$, we introduce scaled coefficients
    \begin{equation}\label{scaled coefff}%{scaled eta n}
    \ba_j:=\lambda^{-1} a_i \xi_\infty^{2j},\quad
    \bg_j:=\gamma_j \xi_\infty^{j-1},\quad \text{and}\quad
    \bgg_j:=g_j \xi_\infty^j,
    \end{equation}
    with $a_0:=\lambda$ and $\gamma_0:=\xi_\infty$.  Note that
    $\ba_0=\bg_0=1$ and $\bgg_0=g_0$.

    In the following, for any integer $n \geq n_0$, we also let
    \begin{equation}\label{scaled eta n}
    \bar\eta_n:=(\xi_\infty \lambda^n)^{-1}\eta_n, %\quad %\zeta_n:=\xi_n\eta_n,%=\eta_n(\xi_\infty+\gamma(\eta_n)),
    \quad \text{and}\quad \bz_n:=(\xi_\infty^2\lambda^{n})^{-1}\zeta_n=(\xi_\infty^2\lambda^{n})^{-1} \xi_n \eta_n.
    \end{equation}
  \item Given the homoclinic parameter $\xi_\infty \in \R$, the
    Birkhoff Normal Form $N$ and the gluing map $\mathcal{G}$, we will
    find an explicit expression for the parameters $(\xi_n,\eta_n)_n$
    of the periodic orbits $(h_n)_{n\geq 0}$ and of their Lyapunov
    exponent. More precisely, under the assumption that the first
    Birkhoff invariant $a_1$ does not vanish, we show in Lemma~\ref{lemma exp lya detailll} and Corollary~\ref{coroll detm inv}
    that there is a one-to-one correspondence between the sequence of
    Lyapunov exponents $(\mathrm{LE}(h_n))_{n\geq 0}$ and the
    coefficients $(\ba_j, \bg_j, \bgg_j)_{j=0}^{\infty}$.
  \end{enumerate}
\end{remark}
By the above remark, if we know the value of $\xi_\infty$ and of the
scaled coefficients $\{\ba_j\}_{j}$, $\{\bg_j\}_{j}$ and
$\{\bgg_j\}_{j}$, then it is possible to reconstruct $\{a_j\}_{j}$,
$\{\gamma_j\}_{j}$ and $\{g_j\}_{j}$. In order to ease our notation,
we henceforth assume that $\xi_\infty=1$ in the rest of this section.

The next lemma tells us how the Lyapunov exponent of the associated
orbit can be expressed in terms of the new coordinates.

\begin{lemma}\label{lemme exposant de lya}
  For each integer $n \geq n_0$, we let
  \begin{align*}
  \Delta_n':=\Delta'(\zeta_n)\zeta_n=\sum_{k=1}^{+\infty} k a_k
    \zeta_n^k.
  \end{align*}
  Then, the Lyapunov exponent of the periodic orbit $h_n$ satisfies
  \begin{align*}
    2 \cosh(2(n+1)\mathrm{LE}(h_n))=\lambda^{-n}\mathrm{I}_n+\mathrm{II}_n+\lambda^{n}\mathrm{III}_n,
  \end{align*}
  where
  \begin{align*}
    \mathrm{I}_n&:=\lambda^{n}\Delta_n^{-n}\big(1-n\Delta'_n\Delta_n^{-1}\big)g(\eta_n),\\
    \mathrm{II}_n&:=2n\Delta'_n\Delta_n^{-1}\big(1-\gamma'(\eta_n)g(\eta_n)\big),\\
    \mathrm{III}_n&:=\lambda^{-n}\Delta_n^{n}\big(1+n\Delta'_n\Delta_n^{-1}\big)\gamma'(\eta_n) \big(2-\gamma'(\eta_n)g(\eta_n)\big).
  \end{align*}
    Let us recall that here, $\Delta_n=\Delta(\zeta_n)$, and $\Delta(0)=\lambda$.
\end{lemma}

\begin{proof}
  By the $(2n+2)$-periodicity of $h_n$, we have
  $T^{n+1}(x_n(-1))=x_n(-1)$, and since
  $D_{x_n(-1)} T^{n+1} \in \mathrm{SL}(2,\R)$, we obtain
  \begin{align*}
    &2 \cosh(2(n+1)\mathrm{LE}(h_n))=\mathrm{tr}(D T^{n+1}_{x_n(-1)})=\mathrm{tr}(D (T^{n}R^{-1}  \circ RT R^{-1}\circ R)_{x_n(-1)} )\\
    &=\mathrm{tr}(D (RT^{n}R^{-1})_{(\xi_n,\eta_n)} \cdot D(RTR^{-1})_{(\eta_n,\xi_n)})=\mathrm{tr}(D N^{n}_{(\xi_n,\eta_n)} \cdot  D\mathcal{G}_{(\eta_n,\xi_n)}).
  \end{align*}
  By Lemma~\ref{claim deuxis}, we have
  \begin{align*}
	DN^{n}_{(\xi_n,\eta_n)}
	&=\begin{pmatrix}\Delta_n^{n} & 0\\
      0 & \Delta_{n}^{-n}
	\end{pmatrix}+n\Delta'_n\Delta_n^{-1}\begin{pmatrix}
      \Delta_n^{n} & 1\\
      -1 &- \Delta_n^{-n}
    \end{pmatrix},
  \end{align*}
  with $\Delta_n':=\Delta'(\zeta_n)\zeta_n$, and then, it follows
  from~\eqref{expre differentielllle} that
  \begin{align*}
    &2 \cosh(2(n+1)\mathrm{LE}(h_n))\\
    &=\mathrm{tr} \left(\begin{pmatrix}\Delta_n^{n} & 0\\
        0 & \Delta_{n}^{-n}
      \end{pmatrix}\cdot \begin{pmatrix}
		G_1^+(\eta_n) & G_2^+(\eta_n)\\
		G_1^-(\eta_n) & G_2^-(\eta_n)
      \end{pmatrix}\right)\\
    &+n\Delta_n'\Delta_n^{-1}\mathrm{tr} \left(\begin{pmatrix}
        \Delta_n^{n} & 1\\
        -1 &- \Delta^{-n}
      \end{pmatrix}\cdot \begin{pmatrix}
        G_1^+(\eta_n) & G_2^+(\eta_n)\\
        G_1^-(\eta_n) & G_2^-(\eta_n)
      \end{pmatrix}\right)\\
    &=\Delta_n^{-n} g(\eta_n)+\Delta_n^{n}\gamma'(\eta_n) (2-\gamma'(\eta_n)g(\eta_n))\\
    &-n\Delta_n'\Delta_n^{-1}\big[\Delta_n^{-n} g(\eta_n)+2(\gamma'(\eta_n)g(\eta_n)-1)-\Delta_n^{n} \gamma'(\eta_n)(2-\gamma'(\eta_n)g(\eta_n))\big]\\
    &= \Delta_n^{-n}\big(1-n\Delta_n'\Delta_n^{-1}\big)g(\eta_n)+2n\Delta_n'\Delta_n^{-1}\big(1-\gamma'(\eta_n)g(\eta_n)\big)\\
    &+\Delta_n^{n}\big(1+n\Delta_n'\Delta_n^{-1}\big)\gamma'(\eta_n) (2-\gamma'(\eta_n)g(\eta_n)).\qedhere
  \end{align*}
\end{proof}
\begin{remark}\label{remarque prelimina}
	Our choice for the definitions of $\mathrm{I}_n$, $\mathrm{II}_n$, $\mathrm{III}_n$ will become clearer in the following.  Roughly speaking, we write them in  this way so that their expansions begin with the ``same weight'', i.e., are $0$-triangular  in the sense of Definition~\ref{triangulaire series}.
\end{remark}

In the following, as explained in Remark~\ref{remark approche}, we
derive asymptotic expansions with respect to $n$ of the parameters
$\eta_n$ and of the other symbols which appear in the expression of
the Lyapunov exponent $\mathrm{LE}(h_n)$ obtained in Lemma~\ref{lemme
  exposant de lya}. In Lemma~\ref{corollar dofe}, we compute the first
terms in these expansions. In Lemma~\ref{prop structure du dev}, we
study their general structure, and show that they can be expressed as
certain series mixing polynomials and exponentials in $n$, and whose
coefficients are ``homogeneous" combinations of the gluing terms and
of Birkhoff coefficients. In Lemma~\ref{lemme utile mochomoge}, we
compute the value of the coefficients of the different terms in the
expansions of the parameters; each time, we focus on the terms with
the largest index, as we see them for the first time, while the
previous terms appear as additive constants.  These
estimates will later be used (see \eg~Lemma \ref{lemma exp lya
  detailll}) to show that the expression involving the Lyapunov
exponents $\mathrm{LE}(h_n)$ in Lemma~\ref{lemme exposant de lya}
admits an asymptotic expansion as $n$ goes to $+\infty$, of the form
$\sum_{q,p \geq 0} L_{q,p} n^q (\lambda^n)^p$, where the coefficients
$(L_{q,p})_{q,p}$ depend on the constants $\{\bar a_i\}_{i \geq 1}$
and $\{\bar \gamma_i\}_{i \geq 1}$. In particular, as
$\lambda\in (0,1)$, each of these terms decays at a different speed
with respect to $n$, hence the coefficients $(L_{q,p})_{q,p}$ are
determined by the collection of Lyapunov exponents
$(\mathrm{LE}(h_n))_{n\geq 1}$.

Due to the different roles that the various coefficients
$\{\bar a_i\}_{i \geq 1}$, $\{\bar \gamma_i\}_{i \geq 1}$ of $\Delta$
and $\gamma$ play (see \eg~the formula given by Lemma~\ref{claim
  deuxis}), we hope to be able to distinguish between them in the
estimates; in particular, we expect Birkhoff coefficients to ``weigh
more'' than the gluing terms, since the periodic orbits $h_n$ spend
much more time in a neighborhood of the saddle than in the gluing
region. The way we actually show this is by carefully analyzing the
dependence of the quantities $(L_{q,p})_{q,p}$ on the coefficients
$\{\bar a_i\}_{i\geq 1}$, $\{\bar \gamma_i\}_{i \geq 1}$: in a first
time, we compute inductively the expansion of $\eta_n$ in terms of
$\{\bar a_i\}_{i \geq 1}$, $\{\bar \gamma_i\}_{i \geq 1}$ thanks to
the formula given by Lemma~\ref{claim deuxis}.  Next, we compute the
expansions of the other expressions which appear in the formula given
by Lemma~\ref{lemme exposant de lya}. Finally, we show that, provided
a certain non-degeneracy condition is satisfied, the system of
equations given by the coefficients $(L_{q,p})_{q,p} $ can be
inverted, i.e., the coefficients $\{\bar a_i\}_{i \geq 1}$ and
$\{\bar \gamma_i\}_{i \geq 1}$ can be reconstructed from these data
(see \eg~Corollary \ref{coroll detm inv} for more details).

\begin{lemma}\label{corollar dofe}
  With the notation introduced in~\eqref{scaled coefff}--\eqref{scaled
    eta n}, it holds:
  \begin{align*}
    \bar\eta_n&=1+[n\ba_1+\bg_1]
                \lambda^{n}+\Big[n^2\frac{3\ba_1^2}{2}+n\big(-\frac{\ba_1^2}{2}+4\ba_1
                \bg_1+\ba_2\big)+(\bg_1^2+\bg_2) \Big]\lambda^{2n}  \\
    &\phantom = +O(n^3\lambda^{3n}).
  \end{align*}
  By the fact that $\Delta_n=\Delta(\zeta_n)$ and
  $\Delta_n':=\Delta'(\zeta_n)\zeta_n$, the previous estimates give
  \begin{align*}
    \lambda^{n}\Delta_n^{-n}&=
                              1-n\ba_1 \lambda^{n}-\Big[ n^2\frac{\ba_1^2}{2}+ n\big(2\ba_1\bg_1+\ba_2-\frac{\ba_1^2}{2}\big)\Big]\lambda^{2n}+O(n^3 \lambda^{3n}),\\
    1-n\Delta'_n\Delta_n^{-1}
                            &=1-n\ba_1\lambda^{n}-[n^2\ba_1^2+n(2
                              \ba_1 \bg_1+2\ba_2-\ba_1^2)]\lambda^{2n}+O(n^3\lambda^{3n}),
  \end{align*}
  and
  \begin{align*}
	\lambda^{n}\Delta_n^{-n}\big(1-n\Delta'_n\Delta_n^{-1}\big)=1-2n\ba_1\lambda^{n}-\Big[n^2\frac{\ba_1^2}{2}+n\big(4 \ba_1 \bg_1+3\ba_2-\frac{3\ba_1^2}{2}\big)\Big]\lambda^{2n}+O(n^3\lambda^{3n}).
  \end{align*}
  In particular,
  \begin{align}\label{eq intermedi}
    2 \cosh(2(n+1)\mathrm{LE}(h_n))=\lambda^{-n}g_0-2ng_0 \ba_1 +O(1),
  \end{align}
  hence the coefficients $g_0$ and $\ba_1$ are determined by the
  Marked Length Spectrum.
\end{lemma}
\begin{proof}
  Let $n \geq n_0$. By Lemma~\ref{claim deuxis}, and since we assume
  that $\xi_\infty=1$, we have
  \begin{align*}
    % \eta_n&=
    &\eta_n=\Delta(\eta_n(1+ \gamma(\eta_n)))^{n}(1+\gamma(\eta_n))\\
            &=\Big(\sum_{j=0}^3 a_j \eta_n^j\cdot\Big(\sum_{k=0}^3 \gamma_k \eta_n^k\Big)^j+O(\eta_n^4)\Big)^{n}\cdot\Big(\sum_{\ell=0}^3 \gamma_\ell \eta_n^\ell+O(\eta_n^4)\Big),
  \end{align*}
  where recall that $\gamma_0=\xi_{\infty} = 1$; this yields the
  expansion
  \begin{align*}
    \eta_n&=\lambda^{n}\Big(1+[n\lambda^{-1}a_1 +\gamma_1] \eta_n+\Big[n^2\frac{(\lambda^{-1}a_1 )^2}{2}\\
          &+n\Big(-\frac{(\lambda^{-1}a_1 )^2}{2}+2\lambda^{-1}a_1 \gamma_1+\lambda^{-1} a_2 \Big)+\gamma_2 \Big]\eta_n^2+
            O(n^3\eta_n^3)\Big).
  \end{align*}

  By considering first order terms, we obtain
  $\eta_n=\lambda^{n} +O(n\lambda^{2n})$.  Plugging this back into the
  previous equation, we deduce that
  \begin{align*}
    \eta_n&=\lambda^{n}+[n\ba_1+\bg_1] \lambda^{2n}+O(n^2\lambda^{3n}).
  \end{align*}
  We thus obtain
  \begin{align*}
    &\bar\eta_n=\lambda^{-n}\Big(\lambda^{n}+[n\ba_1+\bg_1] \lambda^{2n}+\Big[(n\ba_1+\bg_1)^2+n^2\frac{\ba_1^2}{2}+\\
    &+n\Big(-\frac{\ba_1^2}{2}+2 \ba_1\bg_1+\ba_2 \Big)+\bg_2 \Big]\lambda^{3n}+O(n^3\lambda^{4n})\Big)\\
    &= 1+[n\ba_1+\bg_1] \lambda^{n}+\Big[n^2\frac{3\ba_1^2}{2}+n\Big(-\frac{\ba_1^2}{2}+4\ba_1 \bg_1+\ba_2\Big)+(\bg_1^2+\bg_2) \Big]\lambda^{2n}+O(n^3\lambda^{3n}).
  \end{align*}
  To obtain the expansions of $\Delta_n^{\pm n}$, we argue as follows:
  by definition, we have $\Delta_n=\Delta(\zeta_n)$, with
  $\zeta_n=\eta_n(1+\gamma(\eta_n)) %Let $\bar{\eta}_{n}:=(\xi_\infty \lambda^{n})^{-1} \eta_n$.
  =\lambda^{n} \bz_n$ as in~\eqref{scaled eta n}, so that
  \begin{align*}
    &\bz_n=\bar{\eta}_{n}+\bg_1 \lambda^{n}\bar{\eta}_{n}^2+\bg_2 \lambda^{2n} \bar{\eta}_{n}^3+O(n^3 \lambda^{3n})\\
    &=  1+[n\ba_1+2 \bg_1]\lambda^{n}+\Big[n^2 \frac{3\ba_1^2}{2}+n\Big(-\frac{\ba_1^2}{2}+6 \ba_1 \bg_1 +\ba_2\Big)+(3\bg_1^2+2 \bg_2)\Big] \lambda^{2n}+O(n^3\lambda^{3n}).
  \end{align*}
  To conclude, it suffices to expand the following expressions:
  \begin{equation*}
    % (\Delta(\zeta_n))^{n-1}&=\lambda^{n-1}\big(1+\ba_1 \lambda^{n-1}\bz_n+a_2  \lambda^{2(n-1)}\bz_n^2+ \ba_3 \lambda^{3(n-1)}\bz_n^3\big)^{n-1}+O(n^4 \lambda^{5(n-1)}) \\
    \Delta_n^{-n}=(\Delta^{-1}(\zeta_n))^{n}=\lambda^{-n}\big(1-\ba_1  \lambda^{n} \bz_n+(\ba_1^2-\ba_2) \lambda^{2n}\bz_n^2%+(-\ba_1^3+2 \ba_1 \ba_2-\ba_3) \lambda^{3n}\bz_n^3
    \big)^{n}+O(n^3 \lambda^{2n}),
  \end{equation*}
  and
  \begin{align*}
	1-n\Delta'_n\Delta_n^{-1}&=1-n\Delta'(\zeta_n)\zeta_n\cdot\Delta^{-1}(\zeta_n)\\
                             &=1-n \lambda^{n}\bz_n\big( \ba_1+2 \ba_2\lambda^{n}\bz_n\big)\cdot\big(1-\ba_1  \lambda^{n} \bz_n\big)+O(n^3 \lambda^{3n})\\
                             &=1-n\lambda^{n}\big(\ba_1\bz_n+(2\ba_2-\ba_1^2)\lambda^{n}\bz_n^2\big)+O(n^3 \lambda^{3n}).
  \end{align*}

  The previous estimates and the expression obtained in Lemma~\ref{lemme exposant de lya} yield
  \begin{align*}
    2 \cosh(2(n+1)\mathrm{LE}(h_n))=\lambda^{-n}g_0-2n g_0 \ba_1 +O(1).
  \end{align*}
  Indeed, the other expressions in the formula given by Lemma~\ref{lemme exposant de lya} are bounded, since
  $\Delta_n^{\pm n}=O(\lambda^{\pm n})$, while
  $\eta_n=O(\lambda^{n})$.%, and thus, $\Delta_n^{-(n-1)}\cdot
                          %g^-(\eta_n)=G^-_{0,1}\Delta_n^{-(n-1)}+O(1)$,
                          %while $\Delta_n^{n-1}\cdot
                          %g^+(\eta_n)=O(\lambda^{n-1})$.

  By Theorem~\ref{corollaire lyapu}, the quantities on the left hand
  side can be computed, thus we can recover the value of $g_0$ and
  $\ba_1$ by separating terms growing at different speeds:
  \begin{align*}
    g_0&=\lim_{n \to +\infty} 2 \cosh(2(n+1)\mathrm{LE}(h_n)) \lambda^{n},\\
    \ba_1&=\lim_{n \to +\infty} \frac{1}{2n}(\lambda^{-n}- 2g_0^{-1} \cosh(2(n+1)\mathrm{LE}(h_n))).
  \end{align*}
  Indeed, recall that by Remark~\ref{remark inv man}, the coefficient $g_0$ does not vanish.
  % Let us also recall that by Lemma~\ref{petitt lemmmma}, we know that $\gamma_1=\alpha$.
\end{proof}

More generally, we prove the following lemma.
\begin{lemma}\label{prop structure du dev}
  There exists $n_0 > 0$ so that for any integer $i \geq 1$, there
  exists a sequence $\left(P_k^{(i)}\right)_{k \geq 1}$ of polynomials
  such that for any integer $n \geq n_0$:
  \begin{align*}%\label{lemme int sans int}
    \bar{\eta}_n^i &= 1+\sum_{k=1}^{+\infty} P_k^{(i)}(n) \lambda^{nk},%\big),
  \end{align*}
  where for each $k \geq 1$, the polynomial
  $P_k^{(i)}(X)=\sum_{j=0}^k \mu_{j,k}^{(i)} X^j$ has degree
  $k$. %, with $\mu_{0,0}^{(i)}:=1$. %and $\mu_{1,0}^{(i)}:=0$, and f
  For simplicity, we abbreviate $P_k^{(1)}=P_k$ and
  $\mu_{j,k}^{(1)}=\mu_{j,k}$ in the following.\footnote{ Note that it is
    sufficient to show the result for $i=1$, as such expansions are
    stable by taking powers. This is what we are going to do in the
    following proof. }

  Similarly, there exist three sequences $(Q_k^\pm)_{k \geq 1}$,
  $(R_k)_{k\geq 1}$ of polynomials such that
  \begin{align*}
    \lambda^{\mp n}\Delta_n^{\pm n}&=1+\sum_{k=1}^{+\infty} Q_k^\pm (n) \lambda^{nk},\\
    1-n\Delta'_n\Delta_n^{-1}&=1+\sum_{k=1}^{+\infty} R_k(n) \lambda^{nk},
  \end{align*}
  where for each $k \geq 1$, the polynomials $Q_k^\pm(X)=\sum_{j=0}^k \nu_{j,k}^\pm X^j$ and  $R_k(X)=\sum_{j=0}^k \rho_{j,k}^\pm X^j$ have degree $k$.

  In particular, by Lemma~\ref{corollar dofe}, it holds
  \begin{align}\label{puissances de mu}
    \left\{
    \begin{array}{lllllll}
      \mu_{0,0}^{(i)}=1, & & \mu_{1,1}^{(i)}=i \ba_1, &  & \mu_{2,2}^{(i)}=\frac{i(i+2)}{2}\ba_1^2, \\
      % &  & \mu_{3,3}^{(i)}= \frac{i(i+3)^2}{6}\ba_1^3,\\
      \nu_{0,0}^-=1, & & \nu_{1,1}^-=-\ba_1, & & \nu_{2,2}^-=-\frac{\ba_1^2}{2},\\ %& & \nu_{3,3}^-=-\frac{2\ba_1^3}{3},\\
      % \nu_{0,0}^+=1, & & \nu_{1,1}^+=\ba_1 & & \nu_{2,2}^+=\frac{3\ba_1^2}{2}, & & \nu_{3,3}^+=\frac{8\ba_1^3}{3}.
                                                                                     \rho_{0,0}=1, & & \rho_{1,1}=- \ba_1, & & \rho_{2,2}=-\ba_1^2. %& & \nu_{3,3}^-=-\frac{2\ba_1^3}{3},\\
    \end{array}
  \right.
\end{align}
\end{lemma}
\begin{proof}
   Let us first consider
  $\bar{\eta}_n$,
  $n \geq
  n_0$.
  We will prove by induction on $\ell \geq
  0$ that $\bar{\eta}_n=\bar{\eta}_{n,\ell}+O
  \big(n^{\ell+1}\lambda^{n(\ell+1)}\big)$, where
  \begin{equation}\label{recurrence}
    \bar{\eta}_{n,\ell}= 1+\sum_{k=1}^{\ell} P_k(n)\lambda^{nk},%+O \big(n^{\ell+1}\lambda^{(\ell+1)(n-1)}\big),
  \end{equation}
  for certain polynomials
  $P_1,P_2,\cdots,P_{\ell}$ satisfying the above properties.

  It is clear for $\ell=0$. Assume that it holds for $\ell-1 \geq
  0$.  To show the result for
  $\ell$, we use the formula given by Lemma~\ref{claim deuxis}:
  \begin{align}
\bar\eta_n&= \lambda^{-n}\Delta(\eta_n(1+ \gamma(\eta_n)))^{n}(1+\gamma(\eta_n))\nonumber \\
&=\Big(1+\sum_{p=1}^{\ell}\lambda^{-1} a_p \eta_n^p (1+ \gamma(\eta_n))^p+O(\eta_n^{\ell+1})\Big)^{n} (1+\gamma(\eta_n)) \nonumber \\
&=\sum_{r=0}^{n} \begin{pmatrix}
n\\
r
\end{pmatrix}
\bigg(\sum_{p=1}^{\ell}\lambda^{-1} a_p \eta_n^p \Big(1+ \sum_{q=1}^{\ell} \gamma_q  \eta_n^q\Big)^p+O(\eta_n^{\ell+1})\bigg)^{r} \Big(1+ \sum_{s=1}^{\ell} \gamma_s \eta_n^s+O(\eta_n^{\ell+1})\Big) \nonumber \\
&=\sum_{r=0}^{\ell} \begin{pmatrix}
n\\
r
\end{pmatrix}
\bigg(\sum_{p=1}^{\ell}\ba_p\lambda^{np} \bar{\eta}_{n,\ell-1}^p \Big(1+ \sum_{q=1}^{\ell} \bg_q\lambda^{nq}\bar{\eta}_{n,\ell-1}^q\Big)^p\bigg)^{r}  \Big(1+ \sum_{s=1}^{\ell} \bg_s \lambda^{ns}\bar{\eta}_{n,\ell-1}^s\Big)\label{eq au dessus} \\
&+O\big(n^{\ell+1}\lambda^{n(\ell+1)}\big). \nonumber
\end{align}
Indeed, it is sufficient to consider the $(\ell-1)$-expansion
$\bar{\eta}_{n,\ell-1}=1+\sum_{k=1}^{\ell-1} P_k(n)\lambda^{nk}$ of
$\bar\eta_n$ obtained previously to go from $\ell-1$ to $\ell$, as the
summations indices $p,q,s$ are all at least equal to one, and hence
each term $\bar\eta_n^*$ in the above expression is multiplied by a
factor $\lambda^{nk}$, with $k \geq 1$. Moreover, we can restrict
ourselves to indices $r,p,q,s \in \{0,\cdots,\ell\}$, since for
$r,p,q,s \geq \ell+1$, the associated terms are of order
$O\big(n^{\ell+1}\lambda^{n(\ell+1)}\big)$.

We claim that the degree of the polynomial in $n$ associated to the
factor $\lambda^{n\ell}$ is at most $\ell$. Indeed, the expansion of
the previous expression is a combination of powers of
$\bar{\eta}_{n,\ell-1}$ (which are themselves combinations of
polynomials in $n$ multiplied by powers of $\lambda^n$, where the
degree of the polynomial is at most equal to the exponent of
$\lambda^n$) multiplied by binomial coefficients $ \begin{pmatrix}
  n\\
  *
\end{pmatrix}$ and powers of $\ba_* \lambda^{n*}$ or
$\bg_* \lambda^{n*}$. Besides the degree of the polynomial in $n$
associated to binomial coefficients is always less than or equal to
the exponent of $\lambda^n$ (for each $r \geq 1$, the coefficient
$\begin{pmatrix}
  n\\
  r
\end{pmatrix}$ gives a polynomial of degree $r$ in $n$, while the
second factor $\Big(\sum_{p=1}^{\ell}\ba_p\lambda^{np}\dots\Big)^{r}$
in formula \eqref{eq au dessus} above already contributes to a power
of $\lambda^n$ with an exponent $rp \geq r$).

% Moreover,  if we plug the expression~\eqref{recurrence} given by the induction hypothesis into this formula, then it is clear
We conclude that the new expansion will be of the same form as before, i.e., for some polynomial $P_{\ell}$ of degree at most $\ell$, we have
\begin{equation*}
\bar\eta_n=1+\sum_{k=1}^{\ell} P_k(n)\lambda^{nk}+O \big(n^{\ell+1}\lambda^{n(\ell+1)}\big).%,\qquad \bar\eta_{n,\ell}= 1+\sum_{k=1}^{\ell} P_k(n)\lambda^{kn}.
\end{equation*}

Let us now consider the expansion of $\Delta_n^{\pm n}$. %, we proceed in the same way. Let us
We first remark that $\Delta(z)^{\pm}=\lambda^{\pm} + \sum_{k=1}^{+\infty} a_k^\pm z^k$, where for each $k \geq 1$, $a_k^+=a_k$, and
\begin{equation*}
a_k^-=-\lambda^{-2} a_k-\lambda^{-1}(a_1 a_{k-1}^-+\hdots+ a_{k-1} a_1^-).
\end{equation*}
Similarly, for   $i \geq 1$, let $\ba_i^\pm:=\lambda^{\mp} a_i^\pm \xi_\infty^{2i}=\lambda^{\mp} a_i^\pm $.
As a result, for each $k \geq 0$, it holds
\begin{equation}\label{forme a tilde}
\ba_k^-=-\ba_k-(\ba_1 \ba_{k-1}^-+\hdots+ \ba_{k-1} \ba_1^-).
\end{equation}
Thus, for any integer $\ell\geq 0$, we obtain
\begin{align*}
&\Delta_n^{\pm n}=(\Delta(\eta_n(1+ \gamma(\eta_n)))^{\pm 1})^{n}\\
&=\lambda^{\pm n}\Big(1+\sum_{p=1}^{\ell}\lambda^{\mp 1} a_p^\pm \eta_n^p (1+ \gamma(\eta_n))^p+O(\eta_n^{\ell+1})\Big)^{n}\\
&= \lambda^{\pm n}\bigg(1+\sum_{r=1}^{\ell} \begin{pmatrix}
n\\
r
\end{pmatrix}
  \Big(\sum_{p=1}^{\ell}\ba_p^\pm \lambda^{ np} \bar{\eta}_{n}^p \Big(1+ \sum_{q=1}^{\ell} \bg_q  \lambda^{nq}\bar{\eta}_{n}^q\Big)^p\Big)^{r}\bigg)+O(n^{\ell+1}\lambda^{n(\ell+1\pm 1)}).%\label{ligne test}
\end{align*}
The form of the expansions of $\Delta_n^{\pm n}$ and
$1-n\Delta'_n\Delta_n^{-1}$ follows from the expression of
$\bar\eta_{n}$ obtained previously, since $\Delta_n=\Delta(\zeta_n)$
and $\Delta_n'=\Delta'(\zeta_n)\zeta_n$, with
$\zeta_n=\lambda^n \bar\eta_n(1+\gamma(\lambda^n \bar \eta_n))$.
\end{proof}
\begin{remark}\label{rem:balanced-series}
  On a formal level, we see that $\bar\eta_{n}$,
  $\lambda^{\mp n}\Delta_{n}^{\pm n}$ and
  $1-n\Delta_{n}'\Delta_{n}^{-1}$ can be expressed as (formal) series
  in $\lambda^{n}$ with coefficients in the ring of polynomials in
  $n$.  Moreover the coefficient of order $k$ is a polynomial of
  degree $k$.  Let us call \emph{balanced} those formal series with
  coefficients in the ring of polynomials in $n$ with the property
  that the coefficient of order $k$ is a polynomial of degree \emph{at
    most} $k$.  Observe that such series are closed under sum and
  product; moreover they are also closed under composition with an
  analytic function.  We conclude that the quantities $\I_n$,
  $\II_{n}$, $\III_{n}$ introduced in Lemma \ref{lemme exposant de
    lya} are also balanced series.  Let us also note that
  expansions of a similar type were studied earlier in the paper
  \cite{FY} for a different purpose.
\end{remark}

\begin{lemma}\label{rem:coefficients-homogeneous}
  For any integers $k \geq 1$, $j \in \{0,\cdots,k\}$, and
  $\ell\geq 1$, the coefficients
  $\mu_{j,k}^{(\ell)},\nu_{j,k}^\pm,\rho_{j,k}$ are ``homogeneous"
  expressions in the parameters $\{\ba_i\}_i$, $\{\bg_i\}_i$:
	\begin{align*}
      *_{j,k} &= *_{j,k}(\ba_1,\ba_2,\cdots,\ba_{k-j+1},\bg_1,\bg_2,\cdots,\bg_{k-j}),& *_{j,k}&=\mu_{j,k}^{(\ell)},\nu_{j,k}^\pm,\rho_{j,k}, %\\
	\end{align*}
where $*_{j,k}$ is a linear combination of terms of the form
\begin{equation}\label{detail coefficientss}
  \ba_{1}^{p_1}\ba_{2}^{p_2}\cdots \ba_{k-j+1}^{p_{k-j+1}} \bg_{1}^{q_1}\bg_2^{q_2}\cdots \bg_{k-j}^{q_{k-j}},
\end{equation}
with
$$
\mathrm{i)}\ \sum_{i=1}^{k-j} p_i\geq j,\qquad \mathrm{ii)}\ \sum_{i=1}^{k-j+1} ip_i +\sum_{i=1}^{k-j} iq_i=k.
  %\ba_{p_1}^{\ell_1}\ba_{p_2}^{\ell_2}\cdots \ba_{p_k}^{\ell_k} \bg_{q_1}^{m_1},\qquad \sum_{i=0}^k p_i=k
$$
	%where  as above,  we let $\ba_i:=\lambda^{-1}a_i \xi_\infty^{2i}$, and  $\bg_i:=\gamma_i \xi_\infty^{i-1}$, for all integer $i \geq 1$.
\end{lemma}

\begin{proof}
Let us  study how the coefficients $\{\mu_{j,k}\}_{j,k}$ in the expansion of $\bar \eta_n$  depend on the parameters $\{\ba_i\}_i$, $\{\bg_i\}_i$. The ``homogeneous" structure of the expansion of  the coefficients $\{\mu_{j,k}^{(\ell)}\}_{j,k}$,  $\{\nu_{j,k}^\pm\}_{j,k}$ and $\{\rho_{j,k}\}_{j,k}$ is shown in the same way as for the coefficients $\{\mu_{j,k}\}_{j,k}$.

For any integers $\ell\geq 1$ and $n \geq n_0$, recall the equation for $\bar\eta_n$ obtained in the proof of Lemma~\ref{prop structure du dev}:
\begin{equation*}
	\bar\eta_n=\sum_{r=0}^{\ell} \begin{pmatrix}
		n\\
		r
	\end{pmatrix}
	\bigg(\sum_{p=1}^{\ell}\ba_p\lambda^{np} \bar{\eta}_{n}^p \Big(1+ \sum_{q=1}^{\ell} \bg_q\lambda^{nq}\bar{\eta}_{n}^q\Big)^p\bigg)^{r}\\ %(\xi_\infty+\gamma(\eta_n)).
	 \Big(1+ \sum_{s=1}^{\ell} \bg_s \lambda^{ns}\bar{\eta}_{n}^s\Big)+O\big(n^{\ell+1}\lambda^{n(\ell+2)}\big).
\end{equation*}
The  expansion of this expression is a   combination of terms as in~\eqref{detail coefficientss}.  In particular, for any integer $\ell \geq 1$,
%it follows from the previous expansion
we see that   $\ba_\ell$ first appears with the weight $n \lambda^{n\ell }$ (for $r=1$ and $p=\ell$ with the above notation), while $\bg_\ell$ first appears with the weight $\lambda^{n\ell }$ (for $r=0$ and $s=\ell$ with the  above notation).

More generally, the ``homogeneity" property $\mathrm{ii)}$ above  is essentially due to the fact that in the previous  expansion, $\ba_p$ always comes together with the weight $\lambda^{np}$, while $\bg_q$ always comes together with the weight $\lambda^{nq}$.
\end{proof}

\begin{remark}
	Note that increasing the exponent of $n$ in the expansion of
	$\bar\eta_n$ corresponds to taking derivatives of the function
	$\Delta$ in formula~\eqref{delta n eta n xi n}. In terms of the above
	expansion, those derivatives are associated to certain binomial
	coefficients, and each time we increase the exponent of $n$ by one, the exponent of $\lambda^n$ is increased by at least one too, depending on
	the weight of the coefficient $\ba_p$ associated to this
	derivative. Together with the previous remark on the first appearance
	of $\ba_\ell$, $\bg_\ell$, this explains the constraint in
	\eqref{detail coefficientss} and $\mathrm{ii)}
	$ on the coefficients
	which can enter the expression associated to a specific weight, and
	why they depend on the difference $k-j$ between the exponent $j$ of
	$n^j$ and the exponent $k$ of $\lambda^{nk}$ (the coefficients
	$\ba_{k-j+1}^{p_{k-j+1}}$ and $\bg_{k-j}^{q_{k-j}}$ are obtained when
	all the derivatives we take are associated to $\ba_1$).

	Besides, the reason why we have an inequality and not an equality   in point $\mathrm{i)}$  is because  the binomial coefficients $\begin{pmatrix}
	n\\
	r
	\end{pmatrix}$ are not homogeneous polynomials in $n$.
\end{remark}

\begin{remark}
  The reason why the respective weights of $\ba_\ell$ and $\bg_\ell$
  on their first appearance differ by a factor $n$ is due to the fact
  that any orbit under consideration spends much more time ($n$ steps)
  in the region where we have Birkhoff coordinates, while the gluing
  term associated to the coefficient $\bg_\ell$ accounts for a bounded
  number of steps in the orbit (i.e., the number of such steps does not depend on $n$).
\end{remark}

\begin{lemma}\label{lemme utile mochomoge}

With the notation introduced in Lemma~\ref{prop structure du dev}, for every $k \geq 1$, there exist constants  $c_{0,k},c_{j,k+j},c_{j,k+j}^\pm,c_{j,k+j}'\in \R$, $j =1,2$, with %$d_{1,k}^a,d_{0,k}^a,d_{1,k}^b,d_{0,k}^b  \in \R$, with
\begin{align*}
c_{0,k}&=c_{0,k}(\bg_1,\bg_2,\cdots,\bg_{k-1}),\\
*_{j,k+j}&=*_{j,k+j}(\ba_1,\ba_2,\cdots,\ba_{k} ,\bg_1,\bg_2,\cdots,\bg_{k-1}), & *=c,c^\pm,c', %,\ j=1,2,
%c_{j,k+j}^\pm&=c_{j,k+j}^\pm(\ba_1,\ba_2,\cdots,\ba_{k},\bg_1,\bg_2,\cdots,\bg_{k-1}),&j=1,2,\\%\quad *=a,b,
%c_{j,k+j}'&=c_{j,k+j}'(\ba_1,\ba_2,\cdots,\ba_{k},\bg_1,\bg_2,\cdots,\bg_{k-1}),&j=1,2,
\end{align*}
such that
\begin{subequations}\label{eq:mu-nu-rho}
  \begin{align}\label{eq:mu-rels}
    &\left\{
    \begin{array}{rcl}
      \mu_{0,k}&=& \bg_{k}+c_{0,k},\\
      \mu_{1,k+1}&=&(k+3)\ba_1\cdot \bg_k +\ba_{k+1}+c_{1,k+1},\\
      \mu_{2,k+2}&=&\frac{(k+3)(k+5)}{2}\ba_1^2\cdot \bg_{k}+ (k+3)\ba_1\cdot \ba_{k+1}+c_{2,k+2};
    \end{array}\right.\\%}
    \label{eq:nu-rels}
    &\left\{
    \begin{array}{rcl}
      \nu_{0,k}^\pm&=&0,\\
      \nu_{1,k+1}^\pm&=&\pm(2\ba_1\cdot \bg_{k}+\ba_{k+1})+c_{1,k+1}^\pm,\\
      \nu_{2,k+2}^\pm&=&\pm (k+2\pm 1)\ba_1(2\ba_1\cdot \bg_{k}+\ba_{k+1})+c_{2,k+2}^\pm;
    \end{array}\right.\\%}
    \label{eq:rho-rels}
    &\left\{
    \begin{array}{rcl}
      \rho_{0,k}&=&0,\\
      \rho_{1,k+1}&=&-2 \ba_1 \cdot \bg_{k}-(k+1)\ba_{k+1}+c_{1,k+1}',\\
      \rho_{2,k+2}&=&-2 (k+2)\ba_1^2 \cdot \bg_{k}-((k+1)^2+1) \ba_1\cdot\ba_{k+1}+c_{2,k+2}'.
    \end{array}\right.%}
  \end{align}
\end{subequations}
\end{lemma}

\begin{remark}\label{remark partie technique}
  Before giving the details of the proof, let us explain how the
  computations are carried out. We first focus on $\bar{\eta}_n$ and
  study the coefficients of the expansion given by Lemma~\ref{prop
    structure du dev}; the expressions of
  $\lambda^{\mp n}\Delta_n^{\pm n}$ and $1-n\Delta'_n\Delta_n^{-1}$
  follow from that of $\bar{\eta}_n$, as they are obtained by
  evaluating the functions $\Delta,\Delta',\gamma,\dots$ at the point
  $\bar{\eta}_n$.
  Note that the equation in Lemma~\ref{claim deuxis} can be rewritten
  as
	\begin{equation}\label{new claim 5.6}
	\bar{\eta}_n=\bar{\Delta}(\lambda^n \bar{\eta}_n \bg(\lambda^n \bar{\eta}_n))^n \bg(\lambda^n\bar{\eta}_n),
	\end{equation}
	denoting
	$$
	\bar{\Delta}\colon z \mapsto 1+\sum_{j=1}^{+\infty} \ba_j z^j,\quad\text{and} \quad \bar \gamma\colon z \mapsto 1 +\gamma(z)=1+\sum_{j=1}^{+\infty} \bg_j z^j.
	$$
	Fix an integer $k\geq 1$. Based on the implicit
    equation~\eqref{new claim 5.6} satisfied by $\bar{\eta}_n$, we
    determine inductively the coefficients of $\ba_{k+1}$ and
    $\bg_{k}$ in the expression of $\bar{\eta}_n$. More precisely, in
    order to explicit the dependence of $\bar{\eta}_n$ on $\ba_{k+1}$,
    resp. $\bg_{k}$, we differentiate~\eqref{new claim 5.6} with
    respect to $\ba_{k+1}$, resp. $\bg_{k}$, and plug the expansion we
    already have in the right hand side. The presence of the extra
    factor $\lambda^n$ acts as a shift, i.e., it ``propagates" the
    information we have one step further.  Besides, in order to avoid
    seeing new unknown quantities $\ba_{k+2},\bg_{k+1},\dots$, we only
    consider the
    terms %with maximum weight containing $\ba_{k+1}$ and $\bg_{k}$. This corresponds to the terms
    in the expansion of $\bar{\eta}_n$ which are aligned along the
    line of slope $1$ based at the points where $\ba_{k+1}$ and
    $\bg_{k}$ first appear (see Fig.~\ref{image
      tableau}). %This corresponds to the case where we always differentiate with respect to $\ba_1$ or $\bg_{k-1}$, resp. $\ba_1$ or $\ba_{k}$.
	At each step, the expressions we obtain in the derivative
    of~\eqref{new claim 5.6} are combinations of terms of two kinds:
	\begin{itemize}
    \item terms of the form
      $n^A \ba_1^B \ba_{k+1} \lambda^{nC} \bar{\eta}_n^D$ or
      $n^A\ba_1^B \bg_{k} \lambda^{nC} \bar{\eta}_n^D$ with $C$ large;
      in this case, we use the expressions of the first coefficients
      of $\bar{\eta}_n^D$ given by~\eqref{puissances de mu};
    \item terms of the form $n^A\ba_1^B \lambda^{nC} \bar{\eta}_n^D$
      with $C$ small; in this case, based on the expansion computed
      previously, we identify the coefficients of $\ba_{k+1}$ and
      $\bg_{k}$ in the expression of $\bar{\eta}_n^D$ in order to go
      one step further in the expansion.
	\end{itemize}
\end{remark}

\begin{figure}[H]
	\begin{center}
      \includegraphics [width=14cm]{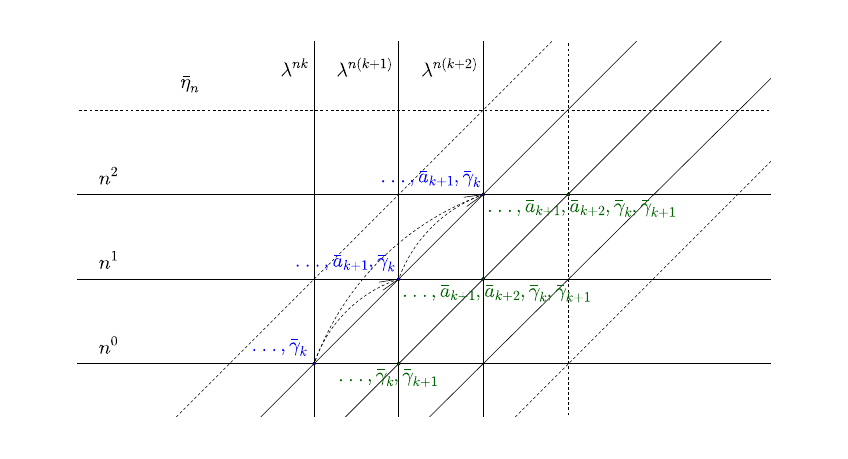}
		\caption{Coefficients in the series expansion of $\bar{\eta}_n$.} \label{image tableau}
	\end{center}
\end{figure}

\begin{defi}\label{triangulaire series}
	For any integer $k \geq 0$, a formal series $\mathcal{S}$ is called $k$-triangular if
	\begin{enumerate}
		\item it is of the form $\mathcal{S}=\sum_{p-q\geq k, q \geq 0} s_{q,p}n^q \lambda^{np}$;
		\item its principal part
		$\mathbb{P}(\mathcal{S}):=\sum_{p-q=k} s_{q,p}n^q \lambda^{np}$ is non-zero, i.e., $\mathbb{P}(\mathcal{S})\neq 0$.
	\end{enumerate}
	In other words, $\mathcal{S}$ is the product of $\lambda^{nk}$ and of a balanced series with non-zero principal part.
	%coefficients along the line of slope one as in Fig.~\ref{image tableau},
\end{defi}

Let us state some basic properties of triangular series which will be
useful in the following.
\begin{remark}\label{remark triangulaire} Let $\mathcal{S}_0$ be a
$k$-triangular series, for some integer $k \geq 0$.
\begin{itemize}
\item for any $k$-triangular series $\mathcal{S}_1$ such that
  $\mathbb{P}(\mathcal{S}_0)+\mathbb{P}(\mathcal{S}_1)\neq 0$, the
  series $\mathcal{S}_0+\mathcal{S}_1$ is $k$-triangular, and
  $\mathbb{P}(\mathcal{S}_0+ \mathcal{S}_1)=\mathbb{P}(\mathcal{S}_0)+
  \mathbb{P}(\mathcal{S}_1)$.
\item For any integer $\ell \geq 0$ and any $\ell$-triangular series
  $\mathcal{S}_2$ such that
  $\mathbb{P}(\mathcal{S}_0) \mathbb{P}(\mathcal{S}_2)\neq 0$, the
  series $\mathcal{S}_0 \mathcal{S}_2$ is $(k+\ell)$-triangular and it
  holds
	$$
	\mathbb{P}(\mathcal{S}_0 \mathcal{S}_2)=\mathbb{P}(\mathcal{S}_0)
    \mathbb{P}(\mathcal{S}_2).
	$$
  \end{itemize}
  In particular, if $\mathcal{S}$ is a balanced series and if
  $\omega\colon z \mapsto \sum_{j=0}^{+\infty} \omega_j z^j$ is an
  analytic function with $\omega_0\neq 0$, then
  $\omega(\lambda^n \mathcal{S})$ is a $0$-triangular series, and
  $\mathbb{P}(\omega(\mathcal{S}))=\omega_0$.
\end{remark}

\begin{proof}[Proof of Lemma~\ref{lemme utile mochomoge}] Let $k \geq 1$.

\noindent\textsc{Proof of~\eqref{eq:mu-rels}}: since $\partial_{\ba_{k+1}}\bar{\Delta}\colon z \mapsto z^{k+1}$, differentiating \eqref{new claim 5.6} with respect to $\ba_{k+1}$, we thus get
	\begin{align*}
		&\partial_{\ba_{k+1}} \bar{\eta}_n= n\bar{\Delta}(\lambda^n \bar{\eta}_n \bg(\lambda^n \bar{\eta}_n))^{n-1} \Big[ \lambda^{n(k+1)} \bar{\eta}_n^{k+1} (\bg(\lambda^n \bar{\eta}_n))^{k+2}+\lambda^n
		\partial_{\ba_{k+1}}\bar{\eta}_n[\bg(\lambda^n \bar{\eta}_n)+\\
		&+ \lambda^{n} \bar{\eta}_n\gamma'(\lambda^n\bar{\eta}_n)] \bar{\Delta}'(\lambda^n \bar{\eta}_n\bg(\lambda^n \bar{\eta}_n))\bg(\lambda^n \bar{\eta}_n)\Big]+\partial_{\ba_{k+1}}\bar{\eta}_n\cdot \bar{\Delta}(\lambda^n \bar{\eta}_n \bg(\lambda^n \bar{\eta}_n))^{n}\lambda^n \gamma'(\lambda^n\bar{\eta}_n).
		%&\cdot (1+\gamma(\lambda^n\bar{\eta}_n))+\bar{\Delta}(\lambda^n \bar{\eta}_n(1+\gamma(\lambda^n \bar{\eta}_n)))^n \partial_{\ba_{k+1}}(1+\gamma(\lambda^n\bar{\eta}_n)).
	\end{align*}
	By the fact that $\bar\eta_n$ is a $0$-triangular series, it follows from this expression that $	\partial_{\ba_{k+1}} \bar{\eta}_n$ is a $k$-triangular series with leading term
	\begin{equation}\label{leading temr eta n}
	\partial_{\ba_{k+1}} \bar{\eta}_n=n \lambda^{n(k+1)}+\dots %O(n^2 \lambda^{({k+1}+1)n}).
	\end{equation}
	Besides, the terms $\gamma(\lambda^n\bar{\eta}_n)$ and
    $\lambda^n \gamma'(\lambda^n \bar{\eta}_n)$ are $1$-triangular and
    are mutiplied by $k$-triangular terms in the previous expression,
    hence by Remark~\ref{remark triangulaire} they do not contribute
    to the principal part of $\partial_{\ba_{k+1}} \bar{\eta}_n$. For
    the same reason, the only term in the expansion of
    $ \bar{\Delta}'(\lambda^n \bar{\eta}_n \bg(\lambda^n
    \bar{\eta}_n))$ which contributes to the principal part of
    $\partial_{\ba_{k+1}} \bar{\eta}_n$ is the constant term
    $\ba_1$.  By~\eqref{puissances de mu} and~\eqref{leading temr eta
      n}, we thus get
	\begin{align*}
		\mathbb{P}(\partial_{\ba_{k+1}} \bar{\eta}_n)&=\mathbb{P}\Big( \bar{\Delta}(\lambda^n \bar{\eta}_n)^{n-1} \Big[n \lambda^{n(k+1)} \bar{\eta}_n^{k+1} +n\lambda^n
		\partial_{\ba_{k+1}}\bar{\eta}_n \ba_1
		%\cdot \bar{\Delta}'(\lambda^n \bar{\eta}_n\bg(\lambda^n \bar{\eta}_n))
		\Big]\Big)\\
		%&\cdot (1+\gamma(\lambda^n\bar{\eta}_n))+\bar{\Delta}(\lambda^n \bar{\eta}_n(1+\gamma(\lambda^n \bar{\eta}_n)))^n \partial_{\ba_{k+1}}(1+\gamma(\lambda^n\bar{\eta}_n)).
		&=(1+n\ba_1 \lambda^n )\Big[n \lambda^{n(k+1)}(1+(k+1) n \ba_1 \lambda^n )+n^2\ba_1 \lambda^{n(k+2)}\Big]+O(n^3 \lambda^{n(k+3)})\\
		&=n \lambda^{n(k+1)}+(k+3) n^2 \ba_1 \lambda^{n(k+2)}+ O(n^3 \lambda^{n(k+3)}).
	\end{align*}

	Similarly, $\partial_{\bg_{k}}\bar{\gamma}=\partial_{\bg_{k}}\gamma\colon z \mapsto z^{k}$, hence
	\begin{align*}
		&\partial_{\bg_{k}} \bar{\eta}_n= \bar{\Delta}(\lambda^n \bar{\eta}_n \bg(\lambda^n \bar{\eta}_n))^{n} \Big[\lambda^{nk} \bar{\eta}_n^k+\lambda^{n} \partial_{\bg_{k}} \bar{\eta}_n \gamma'(\lambda^n \bar{\eta}_n)\Big]+  n  \bar{\Delta}'(\lambda^n \bar{\eta}_n \bg(\lambda^n \bar{\eta}_n))\cdot \\
		&\bar{\Delta}(\lambda^n \bar{\eta}_n \bg(\lambda^n \bar{\eta}_n))^{n-1}\Big[\lambda^n \partial_{\bg_{k}}\bar{\eta}_n [\bg(\lambda^n \bar{\eta}_n)+\lambda^n \bar{\eta}_n\gamma'(\lambda^n \bar{\eta}_n)]+\lambda^{n(k+1)}\bar{\eta}_n^{k+1}\Big]\bg(\lambda^n \bar{\eta}_n).
		%&\cdot (1+\gamma(\lambda^n\bar{\eta}_n))+\bar{\Delta}(\lambda^n \bar{\eta}_n(1+\gamma(\lambda^n \bar{\eta}_n)))^n \partial_{\ba_{k+1}}(1+\gamma(\lambda^n\bar{\eta}_n)).
	\end{align*}
	By the fact that $\bar\eta_n$ is a balanced series, it follows from this expression that $\partial_{\bg_{k}} \bar{\eta}_n$ is a $k$-triangular series with leading term
	$$
	\partial_{\bg_{k}} \bar{\eta}_n=\lambda^{nk}+\dots %O(n^2 \lambda^{({k+1}+1)n}).
	$$
	As previously, the terms  $\gamma(\lambda^n\bar{\eta}_n)$ and $\lambda^n \gamma'(\lambda^n \bar{\eta}_n)$ do not contribute to the principal part of $\partial_{\bg_{k}} \bar{\eta}_n$, as they are $1$-triangular and are mutiplied by $k$-triangular terms in the previous expression, and $ \bar{\Delta}'(\lambda^n \bar{\eta}_n \bg(\lambda^n \bar{\eta}_n))$ can be replaced with  $\ba_1$.  Together with~\eqref{puissances de mu}, the previous expansion   thus yields
	\begin{align*}
		\mathbb{P}(\partial_{\bg_{k}} \bar{\eta}_n)&=\mathbb{P}\Big(\bar{\Delta}(\lambda^n \bar{\eta}_n )^{n} \lambda^{nk} \bar{\eta}_n^k+  n \ba_1\lambda^n
		%\bar{\Delta}'(\lambda^n \bar{\eta}_n )
		\bar{\Delta}( \lambda^n \bar{\eta}_n )^{n-1}\Big[ \partial_{\bg_{k}}\bar{\eta}_n %\cdot\bg(\lambda^n \bar{\eta}_n)
		+\lambda^{nk}\bar{\eta}_n^{k+1}\Big]\Big)\\
		&= (1+n\ba_1\lambda^n)\Big[\lambda^{nk}(1+kn \ba_1 \lambda^n)+n\ba_1 \lambda^n \cdot 2\lambda^{nk}\Big]+O(n^2 \lambda^{n(k+2)})\\
		&=\lambda^{nk}+(k+3) n \ba_1 \lambda^{n(k+1)}+ O(n^2 \lambda^{n(k+2)}).
	\end{align*}
	Plugging this back in the expression of $\mathbb{P}(\partial_{\bg_{k}} \bar{\eta}_n)$, and going one step further in the expansion of $\bar{\Delta}(\lambda^n \bar{\eta}_n )^{n}$, we can compute the third term in the principal part of $\partial_{\bg_{k}} \bar{\eta}_n$ (we do not give the details here as it will not be needed in the following). \\

\noindent\textsc{Proof of~\eqref{eq:nu-rels}}: the derivation of the
    coefficients $\{\nu_{j,k}^\pm\}_{j,k}$ is done in a similar
    way. Yet, unlike $\bar{\eta}_n$, now there is no implicit equation
    anymore, thus we can differentiate
    $\lambda^{\mp n} \Delta_n^{\pm n}$ directly and use the
    expressions of $\{\mu_{j,k}\}_{j,k}$ obtained above. We note that
	$$
	\lambda^{\mp n} \Delta_n^{\pm n}=\bar{\Delta}(\lambda^n \bar{\eta}_n \bg(\lambda^n \bar{\eta}_n))^{\pm n}.
	$$
	We will detail the calculations only  for $\pm=-$ which is the case we will need in the following. The case where $\pm=+$ is analogous.
	By the above formula, we thus get
	\begin{align*}
		\partial_{\ba_{k+1}}\lambda^{ n} \Delta_n^{- n}&=-n \bar{\Delta}(\lambda^n \bar{\eta}_n \bg(\lambda^n \bar{\eta}_n))^{-(n+1)}\Big[\lambda^{n(k+1)} \bar{\eta}_n^{k+1}(\bg(\lambda^n \bar{\eta}_n))^{k+1}+\\
		&+\lambda^n
		\partial_{\ba_{k+1}}\bar{\eta}_n[\bg(\lambda^n \bar{\eta}_n)+ \lambda^{n} \bar{\eta}_n\gamma'(\lambda^n\bar{\eta}_n)] \bar{\Delta}'(\lambda^n \bar{\eta}_n\bg(\lambda^n \bar{\eta}_n))\Big].
	\end{align*}
	We see that the associated series is also $k$-triangular. For the same reason as before, the terms $\gamma(\lambda^n\bar{\eta}_n)$ and $\lambda^n \gamma'(\lambda^n \bar{\eta}_n)$  need not be considered, and the only term of $\bar{\Delta}'(\lambda^n \bar{\eta}_n\bg(\lambda^n \bar{\eta}_n))$ which contributes to the principal part of $\partial_{\ba_{k+1}}\lambda^{ n} \Delta_n^{- n}$ is $\ba_1$.  Replacing $	\partial_{\ba_{k+1}}\bar{\eta}_n$ with the value computed previously, and by~\eqref{puissances de mu}, we thus obtain
	\begin{align*}
		&\mathbb{P}(\partial_{\ba_{k+1}}\lambda^{ n} \Delta_n^{- n})=-\mathbb{P}\Big( \bar{\Delta}(\lambda^n \bar{\eta}_n )^{-(n+1)}\Big[n\lambda^{n(k+1)} \bar{\eta}_n^{k+1}+n \ba_1 \lambda^n
		\partial_{\ba_{k+1}}\bar{\eta}_n\Big]\Big)\\
		&=-(1-n\ba_1 \lambda^n)\Big[n\lambda^{n(k+1)}(1+(k+1)n\ba_1 \lambda^n)+n^2 \ba_1 \lambda^{n(k+2)}\Big]+O(n^3 \lambda^{n(k+3)})\\
		&=-n \lambda^{n(k+1)}-(k+1)n\ba_1 \lambda^{n(k+2)}+O(n^3 \lambda^{n(k+3)}).
	\end{align*}

	Similarly,
	\begin{align*}
		\partial_{\bg_{k}}\lambda^{ n} \Delta_n^{- n}&=-n\bar{\Delta}'(\lambda^n \bar{\eta}_n\bg(\lambda^n \bar{\eta}_n)) \bar{\Delta}(\lambda^n \bar{\eta}_n \bg(\lambda^n \bar{\eta}_n))^{-(n+1)}\Big[\lambda^{n(k+1)} \bar{\eta}_n^{k+1}+\\
		&+\lambda^n \partial_{\bg_{k}}\bar{\eta}_n [\bg(\lambda^n \bar{\eta}_n)+\lambda^n \bar{\eta}_n\gamma'(\lambda^n \bar{\eta}_n)] \Big].
	\end{align*}
	Arguing as before, and replacing $	\partial_{\bg_{k}}\bar{\eta}_n$ with the value computed previously, we get
	\begin{align*}
		&\mathbb{P}(\partial_{\bg_{k}}\lambda^{ n} \Delta_n^{- n})=-\mathbb{P}\Big(n\ba_1 \bar{\Delta}(\lambda^n \bar{\eta}_n )^{-(n+1)}\Big[\lambda^{n(k+1)} \bar{\eta}_n^{k+1}+\lambda^n \partial_{\bg_{k}}\bar{\eta}_n \Big]\Big)\\
		&=-n \ba_1 \lambda^n (1-n\ba_1 \lambda^n)\Big[\lambda^{nk}(1+(k+1)n\ba_1 \lambda^n)+\lambda^{nk}+(k+3) n \ba_1 \lambda^{n(k+1)}\Big]+O(n^3 \lambda^{n(k+3)})\\
		&=-2n \lambda^{n(k+1)}-2(k+1)n^2\ba_1 \lambda^{n(k+2)}+O(n^3 \lambda^{n(k+3)}).
	\end{align*}

	\noindent\textsc{Proof of~\eqref{eq:rho-rels}}: let us now deal
    with the coefficients $\{\rho_{j,k}\}_{j,k}$. The computations are
    carried out in the same way as for the coefficients
    $\{\nu^-_{j,k}\}_{j,k}$. Set
    $\bar D\colon z \mapsto \sum_{j=1}^{+\infty} j \ba_j z^j$. It
    holds
	$$
	1-n\Delta'_n\Delta_n^{-1}=1-n \bar D(\lambda^n \bar{\eta}_n \bg(\lambda^n \bar{\eta}_n))\bar{\Delta}(\lambda^n \bar{\eta}_n \bg(\lambda^n \bar{\eta}_n))^{-1}.
	$$
	We have $\partial_{\ba_{k+1}} \bar D\colon z \mapsto (k+1)z^{k+1}$. It follows that
	\begin{align*}
		&\partial_{\ba_{k+1}}(1-n\Delta'_n\Delta_n^{-1})=-n\lambda^{n(k+1)} \bar{\eta}_n^{k+1} (\bg(\lambda^n \bar{\eta}_n))^{k+1}\bar{\Delta}(\lambda^n \bar{\eta}_n \bg(\lambda^n \bar{\eta}_n))^{-1}\Big[(k+1)-\\
		&-\Delta'_n\Delta_n^{-1}\Big]-n\lambda^n
		\partial_{\ba_{k+1}}\bar{\eta}_n[\bg(\lambda^n \bar{\eta}_n)+ \lambda^{n} \bar{\eta}_n\gamma'(\lambda^n\bar{\eta}_n)] \bar{\Delta}(\lambda^n \bar{\eta}_n\bg(\lambda^n \bar{\eta}_n))^{-1}\cdot \\
		&\cdot\Big[\bar D'(\lambda^n \bar{\eta}_n \bg(\lambda^n \bar{\eta}_n))-\bar \Delta'(\lambda^n \bar{\eta}_n \bg(\lambda^n \bar{\eta}_n)) \Delta'_n\Delta_n^{-1}\Big].
	\end{align*}
	In the previous expression, $[(k+1)- \Delta'_n\Delta_n^{-1}]$ is a $0$-triangular series whose principal part is reduced to $k+1$, as $\Delta'_n\Delta_n^{-1}$ is $1$-triangular.  Similarly, $\bar D'(\lambda^n \bar{\eta}_n \bg(\lambda^n \bar{\eta}_n))-\bar \Delta'(\lambda^n \bar{\eta}_n \bg(\lambda^n \bar{\eta}_n)) \Delta'_n\Delta_n^{-1}$  is $0$-triangular, and its principal part is equal to $\ba_1$. We see that $\partial_{\ba_{k+1}}(1-n\Delta'_n\Delta_n^{-1})$ is $k$-triangular, and as before, the terms $\gamma(\lambda^n\bar{\eta}_n)$ and $\lambda^n \gamma'(\lambda^n \bar{\eta}_n)$ do not contribute to the principal part of  $\partial_{\ba_{k+1}}(1-n\Delta'_n\Delta_n^{-1})$. We thus obtain
	\begin{align*}
		&\mathbb{P}\Big(\partial_{\ba_{k+1}}(1-n\Delta'_n\Delta_n^{-1})\Big)\\
		&=-\mathbb{P}\Big(  \bar{\Delta}(\lambda^n \bar{\eta}_n \bg(\lambda^n \bar{\eta}_n))^{-1}\Big[n\lambda^{n(k+1)}\bar{\eta}_n^{k+1} (k+1)+ n\ba_1\lambda^n
		\partial_{\ba_{k+1}}\bar\eta_n\Big]\Big)\\
		&=-n\lambda^{n(k+1)}(1+(k+1)n\ba_1 \lambda^n )(k+1)-n\ba_1\lambda^n \cdot n \lambda^{n(k+1)}+ O(n^3 \lambda^{n(k+3)})\\
		&=-(k+1)n\lambda^{n(k+1)}-[(k+1)^2+1]n^2\ba_1  \lambda^{n(k+2)}+ O(n^3 \lambda^{n(k+3)}).
	\end{align*}
	Finally, we have
	\begin{align*}
		&\partial_{\bg_{k}}(1-n\Delta'_n\Delta_n^{-1})=-n\Big[\lambda^n
		\partial_{\bg_{k}}\bar{\eta}_n[\bg(\lambda^n \bar{\eta}_n)+ \lambda^{n} \bar{\eta}_n\gamma'(\lambda^n\bar{\eta}_n)]+\lambda^{n(k+1)} \bar{\eta}_n^{k+1}\Big] \cdot\\
		&\cdot\bar{\Delta}(\lambda^n \bar{\eta}_n\bg(\lambda^n \bar{\eta}_n))^{-1}\cdot\Big[\bar D'(\lambda^n \bar{\eta}_n \bg(\lambda^n \bar{\eta}_n))-\bar \Delta'(\lambda^n \bar{\eta}_n \bg(\lambda^n \bar{\eta}_n)) \Delta'_n\Delta_n^{-1}\Big].
	\end{align*}
	Arguing as above, we deduce that
	\begin{align*}
		&\mathbb{P}\Big(\partial_{\bg_{k}}(1-n\Delta'_n\Delta_n^{-1})\Big)
		=-\mathbb{P}\Big( n\ba_1(\lambda^n
		\partial_{\bg_{k}}\bar{\eta}_n+\lambda^{n(k+1)} \bar{\eta}_n^{k+1})\Big)\\
		&=-n\ba_1 \lambda^{n(k+1)}-(k+3) n^2 \ba_1^2 \lambda^{n(k+2)}-n\ba_1 \lambda^{n(k+1)}-(k+1)n^2 \ba_1^2 \lambda^{n(k+2)}+ O(n^3 \lambda^{n(k+3)})\\
		&=-2n\ba_1 \lambda^{n(k+1)}-2(k+2)n^2 \ba_1^2 \lambda^{n(k+2)}+ O(n^3 \lambda^{n(k+3)}).
	\end{align*}
\end{proof}
We reported the above computations since they could be useful in some
further developments of this work.  In the current section we will in
fact only rely upon some specific combinations, which occur in the
term denoted with $\I_{n}$ and we collect in the following corollary
(recall that the quantities $\I_{n}$, $\II_{n}$, $\III_{n}$ below are
those which were introduced in Lemma \ref{lemme exposant de lya}).
\begin{corollary}\label{cor:tildeIn}
  The following holds:
  \begin{align*}
    \lambda^{n}\Delta_{n}^{-n}(1-n\Delta'_{n}\Delta_{n}^{-1}) &= %\\
    \sum_{p = 0}^{\infty}\sum_{q = 0}^{p}L_{q,p}^{*}n^{q}\lambda^{np},
  \end{align*}
  and we have:
  \begin{subequations}
    \begin{align}
      L^{*}_{0,p}   & = c^{*}_{0,p},   \label{eq:L0p-star}\\
      L^{*}_{1,p+1} & = -4\ba_{1}\bg_{p}-(p+2)\ba_{p+1} + c^{*}_{1,p+1}, \label{eq:L1p-star}\\
      L^{*}_{2,p+2} & = -2(2p+1)\ba^{2}_{1}\bg_{p} - (p+1)^{2}\ba_{1}\ba_{p+1} + c^{*}_{2,p+2}, \label{eq:L2p-star}
    \end{align}
  \end{subequations}
  where $c^{*}_{i,p+1}$ depend only on the coefficients
  $\{\bg_{\ell},\ba_{\ell+1}\}_{0\leq \ell < p}$.
  Moreover
  \begin{align}\label{eq:L00-star}
    L^{*}_{0,0} &= 1,&%\\
    L^{*}_{1,1} &= -2\ba_{1}, &%\\
    L^{*}_{2,2} &= -\frac12\ba_{1}^{2}.
  \end{align}
\end{corollary}
\begin{proof}
  By definition of $L^{*}_{q,p}$, we gather that for any $p\ge 0$ and
  $0\le q \le p$:
  \begin{align*}
    L^{*}_{q,p} &= \sum_{\substack{p'+p'' = p\\q'+q'' = q}}\nu_{q',p'}^{-}\rho_{q'',p''}.
  \end{align*}
  Observe that the contribution of terms for which both $p'-q' < p$
  and $p''-q'' < p$ can be absorbed in the terms $c^*$, since they do
  not depend on either $\ba_{p+1}$ nor $\bg_{p}$ by
  Lemma~\ref{rem:coefficients-homogeneous}.

 In particular, if $q = 0$, then necessarily
  $q' = q'' = 0$, hence:
  \begin{align*}
    L_{0,p}^{*} &= \nu_{0,p}^{-}\rho_{0,0}+\nu_{0,0}^{-}\rho_{0,p}+ c_{0,p};
  \end{align*}
  and using~\eqref{puissances de mu} and Lemma~\ref{lemme utile
    mochomoge} we obtain~\eqref{eq:L0p-star}.
  If $q = 1$ then either $q' = 1$ and $q'' = 0$ or $q' = 0$ and $q'' =
  1$.  Observe that by Lemma~\ref{lemme utile mochomoge}, the
  coefficients $\nu^{-}_{0,k}$ and $\rho_{0,k}$ are $0$ unless $k =
  0$; we conclude that:
  \begin{align*}
    L_{1,p_+1}^{*} &= \nu_{1,p+1}^{-}\rho_{0,0}+\nu_{0,0}^{-}\rho_{1,p+1}+c_{1,p+1},
  \end{align*}
  which yields~\eqref{eq:L1p-star}.
  Finally, we consider the case $q = 2$; in this case one could have
  $q' = 0, 1, 2$ and correspondingly $q'' = 2-q'$.  This leads to:
  \begin{align*}
    L^{*}_{2,p+2} &=%\\
    \nu^{-}_{2,p+2}\rho_{0,0}+\nu^{-}_{1,p+1}\rho_{1,1}+\nu^{-}_{1,1}\rho_{1,p+1}+\nu^{-}_{0,0}\rho_{2,p+2},
  \end{align*}
  which yields~\eqref{eq:L2p-star}.  Equations~\eqref{eq:L00-star}
  then follow from similar arguments, or directly from Lemma~\ref{corollar dofe}.
\end{proof}
\subsection{Determination of the scaled coefficients $\{\bgg_{\ell},\bg_{\ell},\ba_{\ell}\}_{\ell\ge 0}$}\label{detertmin coeffci}

In this part, we keep the same notation and show how the above
estimates can be employed to show that the scaled coefficients $\{\bgg_{\ell},\bg_{\ell},\ba_{\ell}\}_{\ell\ge 0}$ introduced in~\eqref{scaled coefff} are \mlsinv{s}.
% As a consequence of Lemmata~\ref{lemme exposant de lya},~\ref{corollar
%   dofe},~\ref{prop structure du dev} and~\ref{lemme utile mochomoge},
% we obtain:

\begin{lemma}\label{lemma exp lya detailll}
There exists a sequence of real numbers
\begin{align*}
  (L_{q,p})_{\substack{p=0,\cdots,+\infty\\ q=0,\cdots,p}}
\end{align*}
such that for any integer $n \geq n_0$, we have the following
expansion:
\begin{align}\label{eq:exp-lyap-series}
  2\lambda^{n} \cosh(2(n+1)\mathrm{LE}(h_n)) %\\
  &= \sum_{p=0}^{+\infty}\sum_{q=0}^{p}L_{q,p} n^{q}\lambda^{np}.
\end{align}
Moreover, for any $p\ge 1$, the following linear relation holds:
\begin{align}\label{eq:linear-relation-WVC}
  W_p = A_p V_{p} +C_p,
\end{align}
where $V_p,W_p, C_p\in \R^3$ are defined as:
\begin{align}\label{e_VWCdefs}
  V_p&:=\threevec{\bgg_{p}}{\bg_{p}}{\ba_{p+1}},&%\\
  W_p&:=\threevec{L_{0,p}}{L_{1,p+1}}{L_{2,p+2}},&%\\
  C_p&:=\threevec{C_{0,p}}{C_{1,p+1}}{C_{2,p+2}},
\end{align}
$A_p \in \mathfrak{M}_3(\R)$ is given by:
\begin{align*}
  A_p:=\begin{pmatrix}
    % \bg_1 & 1 & 0 & \\
    1 & 0 & 0 \\
    % p\ba_1 \bg_1 & p\ba_1& g_0^- & \\
    (p-2)\ba_1 & -4  \ba_1 g_0 & -(p+2)g_0\\
    % p(p+2)\ba_1^2 \bg_1 & \frac{p(p+2)}{2}\ba_1^2 & g_0^- (p+2)\ba_1 & \\
    \frac{p^2- 2p-1}{2}\ba_1^2 & -2(2p+1)\ba_1^2 g_0 & -(p+1)^2\ba_1 g_0\\
    % & \frac{p(p+3)^2}{6}\ba_1^3 & g_0^-\frac{(p+3)^2}{2}\ba_1^2 &
  \end{pmatrix},
\end{align*}
and for $i \in \{0,1,2\}$, the constants $C_{i,p+i}\in \R$ only depend
on the coefficients\footnote{ Recall the notation introduced
  in~\eqref{scaled coefff}.}
$\{\bgg_{\ell},\bg_{\ell},\ba_{\ell+1}\}_{0\leq \ell < p}$.
\end{lemma}
\begin{proof}
  Let $n\geq n_0$. By Lemma~\ref{lemme exposant de lya}, we have
	\begin{align*}
      2\lambda^{n}\cosh(2(n+1)\mathrm{LE}(h_n)) &= %\\
      \mathrm{I}_n+
      \lambda^{n}\mathrm{II}_n+
      \lambda^{2n}\mathrm{III}_n.
	\end{align*}
    By Remark~\ref{rem:balanced-series},  $\mathrm{I}_n$,
    $\mathrm{II}_n$ and $\mathrm{III}_n$ are balanced series, i.e.:
    \begin{align*}
      \I_{n}&= \sum_{p=0}^{+\infty}\sum_{q=0}^{p}L^{\I}_{q,p} n^{q}\lambda^{np},&%\\
      \II_{n}&= \sum_{p=0}^{+\infty}\sum_{q=0}^{p}L^{\II}_{q,p} n^{q}\lambda^{np},&%\\
      \III_{n}&= \sum_{p=0}^{+\infty}\sum_{q=0}^{p}L^{\III}_{q,p} n^{q}\lambda^{np},
    \end{align*}
    and therefore also the left hand side
    of~\eqref{eq:exp-lyap-series} is a balanced series,
    i.e., \eqref{eq:exp-lyap-series} holds.  We thus need to
    show~\eqref{eq:linear-relation-WVC}.  Let us fix an integer
    $p\geq 1$.  Observe that
    \begin{align}\label{eq:Lpq123}
      L_{q,p} &= L_{q,p}^{\I}+L_{q,p-1}^{\II}+L_{q,p-2}^{\III};
    \end{align}
    moreover, by construction (see
    Lemma~\ref{rem:coefficients-homogeneous}), we can also conclude
    that $L_{q,p}^{*}$ only depend on
    $\{\ba_{i+1}, \bar\gamma_i, \bar g_i\}_{i = 0,\cdots, p-q}$ for
    $* = \I_n$, $\II_n$ and $\III_n$.  Hence, the contributions to
    $L_{0,p}$, (\resp $L_{1,p+1}$, $L_{2,p+2}$) of the last two terms
    in~\eqref{eq:Lpq123} do contain no $\bar g_{p}$, $\bar\gamma_{p}$
    or $\bar a_{p+1}$, and can thus be absorbed in $C_{0,p}$ (\resp
    $C_{1,p+1}$, $C_{2,p+2}$).  In order to
    show~\eqref{eq:linear-relation-WVC} it thus suffices to study the
    coefficients $L_{q,p}^{\I}$ of the balanced series:
    \begin{align*}
      \I_n &= \lambda^{n}\Delta_n^{-n}\big(1-n\Delta'_n\Delta_n^{-1}\big)g(\lambda^{n}\bar\eta_n).
    \end{align*}
    We begin to study the dependence of $\I_{n}$ on $\bgg_{p}$ (\ie
    the first column of $A_{p}$).  Thanks to Lemma \ref{prop structure du dev}, we can write:
    \begin{align*}
      g(\eta_n) &= \sum_{\ell= 0}^{+\infty}\bgg_{\ell} \lambda^{n\ell
               }\bar{\eta}_n^\ell = \sum_{\ell= 0}^{+\infty}
                  \bgg_{\ell} \lambda^{n\ell}\sum_{k = 0}^{\infty}\sum_{j
                  = 0}^{k} \mu_{j,k}^{(\ell)}n^{j}\lambda^{nk}.
    \end{align*}
    Observe that in the expansion of $g$, the coefficient
    $\bgg_{\ell}$ is multiplied by
    $\lambda^{n\ell}\bar{\eta}_n^{\ell}$. Therefore, using Corollary \ref{cor:tildeIn}, we obtain
    \begin{align*}
      \I_{n} &=  \left[ \sum_{p' = 0}^{\infty}\sum_{q' =
               0}^{p'}L_{q',p'}^{*}n^{q'}\lambda^{np'} \right] \cdot \left[
               \sum_{\ell= 0}^{+\infty}
               \bgg_{\ell} \lambda^{n\ell}\sum_{k = 0}^{\infty}\sum_{j
               = 0}^{k} \mu_{j,k}^{(\ell)}n^{j}\lambda^{nk}\right]\\
             &= \sum_{p = 0}^{\infty}\sum_{p'+p''+p''' = p}\sum_{q = 0}^{p'+p'''}
               \sum_{q'+q''' = q}
               L_{q',p'}^{*}\bgg_{p''}\mu_{q''',p'''}^{(p'')}n^{q} \lambda^{np},
    \end{align*}
    which yields:
    \begin{align*}
      L^{\I}_{q,p} &= \sum_{p'+p''+p''' = p} \sum_{\substack{q'+q''' =
                     q\\0\le q'\le p'\\0\le q'''\le p'''}} L_{q',p'}^{*}\bgg_{p''}\mu_{q''',p'''}^{(p'')}.
    \end{align*}
    In order to extract the contribution of $\bgg_{p}$ we thus need to
    set $p'' = p$; we conclude that, for $i \in\{0,1,2\}$, the
    coefficient $\bgg_{p}$ appears in $L_{i,p+i}$ multiplied by a
    factor $K_i^{p} = \sum_{r+s=i}\mu_{r,r}^{(p)}L^{*}_{s,s}$ and
    by~\eqref{puissances de mu} and Corollary~\ref{cor:tildeIn} we
    can at last conclude:
    \begin{align*}
      K_0^{p} &= 1, &K_1^{p}&=(p-2)\ba_1, &K_2^{p}&=\frac{p^2-2p-1}{2}\ba_1^2.
    \end{align*}
%%%%%%%%%%%%%%%%%%%%%%%%%%%%%%%%%%%%%%%%%%%%%%%%%%%%%%%%%%%%%%%%%%%%%%%%
    We now proceed to study the second and third columns of $A_{p}$,
    which amounts to study the dependence on $\bg$ and $\ba$.  This,
    in principle, entails more work than the previous task, since the
    coefficients $\bg$ and $\ba$ show up in the expansions of each of
    the terms in $\I_{n}$, and not just the last term.  As a matter of
    fact, the last term does not contribute at all; in fact notice
    that, as before, we can write:
    \begin{align*}
      g(\eta_n) &=  \sum_{\ell= 0}^{+\infty}
                  \bgg_{\ell} \lambda^{n\ell }\sum_{k = 0}^{\infty}\sum_{j
                  = 0}^{k} \mu_{j,k}^{(\ell)}n^{j}\lambda^{nk},
    \end{align*}
    with the convention that $\mu_{j,k}^{(0)}=0$ for all
    $(j,k)\neq(0,0)$, and $\mu_{0,0}^{(0)}=1$.  As noted earlier
    (recall Lemma \ref{rem:coefficients-homogeneous}),
    the coefficients $\ba_{p+1}$ and $\bg_{p}$ would only occur in the
    expression for $\mu_{j,k}^{(\ell)}$ with $k-j\ge p$.  If we
    consider $L_{0,p}$ (\resp $L_{1,p+1}$, $L_{2,p+2}$), we thus must
    set $k = p$ (\resp $p+1$, $p+2$); in turn this implies that
    $\ell = 0$ (since $\ell+k = p+i$). But then $\mu^{(0)}_{j,k} = 0$
    for any $j,k$.  Thus it suffices to consider the expansion of
    \begin{align*}
      \tilde I_{n} &= g_0\cdot \lambda^{n}\Delta_n^{-n}\big(1-n\Delta'_n\Delta_n^{-1}\big)
    \end{align*}
    and the statement follows from Corollary~\ref{cor:tildeIn}.
\end{proof}
We have seen in Lemma~\ref{corollar dofe} that the values of $g_0$ and
$\ba_1$ are \mlsinv{s}; moreover by Remark~\ref{remark inv man},
$g_0 \neq 0$.
\begin{corollary}\label{coroll detm inv}
  Under the assumption that the first Birkhoff coefficient does not
  vanish, i.e., $\ba_1\neq 0$, then the coefficients
  $\{\bgg_{\ell},\bg_{\ell},\ba_{\ell}\}_{\ell\ge 0}$ are \mlsinv{s}.
\end{corollary}
\begin{proof}
  By Theorem~\ref{corollaire lyapu}, the Marked Lyapunov Spectrum is a
  \mlsinv{}; in particular, for each $n \geq n_0$, $\mathrm{LE}(h_n)$
  is a \mlsinv.  By Lemma~\ref{lemma exp lya detailll}, we also have
	\begin{align*}
		2\lambda^{n} \cosh(2(n+1)\mathrm{LE}(h_n)) &= %\\
        \sum_{p=0}^{+\infty}\sum_{q=0}^{p}L_{q,p} n^{q}\lambda^{np}.
	\end{align*}
    Notice that each term
    $\{n^{q}\lambda^{np}\}_{\substack{p=0,\cdots,+\infty,\\
        q=0,\cdots,p}}$ grows at a different rate as $n \to +\infty$;
    hence, their associated weights can be determined inductively,
    i.e., each coefficient $L_{q,p}$ is a \mlsinv{}.  It thus suffices
    to prove that the coefficients $L_{q,p}$ determine the
    coefficients $\{\bgg_{p},\bg_{p},\ba_{p+1}\}_{p\ge 0}$.

    As recalled above, $g_0$ and $\ba_1$ are \mlsinv{s}, and
    $g_0 \neq 0$, by transversality.  We proceed by induction on
    $p$; by Corollary~\ref{corollar dofe},
    $(\bgg_{0},\bg_{0},\ba_{1})=(g_{0},1,\ba_{1})$ is spectrally
    determined. Given $p \ge 1$, let us assume that the coefficients
    $\{\bgg_{\ell},\bg_{\ell},\ba_{\ell+1}\}_{0 \leq \ell < p}$ are
    known; we want to compute $V_{p}$ (recall~\eqref{e_VWCdefs}).  By
    Lemma~\ref{lemma exp lya detailll}, we have $W_p = A_p V_{p}+C_p$
    for some matrix $A_p\in \mathfrak{M}_3(\R)$, which only depends on
    $p$, $\ba_1$ and $g_{0}$ and hence it is spectrally determined;
    moreover:
    \begin{align*}
      \det A_p=-2p\ba_1^2 g_0^{2}.
    \end{align*}
    In particular, under the assumption that $\ba_1 \neq 0$, we have
    $\det A_p\neq 0$, since $g_0\neq 0$, and $p\geq 1$.  Therefore:
    \begin{align*}
      V_{p} &= A_{p}^{-1}(W_{p}-C_{p}).
    \end{align*}
    By Lemma~\ref{lemma exp lya detailll}, the vector $C_{p}$ is
    determined by inductive hypothesis; $W_p$ is obtained from the
    coefficients $L_{q,p}$, which are \mlsinv{s}; we conclude that the
    vector $V_{p}$ is a \mlsinv{}.
\end{proof}

\subsection{Change of Lyapunov exponents in the horseshoe}\label{change lyapu}

In this subsection, we consider the more general case of a $C^\infty$
billiard table $\mathcal{D}=\R^2 \setminus \cup_{i=1}^{m} \obs_i$ with
$m \geq 3$ obstacles that satisfies the non-eclipse condition. For any
periodic orbit encoded by an admissible word $\hat \sigma$ of length
$p \geq 2$, the construction explained above can be suitably adapted,,
by considering a sequence $(\hat h_n)_{n \geq 0}$ of periodic orbits
accumulating some orbit $\hat h_\infty$ homoclinic to $\hat
\sigma$.

Thus, an analogue of Lemma \ref{lemma exp lya detailll} tells us that
the local variation of the Lyapunov exponent in the horseshoe can be
expressed in terms of three sets of data, which we denote analogously
by $\{\hat g_{\ell},\hat \gamma_{\ell},\hat a_{\ell}\}_{\ell\ge 0}$
(associated to the jets of the corresponding functions).

Let us now consider the (degenerate) situation in which Lyapunov
exponents does not change at all, i.e.,
$\mathrm{LE}(\hat h_n)=\hat L=-\frac 1p \log(\hat \lambda)>0$ for all
$n \geq 1$, where $\hat \lambda\in(0,1)$ is the contracting eigenvalue
of $D\mathcal{F}^p$ at the points of $\hat \sigma$.\footnote{
  According to \cite[Proposition 7.13]{BD}, this condition is
  necessary for the measure of maximal entropy $\mu_*$ constructed in
  \cite{BD} and the SRB measure $\mu_{\text{SRB}}$ to coincide.} By
\eqref{eq intermedi}, we deduce that $\hat g_0\neq 0$ and
$\hat a_1 = 0$ (let us also recall that the parameter $\hat \gamma_0$
is equal to some homoclinic coefficient $\hat \xi_\infty\neq 0$ as
previously).  In particular, if the Lyapunov Spectrum is reduced to a
single point, then necessarily the first Birkhoff invariant of each
periodic orbit has to vanish (which is an infinite codimension
condition -- arguing as in Lemma \ref{non vanish} below), and the
coefficient $\hat g_0$ depends very little on the homoclinic orbit
$\hat h_\infty$, as it is prescribed %homoclinic to $\hat \sigma$
%only through
by  the length of the word $\hat h_0$ %and has to be %the same for every $\hat \sigma$ and every
%independent of
(which is also non-generic, as $\hat g_0$ may be changed by perturbing the geometry of the points in $\hat h_\infty$ far from the orbit $\hat \sigma$ -- see Remark \ref{remark inv man}).
This is elaborated in
more detail as~\cite[Theorem 5.1]{DLVY}

Under the same assumption, let us now consider the analogue of
\eqref{eq:exp-lyap-series}. In this case, the LHS is the sum of a
constant and of a multiple of $\hat \lambda^{2n}$, which implies that
the coefficients $(\hat L_{q,p})_{q,p}$ on the RHS satisfy
$\hat L_{q,p}=0$ for all $p \geq 0$ and $0 \leq q \leq p$ but
$(q,p)=(0,0)$ and $(q,p)=(0,2)$. We cannot argue directly as in
Corollary \ref{coroll detm inv} to recover the value of
$\{\hat g_{\ell},\hat \gamma_{\ell},\hat a_{\ell}\}_{\ell\ge 0}$, as
the computations above were carried out under the assumption that
$\hat a_1 \neq 0$. Nevertheless, as explained in Remark \ref{remark
  partie technique}, by considering other directions in the array of
coefficients $(\hat L_{q,p})_{p,q}$, we may be able to derive many
more constraints that those coefficients have to satisfy in this
case. For instance, in the same way as the first order term in the
variation of the Lyapunov exponent is controlled by
$\hat g_0,\hat a_1$, the variation of $\hat a_1$ is controlled by
$\hat g_1,\hat \gamma_1,\hat a_2\dots$, and as $\hat a_1$ vanishes
identically, it is reasonable to expect that $\hat a_2=0$; by
induction, it may then be possible to show that the Birkhoff
invariants are all zero.  This statement is elaborated
as~\cite[Corollary 5.4]{DLVY}.

\section{Further estimates on the Marked Length Spectrum}\label{conseqeunces}

\subsection{Basic facts about twist maps and generating functions}\label{basic facts}

Recall that for an angle $\varphi \in [-\frac \pi 2,\frac \pi 2]$, we denote $r:=\sin \varphi$, so that the billiard map $\mathcal{F}$ takes the form $\mathcal{F}\colon(s,r)\mapsto (s',r')$.
For any point $(s,r)$ whose iterate $(s',r')=\mathcal{F}(s,r)$ is well defined, we let $(s,s'):=\Psi(s,r)$. The billiard map $\mathcal{F}$ is exact symplectic: for the one-form $\omega_1=\omega_1(s,r):=r ds$, it holds

\begin{equation}\label{func gen billiards}
  \mathcal{F}^*  \omega_1-\omega_1=dh(s,s'),
\end{equation}
for the generating function $h(s,s'):=\|\Upsilon(s)-\Upsilon(s')\|$,
letting $\Upsilon(\cdot)$ is the length-parametrization of
$\partial\mathcal D$.   Let
$T:=\mathcal{F}^2\colon (s,r)\mapsto (s'',r'')$ be the square of the
billiard map, and $h^{(2)}(s,s'')=h(s,s')+h(s',s'')$ be the generating
function for $T$, with $s'=s'(s,s'')$ being determined implicitely by
the condition $\partial_2 h(s,s')+\partial_1 h(s',s'')=0$.  We get
\begin{equation}\label{eq gene t}
  T^*\omega_1-\omega_1=dh^{(2)}(s,s'').%+dh(s',s'')\\
                      %=d\mathcal{F}^* \bar{S}_1+d\bar{S}_1=d\bar{S},
\end{equation}

In the neighborhood
$\mathscr{O}_\infty$ of $x_\infty(-1)$, we have $\mathcal{G}R=RT$,
%so that $\mathcal{G}^*=(R^{-1})^*T^* R^*$,
while in a neighborhood of any
other points of the homoclinic orbit $h_\infty$, we have $NR=RT$.
%so that $ N^*=(R^{-1})^*T^* R^*$.
%and %$\tilde{N}:=R \circ \mathcal{F} \circ R^{-1}$, with $R(s,\varphi)=(\xi,\eta)$.
 Let us set $\omega:=(R^{-1})^* \omega_1$ and $H:=\Psi^* h$. In either case, for $\Phi=N$ or $\mathcal{G}$, we obtain
\begin{equation*}
%N^* \omega-\omega&=
\Phi^* \omega-\omega=
%=(R^{-1})^*(T^*\omega_1-\omega_1)=(R^{-1})^* d\bar{S}=
d S,
\end{equation*}
for the generating function
\begin{align*}
  S:=(R^{-1})^*(\mathcal{F}^* H+H).%=(\mathcal{F}R^{-1})^* \bar{S}_1+ (R^{-1})^* \bar{S}_1.
\end{align*}

\subsection{Estimates on the Marked Length Spectrum}\label{sub esti leng}
%where $\bar{S}:=\bar{S}_1 + \bar{S}_1\circ \mathcal{F}$.
In this part, we go back to the length data and derive some asymptotic
expansion of the lengths of the periodic orbits $(h_n)_n$ as
$n\to +\infty$.  Recall the notation $(\xi_n,\eta_n):=R_-(x_n(1))$ and
$\Delta_n:=\Delta(\xi_n\eta_n)$.  Recall also that
$(\xi_\infty,0)=R_-(x_\infty(1))$.

As in Lemma~\ref{claim eta n delta n xi n} and in~\eqref{eq new parmaetere},
$\eta_n=\Delta_n^{n} \xi_n$, and for any $k \in
\{0,\cdots,n\}$, it holds
\begin{align}\label{eq new parmaetere2}
  R_{-}(x_n(2k+1)) &= (\xi_{n}(2k+1),\eta_{n}(2k+1))= \xi_n(\Delta_{n}^{k}, \Delta_{n}^{n-k}).
\end{align}
\begin{prop}\label{prop lourde un}
	For $n\to +\infty$, the following asymptotics hold:
\begin{align}\label{premier equiv}
  \mathcal{L}(h_n)-(n+1)\mathcal{L}(\sigma)-\mathcal{L}^\infty
  =-\xi_\infty^2	\mathrm{tr} (d^2 S_{(0,0)}) \frac{1}{\lambda^{-1}-\lambda} \lambda^{n}+o(\lambda^n),
\end{align}
with
$\mathrm{tr}(d^2S_{(0,0)}):=\partial_{11}S_{(0,0)}+\partial_{22}S_{(0,0)}$. %for $i=1,2$.
%\end{equation*}
\end{prop}
\begin{proof}
Let $\Sigma_n^1$, $\Sigma_n^2$ be as in~\eqref{ligne 1} and~\eqref{ligne 2}.
We will compute the value of the sum $\Sigma_n^1+\Sigma_n^2=\mathcal{L}(h_n)-(n+1)\mathcal{L}(\sigma)-\mathcal{L}^\infty$
in terms of the generating function $S$ for the
$(\xi,\eta)$-coordinates.

In the following, we only consider the case where $n$ is odd, i.e.,
$n=2m-1$ for some integer $m \geq m_0$; in particular, the period of
$h_n=h_{2m-1}$ equals $2n+2=4m$.
  %, and  $$R(x_n(2m+1))=R(x_n(n+1))=(\xi_n(n+1),\eta_n(n+1))= \xi_n\Delta_{n}^{m}(1,1).$$
  By the palindromic
  symmetry of the orbit
  $h_n$, %with respect to the midpoint $x_n(2m)=x_n(n+1)$,
  $x_n(-k)=\mathcal{I}(x_n(k))$ for $k=0,\cdots,n+1$.  Besides, $H\circ \mathcal{I}=H$ and $S\circ \mathcal{I}_0=S$, with $\mathcal{I}_0\colon (\xi,\eta)\mapsto (\eta,\xi)$.  We have
  \begin{align}
  	\Sigma_n^1&=2\sum_{k=0}^{n}\Big(h(s_n(k),s_n(k+1))-h(s_\infty(k),s_\infty(k+1))\Big)\nonumber \\
  	&=-2 \big(h^{(n+1)} (s_\infty(0),s_\infty(n+1))-h^{(n+1)}(s_n(0),s_n(n+1))\big),\label{somme deux longueurs}
  \end{align}
letting $h^{(n+1)}=h^{(n+1)}(s_0,s_{n+1})$ be the generating  function for $\mathcal{F}^{n+1}$ near the point $(s_n(0),s_n(n+1))$; in other words, for $(s_0,s_{n+1})$ close to $(s_n(0),s_n(n+1))$, we set
$$
h^{(n+1)}(s_0,s_{n+1})=h(s_0,s_1)+h(s_1,s_2)+\dots+h(s_{n},s_{n+1}),
$$
the intermediate parameters $s_1=s_1(s_0,s_{n+1}),\dots,s_n=s_n(s_0,s_{n+1})$ being determined by the condition $\partial_2 h(s_{i-1},s_i)+\partial_1 h(s_i,s_{i+1})=0$, for $i\in \{1,\dots,n\}$.

Considering the Taylor expansion of \eqref{somme deux longueurs}, we obtain
\begin{align*}
	\Sigma_n^1=&-2  dh^{(n+1)}(s_n(0),s_n(n+1))\big[(s_\infty(0),s_\infty(n+1)-(s_n(0),s_n(n+1))\big]\\
	&-d^2h^{(n+1)}(s_n(0),s_n(n+1))\big[(s_\infty(0),s_\infty(n+1)-(s_n(0),s_n(n+1))\big]^2+H.O.T.
	\end{align*}
Since $h^{(n+1)}$ is the generating function for $\mathcal{F}^{n+1}$, as in \eqref{func gen billiards}, we have
$$
\mathcal{F}^{n+1} \omega_1-\omega_1=dh^{(n+1)}(s_0,s_{n+1}).
$$
Therefore, $dh^{(n+1)}(s_n(0),s_n(n+1))=0$, since
$x_{n}(0),x_n(n+1)\in \{r=0\}$ are associated to perpendicular
bounces, and $\omega_1=r ds$. In particular, we only have to care
about second order terms to obtain the leading term in the
expansion~\eqref{premier equiv}.

Let us now express the sum $\Sigma_n^1$ in $(\xi,\eta)$-coordinates:
\begin{align*}
	\Sigma_n^1
	=&2\sum_{k=0}^{m-1}\big(S(\xi_n(2k+1),\eta_n(2k+1))- S(\xi_\infty(2k+1),0)\big)\\
	=&-2\sum_{k=0}^{m-1}%S(\Delta_n^{k} \xi_n ,\Delta_n^{n-k}\xi_n )-S (\lambda^k\xi_\infty ,0)+
	\big(S ( \lambda^k\xi_\infty,0)-S( \Delta_n^{k}\xi_n,\Delta_n^{n-k}\xi_n)\big).
	%&+
\end{align*}
By the previous discussion, first order terms in the Taylor expansion vanish, thus
\begin{equation*}
	\Sigma_n^1
	=-\sum_{|\beta|=2}  \frac{2}{\beta!} \sum_{k=0}^{m-1} %\left[
	\partial^\beta S_{\xi_n(\Delta_n^{k}, \Delta_n^{n-k})}\cdot\left(\lambda^k\xi_\infty-\Delta_n^{k}\xi_n,-\Delta_n^{n-k}\xi_n\right)^\beta+H.O.T.,
  \end{equation*}
  where in the above expression, the sum is taken over the
  multi-indices $\beta=(\beta_1,\beta_2)\in \{(2,0),(1,1),(0,2)\}$,
  denoting
  $\partial^\beta:=\partial^{\beta_1} \circ
  \partial^{\beta_2}$, $\beta!:=\beta_1! \beta_2!$, and
  $(v_1,v_2)^\beta:=v_1^{\beta_1}v_2^{\beta_2}$, for $v_1,v_2 \in
  \R$. By Lemma~\ref{corollar dofe} and Lemma \ref{petitt lemmmma}, we
  have
\begin{equation*}
	|\xi_n-\xi_\infty|=O(\lambda^{n}),\qquad |\Delta_n-\lambda|=O(\lambda^{n}),
\end{equation*}
thus for any $0 \leq k \leq m$,
$|\Delta_n^{k}\xi_n-\lambda^k\xi_\infty|=O(n\lambda^n)$, while
$ |\Delta_n^{n-k}\xi_n|%\right)\right\| =
\simeq \lambda^{n-k} \xi_\infty $ (in particular,
$|\Delta_n^{n-k}\xi_n| \simeq \lambda^{\frac n2} \xi_\infty $ when $k$ is close
to $m$), hence the contribution of the second term overcomes that of
the first term in
$(\Delta_n^{k}\xi_n-\lambda^k\xi_\infty,\Delta_n^{n-k}\xi_n)$. For
$k \simeq m$, we have
$\|\xi_n(\Delta_n^{k},\Delta_n^{n-k})\|=O( \lambda^m)$ too, thus in
order to estimate the leading term in $\Sigma_n^1$ we may consider
partial derivatives $\partial^\beta_{(0,0)}$ instead of
$\partial^\beta_{\xi_n(\Delta_n^{k},\Delta_n^{n-k})}$; indeed,
\begin{equation*}
\partial_{22} S_{\xi_n(\Delta_n^{k}, \Delta_n^{n-k})}(\Delta_n^{n-k}\xi_n)^2\\
=\partial_{22} S_{(0,0)}\lambda^{2(n-k)}\xi_\infty^2+O(\lambda^{\frac{3n}{2}}).
\end{equation*}
Therefore, it holds
\begin{equation*}
	\Sigma_n^1
	=- \xi_\infty^2\partial_{22} S_{(0,0)}  \sum_{k=0}^{m-1}%\partial^{(0,2)} S_{(0,0)} \cdot\big(0, \lambda^{n-k}\big)^{(0,2)}+
	\lambda^{2(n-k)}+o(\lambda^n)=- \xi_\infty^2 \partial_{22} S_{(0,0)}\frac{1}{\lambda^{-1}-\lambda} \lambda^n+o(\lambda^n).
\end{equation*}
As we observed, $S\circ \mathcal{I}_0=S$, and $\xi,\eta$ play symmetric roles in the computations, hence $\partial_{22} S_{(0,0)}=\partial_{11} S_{(0,0)}=\frac 12 \mathrm{tr} (d^2 S_{(0,0)})$, and thus,
\begin{equation}\label{asymptotique sigma un n}
	\Sigma_n^1=- \frac{\xi_\infty^2}{2} \mathrm{tr} (d^2 S_{(0,0)})\frac{1}{\lambda^{-1}-\lambda} \lambda^n+o(\lambda^n).
\end{equation}

We argue similarly for $\Sigma_n^2$. Indeed,
\begin{equation*}
\Sigma_n^2:=2\sum_{k=n+1}^{+\infty}\Big(h(s(k),s(k+1))-h(s_\infty(k),s_\infty(k+1))\Big),%\\
%&=-2 \sum_{k=n+1}^{+\infty}\big(\ell_{2} (s_\infty(k),s_\infty(k+1))-\ell_{2}(s(k),s(k+1))\big),
\end{equation*}
and $dh(s(k),s(k+1))=0$ for all $k \in \Z$, since
$x(k),x(k+1)\in \{r=0\}$ are associated to perpendicular bounces,
hence we only have to care about second order terms in the
expansion. Let us now express the sum $\Sigma_n^2$ in
$(\xi,\eta)$-coordinates:
\begin{align}
	\Sigma_n^2
	=&2\sum_{k=m}^{+\infty}\big(S(0,0)- S(\xi_\infty(2k+1),0)\big)%+\sum_{k\geq m}\big(S(0,0)- S(\xi_\infty(2k+1),0)\big)\\
	=-2\sum_{k=m}^{+\infty}%S(\Delta_n^{k} \xi_n ,\Delta_n^{n-k}\xi_n )-S (\lambda^k\xi_\infty ,0)+
	\big(S ( \lambda^k\xi_\infty,0)-S(0,0)\big) \nonumber \\
	=&-\xi_\infty^2\sum_{k=m}^{+\infty}%S(\Delta_n^{k} \xi_n ,\Delta_n^{n-k}\xi_n )-S (\lambda^k\xi_\infty ,0)+
	\partial_{11} S_{(0,0)} \lambda^{2k}+o(\lambda^n)=-\frac{\xi_\infty^2}{2}	\mathrm{tr} (d^2 S_{(0,0)}) \frac{1}{\lambda^{-1}-\lambda} \lambda^{n}+o(\lambda^n). \label{asymptotique sigma deux n}%&+
\end{align}
It follows from \eqref{asymptotique sigma un n}-\eqref{asymptotique sigma deux n} that
\begin{equation*}
	\Sigma_n^1+\Sigma_n^2
	=-\xi_\infty^2	\mathrm{tr} (d^2 S_{(0,0)}) \frac{1}{\lambda^{-1}-\lambda} \lambda^{n}+o(\lambda^n),
\end{equation*}
which concludes the proof.
\end{proof}

\color{black}
\begin{remark}
  The parameter $\xi_\infty$ in the present paper is different from --
  although related to -- the quantity identified by the same symbol
  in~\cite{BDSKL}.  Indeed, in~\cite{BDSKL}, we chose $R$ to be a
  linearization of the dynamics in a neighborhood of the $2$-periodic
  point with a suitable normalization (see the sentence
  above~\cite[(8)]{BDSKL}).  On the other hand, in this paper, the map
  $R$ is the conjugacy with the Birkhoff Normal Form determined in
  Corollary~\ref{conjug canonique}.  Of course at first order the two
  conjugacies coincide up to a normalization, but since we have chosen
  two different normalizations, the corresponding formulae involving
  $\xi_{\infty}$ (see \eg~\cite[(31)]{BDSKL}) therefore differ from
  the ones obtained in this paper (see \eg~the estimates in
  Proposition~\ref{prop lourde un}).
\end{remark}

\subsection{$\mathcal{MLS}$-determination of the Birkhoff
  data}\label{determ homco}

Let us now assume that the billiard table $\mathcal{D}$ presents
additional symmetries in the sense of Definition~\ref{defi sym}, i.e.,
$\mathcal{D} \in \billiards_{\mathrm{sym}}$. As a consequence of the
above estimates on the Marked Length Spectrum, we can conclude the
following result.
\begin{corollary}\label{premier cororrr}
  Let  $\mathcal{D} \in \billiards_{\mathrm{sym}}$. With the same notations as above, %for any $2$-periodic orbit $(jk)$, with $j<k \in \{1,2,3\}$,
  the
  value of the parameter $\xi_\infty=\xi_\infty(\mathcal{D},1,2)$
  associated to the orbit
  $h_\infty=h_\infty(\mathcal{D},1,2)$\footnote{ See
  	the definitions of $h_\infty$ and $\xi_\infty$ in Section
  	\ref{section extension}.} homoclinic to $(12)$ is a \mlsinv{}.
\end{corollary}

\begin{proof}
  By the estimates obtained in Proposition~\ref{prop lourde un}, the
  stable eigenvalue $\lambda$ is a \mlsinv{}; indeed, since $\mathrm{tr}(d^2 S_{(0,0)})>0$ (see Lemma \ref{lem:quadratic-form} and \eqref{relation traces} below), it holds
  \begin{align}
    \lim_{n \to +\infty}\frac 1n \log \left|\mathcal{L}(h_n)-(n+1)\mathcal{L}(\sigma)-\mathcal{L}^\infty\right| = \log \lambda,
  \end{align}
  the left hand side being spectrally determined\footnote{
    Of course this is a particular case of
    Theorem~\ref{prp:palindromic_lyapunov}.}. Still by
  Proposition~\ref{prop lourde un}, we deduce that
  $\xi_\infty^2 \mathrm{tr} (d^2 S_{(0,0)})$ is also \mlsinv{}. We
  claim that in fact, each of the two quantities
  $\mathrm{tr}(d^2 S_{(0,0)})$ and $\xi_{\infty}$ is \mlsinv.

  Indeed, the trace $\mathrm{tr}(d^2 S_{(0,0)})$ of the Hessian of $S$
  at $(0,0)$ can also be computed in $(s,s'')$-coordinates:
  letting $h^{(2)}(s,s'')=h(s,s')+h(s',s'')$ be the generating
  function of $T=\mathcal{F}^2$ at a point $(s,s'')$, for
  $s'=s'(s,s'')$ chosen as in Subsection \ref{basic facts}, and noting
  that the point $(0,0)$ is critical for $h^{(2)}$, it holds
  \begin{equation}\label{relation traces}
  \mathrm{tr}(d^2 S_{(0,0)})=\mathrm{tr}(P^T A P),
  \end{equation}
  where $A$ is the matrix of the Hessian $d^2 h^{(2)}_{(0,0)}$ at $(0,0)$, and $P=D\psi_{(0,0)}$ is the differential of the change of coordinates $\psi \colon (\xi,\eta)\mapsto (s,s'')$. By \eqref{eq gene t}, we have %such that $(s,s')=\Psi(s,\sin\varphi)$, and $(s'',\sin \varphi''):=T(s,\sin \varphi)$, we have
  \begin{align}
  \partial_1 h^{(2)}(s,s'')&=\partial_1 h(s,s'(s,s''))=-\sin\varphi,\label{premier ligne partial un h}\\
  \partial_2 h^{(2)}(s,s'')&=\partial_2 h(s'(s,s''),s'')=\sin\varphi'',\nonumber
  \end{align}
  letting $\varphi=\varphi(s,s'')$,
  resp. $-\varphi''=-\varphi''(s,s'')$ be the angle between the normal
  at $s$, resp. $s''$, with the segment connecting $s$ to $s'$,
  resp. $s''$ to $s'$.

  Recall the formula of the differential
  $(\delta s,\delta \varphi)\mapsto (\delta s',\delta \varphi')$ of
  the billiard map in $(s,\varphi)$-coordinates (see Chernov-Markarian
  \cite[p. 35]{CM}): at a point $(s,\varphi)$, it holds
\begin{equation*}%\label{eq diff bill map s phi}
\begin{pmatrix}
\delta s'\\
\delta \varphi'
\end{pmatrix}=-\frac{1}{\cos \varphi'} \begin{pmatrix}
\mathscr{L}\mathcal{K}+\cos\varphi & \mathscr{L}\\
\mathscr{L}\mathcal{K}\mathcal{K}'+\mathcal{K}\cos\varphi'+\mathcal{K}'\cos  \varphi & \mathscr{L}\mathcal{K}'+\cos\varphi'
\end{pmatrix}\begin{pmatrix}
\delta s\\
\delta \varphi
\end{pmatrix},
\end{equation*}
letting $(s',\varphi')$ be the image of $(s,\varphi)$ by the billiard map, $\mathscr{L}:=h(s,s')$, and $\mathcal{K},\mathcal{K}'$ be the respective curvatures at the points $s,s'$. In particular, the differential of the square of the billiard map is equal to
\begin{equation}\label{eq diff bill map  square s phi}
\begin{pmatrix}
\delta s\\
\delta \varphi
\end{pmatrix}\mapsto \begin{pmatrix}
\delta s''\\
\delta \varphi''
\end{pmatrix}=\frac{1}{\cos \varphi'\cos \varphi''} \begin{pmatrix}
a & b\\
* & *
\end{pmatrix}\begin{pmatrix}
\delta s\\
\delta \varphi
\end{pmatrix},
\end{equation}
with
\begin{align*}
	a&:=2 \mathscr{L}\mathscr{L}' \mathcal{K}\mathcal{K}'+ (\mathscr{L}+\mathscr{L}')\mathcal{K}\cos \varphi'+2\mathscr{L}'\mathcal{K}'\cos \varphi+\cos \varphi\cos \varphi',\\
	b&:=2 \mathscr{L}\mathscr{L}' \mathcal{K}'+(\mathscr{L}+\mathscr{L}')\cos \varphi',
	%c&:=2 \mathscr{L}\mathscr{L}' \mathcal{K}'\mathcal{K}''+2\mathscr{L}\mathcal{K}'\cos \varphi''+ (\mathscr{L}+\mathscr{L}')\mathcal{K}''\cos \varphi'+\cos \varphi'\cos \varphi'',
\end{align*}
letting $(s'',\varphi'')$ be the second iterate of $(s,\varphi)$ under the billiard map, $\mathscr{L}':=h(s',s'')$, and $\mathcal{K}''$ be the curvature at the point $s''$.

In \eqref{premier ligne partial un h}, the angle $\varphi=\varphi(s,s'')$ is seen as a function of $(s,s'')$. In order to compute $\partial_1\varphi(s,s'')$, we set $\delta s''=0$ in \eqref{eq diff bill map  square s phi}, which gives
$
a\delta s + b\delta \varphi=0
$, with $\delta \varphi=\partial_1\varphi(s,s'') \delta s$, so that
\begin{equation}\label{parital phi s}
\partial_1\varphi(s,s'')=-\frac{a}{b}=-\mathcal{K}-\frac{2\mathscr{L}'\mathcal{K}'\cos \varphi+\cos \varphi\cos \varphi'}{2 \mathscr{L}\mathscr{L}' \mathcal{K}'+(\mathscr{L}+\mathscr{L}')\cos \varphi'}.
\end{equation}
Similarly, to compute $\partial_2\varphi(s,s'')$, we set $\delta s=0$ in \eqref{eq diff bill map  square s phi}, which gives
$
b\delta \varphi=\cos \varphi' \cos \varphi'' \delta s''
$, with $\delta \varphi=\partial_2\varphi(s,s'') \delta s''$, so that
\begin{equation}\label{parital phi s seconde}
\partial_2\varphi(s,s'')=\frac{\cos \varphi' \cos \varphi''}{b}=\frac{\cos \varphi' \cos \varphi''}{2 \mathscr{L}\mathscr{L}' \mathcal{K}'+(\mathscr{L}+\mathscr{L}')\cos \varphi'}.
\end{equation}

Differentiating \eqref{premier ligne partial un h} with respect to $s$, and thanks to \eqref{parital phi s}, we deduce that
\begin{equation*}
\partial_{11}h^{(2)}(s,s'')=-\cos \varphi \cdot \partial_1 \varphi(s,s'')=\mathcal{K}\cos \varphi+\frac{2\mathscr{L}'\mathcal{K}'\cos^2 \varphi+\cos^2 \varphi\cos \varphi'}{2 \mathscr{L}\mathscr{L}' \mathcal{K}'+(\mathscr{L}+\mathscr{L}')\cos \varphi'}.%-\mathcal{K}-\frac{\cos\varphi}{\mathscr{L}}.
\end{equation*}
Similarly, differentiating \eqref{premier ligne partial un h} with respect to $s''$, and by \eqref{parital phi s seconde}, we get
\begin{equation*}
	\partial_{21}h^{(2)}(s,s'')=-\cos \varphi \cdot \partial_2 \varphi(s,s'')=-\frac{\cos \varphi\cos \varphi' \cos \varphi''}{2 \mathscr{L}\mathscr{L}' \mathcal{K}'+(\mathscr{L}+\mathscr{L}')\cos \varphi'}.%-\mathcal{K}-\frac{\cos\varphi}{\mathscr{L}}.
\end{equation*}
Moreover, by the time reversal symmetry $(s,\varphi)\mapsto (s,-\varphi)$, we have $\partial_{22} h^{(2)}(s,s'')=\partial_{11} h^{(2)}(s'',s)$, with $s'(s'',s)=s'(s,s'')$, and the orbit segment $(s,\varphi)\mapsto (s',\varphi')\mapsto (s'',\varphi'')$ corresponds to the orbit segment $(s'',-\varphi'')\mapsto (s',-\varphi')\mapsto (s,-\varphi)$, so that
\begin{equation*}
	\partial_{22}h^{(2)}(s,s'')=\mathcal{K''}\cos \varphi''+\frac{2\mathscr{L}\mathcal{K}'\cos^2 \varphi''+\cos^2 \varphi''\cos \varphi'}{2 \mathscr{L}\mathscr{L}' \mathcal{K}'+(\mathscr{L}+\mathscr{L}')\cos \varphi'}.%-\mathcal{K}-\frac{\cos\varphi}{\mathscr{L}}.
\end{equation*}
At the point $(s,s'')=(0,0)$, we have $\varphi = \varphi' = \varphi''=0$, $\mathscr{L}=\mathscr{L}'=\frac{\mathcal{L}(12)}{2}$ is the length of the two-periodic orbit $(12)$,  and $\mathcal{K}=\mathcal{K}'=\mathcal{K}''$ is the common curvature at the bouncing points, hence the matrix $A$ of the Hessian $d^2 h^{(2)}_{(0,0)}$ is  equal to
\begin{equation}\label{expr matrix a}
A=\begin{pmatrix}
\alpha & \beta\\
\beta & \alpha
\end{pmatrix}:=\begin{pmatrix}
\frac{2(\mathscr{L}\mathcal{K}+1)^2-1}{2\mathscr{L}(\mathscr{L}\mathcal{K}+1)} & -\frac{1}{2\mathscr{L}(\mathscr{L}\mathcal{K}+1)}\\
-\frac{1}{2\mathscr{L}(\mathscr{L}\mathcal{K}+1)} & \frac{2(\mathscr{L}\mathcal{K}+1)^2-1}{2\mathscr{L}(\mathscr{L}\mathcal{K}+1)}
\end{pmatrix}.
\end{equation}
Besides, the change of coordinates $\psi\colon (\xi,\eta)\mapsto (s,s'')$ is the composition $\psi=Q^{-1}\circ R^{-1}$, where $R\colon (s,r)=(s,\sin \varphi) \mapsto (\xi,\eta)$ is the map to go to Birkhoff coordinates, and $Q\colon (s,s'') \mapsto (s,\sin \varphi(s,s''))$. By \eqref{parital phi s}-\eqref{parital phi s seconde}, the matrix of $DQ^{-1}_{(0,0)}$ is
\begin{equation}\label{expr matrix qo}
DQ^{-1}_{(0,0)}=-\frac{1}{\beta}\begin{pmatrix}
-\beta & 0\\
\alpha & 1
\end{pmatrix}.
\end{equation}
Moreover, $DR^{-1}_{(0,0)}$ maps $(1,0)$ to a vector $v_s=(v_1,-v_2)$ in the stable space of $D\mathcal{F}^2_{(0,0)}$ at $(0,0)$, and maps $(0,1)$ to a vector $v_u$ in the unstable space of $D\mathcal{F}^2_{(0,0)}$; since $R$ maps $\{\varphi=0\}$ to $\{\xi=\eta\}$, we also have $v_u=(v_1,v_2)$, and $2v_1v_2=\det (v_s,v_u)=1$, as the change of coordinates $R$ is symplectic. Moreover, due to the symmetry at the two-periodic orbit $(12)$, the matrix of $D\mathcal{F}^2_{(0,0)}$ is equal to $B^2$, with
\begin{align*}
B:=\begin{pmatrix}
\mathscr{L} \mathcal{K}+1 & \mathscr{L}\\
\mathscr{L}\mathcal{K}^2+2 \mathcal{K} &  \mathscr{L} \mathcal{K}+1
\end{pmatrix}.
\end{align*}
After computations, the matrix of $DR^{-1}_{(0,0)}$ is
\begin{equation}\label{expr matrix ro}
DR^{-1}_{(0,0)}=\frac{1}{\sqrt{2}} \begin{pmatrix}
\theta^{-1} & \theta^{-1}\\
-\theta & \theta
\end{pmatrix},\qquad \theta^2 :=\frac{\sqrt{(\mathscr{L}\mathcal{K}+1)^2-1}}{\mathscr{L}}.
\end{equation}
By \eqref{relation traces}, and as $P=D\psi_{(0,0)}=DQ^{-1}_{(0,0)}\circ DR^{-1}_{(0,0)}$, with the expressions of $A$, $DQ^{-1}_{(0,0)}$, $DR^{-1}_{(0,0)}$ obtained in \eqref{expr matrix a}-\eqref{expr matrix qo}-\eqref{expr matrix ro}, we deduce that
\begin{align*}
	\mathrm{tr}(d^2 S_{(0,0)})&=\mathrm{tr}(P^T A P)=\frac{1}{\beta^2}\mathrm{tr}\left(\begin{pmatrix}
		\alpha(\alpha^2-\beta^2) & \alpha^2-\beta^2\\
		\alpha^2-\beta^2 & \alpha
	\end{pmatrix}\begin{pmatrix}
	\theta^{-2} & 0\\
	0 & \theta^2
\end{pmatrix}
\right)\\
&=\frac{\alpha(\alpha^2-\beta^2)\theta^{-2}+\alpha \theta^2}{\beta^2}=4\tau(2\tau^2-1)\sqrt{\tau^2-1},
	\end{align*}
with $\tau:=\frac 12 \mathrm{tr}(B)=\mathscr{L}\mathcal{K}+1$.
    By definition, $\mathscr{L}=\frac{\mathcal{L}(12)}{2}$ is
    \mlsinv{}; moreover, by Theorem \ref{rayon courbure 12}, the
    curvature $\mathcal{K}$ at the bouncing points of the orbit $(12)$
    is also \mlsinv{}, hence, %by the expressions of
    %$\alpha,\beta,\theta$ in \eqref{expr matrix a}-\eqref{expr matrix
    %  qo}-\eqref{expr matrix ro} and
    by the above calculation,
    $\mathrm{tr}(d^2 S_{(0,0)})$ is \mlsinv{}. Since we observed at
    the beginning of the proof that the quantity
    $\xi_\infty^2\mathrm{tr} (d^2 S_{(0,0)})$ is also \mlsinv{}, we
    conclude that $\xi_\infty$ is \mlsinv{}.
\end{proof}
\color{black}

\begin{remark}
	Note that the parameter  $\xi_\infty$  can be interpreted in terms of some area in parameter space.  %\color{blue} %(see also Remark \ref{rem palindod}). \color{black}
	Indeed, let us consider the area $\mathcal{A}_n$ of the quadrilateral of the $(s,r)$-plane bounded by the stable/unstable manifolds of the fixed point $(s(2,1),0)$ and the stable/unstable manifolds of the point $(s_{n}(n+1),0)$. Since the change of coordinates $R$ is symplectic, this area is equal to
	\begin{equation}\label{equation aire qu}
	\mathcal{A}_n=\frac 12\tan\left(  \frac\theta 2 \right)s_n(n+1)^2+ o(\lambda^n)=\lambda^n \xi_\infty^2+o(\lambda^n),
	\end{equation}
	where $\theta\in (0,\pi)$ is the angle between the stable/unstable subspaces at $(s(2,1),0)$.

	Actually, this remark gives another way to see that $\xi_\infty$
    is a \mlsinv{};
    indeed, %his gives another way to show that $\xi_\infty$ is a \mslinv{}. Indeed,
	as in the work of Otal (see \eg~the proof of \cite[Th\'eor\`eme
    2]{O}), it can be shown that the area $\mathcal{A}_n$ is a
    \mlsinv{}, and since $\lambda$ is also a \mlsinv{}, so is
    $\xi_\infty$, by \eqref{equation aire qu}. Moreover (although we
    will not need it for the present paper, where the analysis is done
    for the periodic orbit $(12)$), the same construction can be done
    to show that generically, the Birkhoff Normal Form of any periodic
    point is a \mlsinv{}: as for the orbit $(12)$, we consider a
    sequence of periodic orbits $(\hat h_n)_{n \geq 0}$ in the
    horseshoe generated by some homoclinic intersection between the
    invariant manifolds of the periodic point (see Subsection
    \ref{change lyapu} for more details). Then, the analogue of
    Corollary \ref{coroll detm inv} shows that if the first Birkhoff
    coefficient does not vanish, then in fact, all the Birkhoff
    coefficients are \mlsinv{s}, up to a homoclinic parameter
    $\hat \xi_\infty$; moreover, by the previous remark,
    $\hat \xi_\infty$ can be expressed as an area which can also be
    computed from $\mathcal{MLS}$.
\end{remark}

The above result allows us to conclude:
\begin{corollary}\label{corollary retrouver les invariants}
 Let $\mathcal{D} \in \billiards_{\mathrm{sym}}$, and let
$\mathcal{F}=\mathcal{F}(\mathcal{D})$ be the associated billiard map.
We consider the $2$-periodic orbit $(12)$.  Let
$N=N(\mathcal{D},1,2)\colon (\xi,\eta)\mapsto
(\Delta(\xi\eta)\xi,\Delta(\xi\eta)^{-1}\eta)$ be the Birkhoff Normal
Form of $\mathcal{F}^2$ associated to the orbit $(12)$, with
$\Delta=\Delta(\mathcal{D},1,2)\colon z\mapsto
\lambda+\sum_{\ell=1}^{+\infty} a_\ell z^\ell$. If $a_1 \neq 0$, then
\begin{itemize}
	\item the Birkhoff Normal Form $N$ is a \mlsinv{};
	\item the differential of the gluing map $\mathcal{G}$ at
	any point $(\xi,\eta)\in \Gamma_\infty$ is also a
	\mlsinv{}, where $\mathcal{G}=\mathcal{G}(\mathcal{D},1,2)$ and $\Gamma_\infty=\Gamma_\infty(\mathcal{D},1,2)$ are taken as in Subsection~\ref{subs prliemr}.
\end{itemize}
\end{corollary}

\begin{proof}
  %We assume that $(jk)=(12)$ and
 % use the notation of
  %Sections~\ref{section extension}-\ref{section marked lyapuno}.
   Recall that by Corollary~\ref{coroll detm inv}, the parameters
  $\{\ba_{\ell},\bg_{\ell},\bgg_{\ell}\}_{\ell\ge0}$ are \mlsinv{s},
  provided that $\ba_{1}\ne0$; by the above corollary
  (recall~\eqref{scaled coefff}) we thus conclude that
  $\{a_{\ell},\gamma_{\ell},g_\ell\}_{\ell \geq 0}$ are \mlsinv{s}, as
  well as the expressions of
  $\Delta\colon z\mapsto \lambda+\sum_{\ell=1}^{+\infty} a_\ell
  z^\ell$ and of the Birkhoff Normal Form
  $N=R_0\mathcal{F}^{2} R_0^{-1}$ in a neighborhood of $(s(2,1),0)$,
  but also of $\gamma$ and $g=\partial_2 G^-(\cdot,\gamma(\cdot))$.
  By~\eqref{expre differentielllle}, we deduce that the differential
  $D\mathcal{G}_{(\xi,\eta)}$ of the gluing map
  $\mathcal{G}=R \circ \mathcal{F}^{2} \circ R^{-1}|_{\Omega_\infty}$
  at any point $(\xi,\eta)\in \Gamma_\infty$ is also a \mlsinv{}, where
  $\mathcal{G},\Omega_\infty$ and $\Gamma_\infty$ are taken as in
  Subsection~\ref{subs prliemr}.
\end{proof}

\section{Reconstructing the geometry from the Marked Length
  Spectrum}\label{sec recov geom}

In this section, we assume that the billiard table $\mathcal{D}$ has additional symmetries, i.e., $\mathcal{D} \in \billiards_{\mathrm{sym}}$.
It follows from the previous part that
 if
$T:=\mathcal{F}^2\colon (s,r)\mapsto (s'',r'')$ denotes
the square of the billiard map
$\mathcal{F}=\mathcal{F}(\mathcal{D}) \colon (s,r) \mapsto
(s',r')$, then
under some twist condition, the Birkhoff Normal
Form $N=N(\mathcal{D},1,2)=RTR^{-1}$ of $T$ in a neighborhood of
$(s(2,1),0)$ is completely determined by the Marked Length Spectrum
$\mathcal{MLS}(\mathcal{D})$, assuming that $s(2,1)$ is  the arc-length
parameter of the point of $\mathcal{M}_2$ in the $2$-periodic orbit
$(12)$.

In this part, our goal is to see which information on the geometry of the billiard table $\mathcal{D}$ can be reconstructed,
and conclude the proof of our Main Theorem:
\begin{theorem}\label{prop symmetr rr}
  For an open and dense set of billiard tables
  $\mathcal{D} \in \billiards_{\mathrm{sym}}$, the Marked Length
  Spectrum $\mathcal{MLS}(\mathcal{D})$ determines completely the
  geometry of $\mathcal{D}$.
\end{theorem}

Fix a billiard table
$\mathcal{D} = \R^{2}\setminus \bigcup^{3}_{i = 1}\obs_{i} \in
\billiards_{\mathrm{sym}}$, and let
$\mathcal{F}:=\mathcal{F}(\mathcal{D})$ be the associated billiard
map.  After possibly applying some isometry, we assume that in the
plane with $(\bar x,\bar y)$-coordinates, the trace of the point in
the $2$-periodic orbit $(12)$ which is on the first obstacle,
resp. second obstacle, has coordinates $(-\frac 12 \ell, 0)$,
resp. $(\frac 12 \ell, 0)$, where
$\ell=\ell(\mathcal{D}):=\frac 12 \mathcal{L}(12)$ is the half-length
of the orbit $(12)$. In particular, the axis of symmetry is the
vertical axis
$\{\bar{x}=0\}$.

\subsection{A construction for symmetric billiard tables}

We consider the following construction. If $1,2,3$ are the labels of
the three obstacles of $\mathcal{D}$, we define a new billiard table
$\mathcal{D}^*=\mathcal{D}^*(\mathcal{D})$ formed by three obstacles
$1^*,2^*,3^*$ (see Figure~\ref{image z2 deux})

\begin{figure}[H]
	\begin{center}
		\includegraphics[trim = 1cm 1cm 2.5cm 1cm, width=14.5cm]{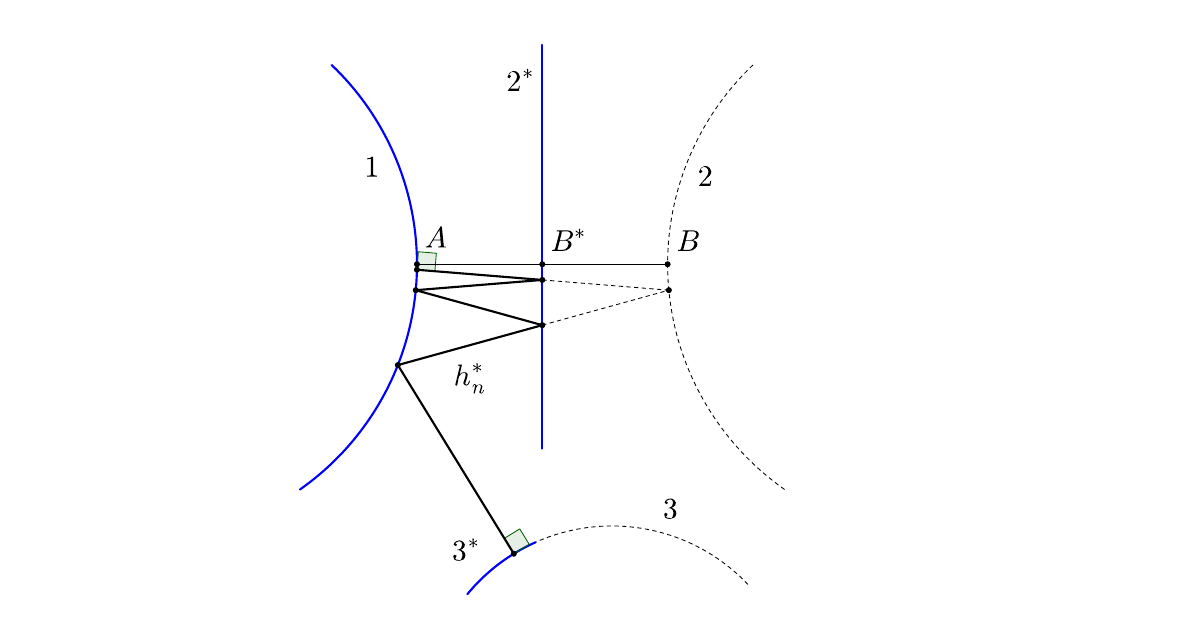}
		\caption{Defining a new table using  the $\Z_2$-symmetry of the pair $\{\obs_1,\obs_2\}$.}\label{image z2 deux}
	\end{center}
\end{figure}

\begin{itemize}
\item we consider a half-plane $\{\bar x \leq 0\}$ or
  $\{\bar x \geq 0\}$ such that it contains the trace of at least one
  of the two points $x_\infty^{(j)}(0)\in \{r=0\}$, $j \in\{1,2\}$, in
  the respective homoclinic orbits
  $h_\infty^{(1)}=(\dots212131212\dots)$ and
  $h_\infty^{(2)}=(\dots121232121\dots)$; in the following we assume
  that this point is $x_\infty(0)=x_\infty^{(1)}(0)$ and that its
  trace is in the half-plane $\{\bar x \leq 0\}$;
\item in this case, we let the obstacle with label $1^*$ in $\mathcal{D}^*$ be the same as the one of $\mathcal{D}$ with label $1$;
\item we define the obstacle $2^*$ as the vertical line segment $\{0\} \times [-\ell^* ,\ell^*]$ for some $\ell^*>0$ such that $\{0\} \times [-\ell^* ,\ell^*]$ does not cross the third obstacle, and the intersection of $\{0\} \times (-\ell^* ,\ell^*)$ and of the line segment between the points of parameters $x_\infty(1)= \mathcal{F}(x_\infty(0))$ and $x_\infty(2)= \mathcal{F}^2(x_\infty(0))$ is non-empty; we parametrize this line segment in arc-length in such a way that the image of the point in the $2$-periodic orbit $(1^*2^*)$ is associated to the parameter $0$;
\item %let $\mathcal{A}$  in the homoclinic trajectory $h_\infty$ defined in Subsection~\ref{section setting};
we let the obstacle $3^*$ be   some small arc in the obstacle $3$ such that it is in the half-plane $\{\bar{x}\leq0\}$ and contains a neighborhood of the point with coordinates $x_\infty(0)$.
\end{itemize}

By construction, the billiard table $\mathcal{D}^*$ satisfies the non-eclipse condition and is of the same type as $\mathcal{D}$, except that the obstacle $2^*$ is now flat.
Moreover, by Proposition~\ref{prop bdskl}, for any sufficiently large integer $n$,
each palindromic orbit $h_n=h_n^{(1)}=(31\underbrace{2121\dots21}_{2n})$ as above in $\mathcal{D}$ shadows   either $h_\infty=h_\infty^{(1)}$ or its image under $\mathcal{I}\colon(s,r)\mapsto (s,-r)$, and can thus be associated to a periodic orbit in $\mathcal{D}^*$, denoted by $h_n^*$, which is defined as follows:
\begin{itemize}
\item we start at the image of the point $x_n(0)\in \{r=0\}$ of $h_n$;  it is close to the image of $x_\infty(0)$ so it is indeed on the obstacle $3^*$;
\item the trace of the first orbit's  segment  is the same as for $h_n$;
\item the trace of the second orbit's segment is the first part of the trace of the second segment of $h_n$ which is contained in  the half-plane $\{\bar{x}\leq 0\}$;
\item the trace of the third orbit's segment is the second part of the second segment of $h_n$, which is contained in  $\{\bar{x}\geq 0\}$; it is also the continuation of the second segment of $h_n^*$ under the billiard flow of $\mathcal{D}^*$, after it gets reflected on $2^*$ according to the usual law of reflection of angles;
\item we repeat this folding procedure along the trajectory each time  we hit the axis $\{\bar{x}=0\}$ until we reach the image of the point $x_n(2n+2)$, so that the trace of the orbit of $h_n^*$ is contained in $\{\bar{x}\geq 0\}$.
\end{itemize}

The points of
$h_n^*$ which are on the boundary
$\partial\mathcal{D}^*$ of the new table still define an orbit under
the dynamics of the associated billiard map
$\mathcal{F}^*=\mathcal{F}(\mathcal{D}^*)\colon(s,r)\mapsto(s',r')$,
with the same length
$\mathcal{L}(h_n^*)=\mathcal{L}(h_n)$ as the original orbit
$h_n$.  The symbolic coding of $h_n^*$ is
\begin{align*}
h_n^*=(3^*1^*\underbrace{(2^*1^*)(2^*1^*)(2^*1^*)(2^*1^*)\dots
  (2^*1^*)(2^*1^*)}_{2 \times 2n=4n}),
\end{align*}
where each word $(2^*1^*)$ with even index replaces a $1$ and each word $(2^*1^*)$ with odd index replaces a $2$ in the previous coding. Formally, we obtain

\begin{lemma}\label{lemma symmm}
The maps $\billiards_{\mathrm{sym}}\ni \mathcal{D}\mapsto \mathcal{D}^*(\mathcal{D})$ and $h_n \mapsto h_n^*$ satisfy the following properties:
\begin{itemize}
\item there exists an integer $m_0\geq 0$ such that the subset of palindromic orbits $(h_{2m-1})_{m \geq m_0}$ of $\mathcal{F}$ embeds into the set of palindromic orbits of $\mathcal{F}^*$ by the map $h_n \mapsto h_n^*$ defined above;
\item for each $n=2m-1$, $m \geq m_0$, we have $\mathcal{L}(h_n)=\mathcal{L}(h_n^*)$;
\item for each $n=2m-1$, $m \geq m_0$, we have $\mathrm{LE}(h_n)=\mathrm{LE}(h_n^*)$.
\end{itemize}
\end{lemma}

\begin{proof}
The fact that each $h_n^*$ is palindromic follows from the preservation of angles under reflections.
It remains to show the third point about Lyapunov exponents.  Indeed, let  $x=(s,r)\in \mathcal{M}$ be a point of the orbit $h_n^*$ associated to a bounce on the obstacle $1^*$,  and set $(s^*,r^*):=\mathcal{F}^*(s,r)$,  $(s',r'):=(\mathcal{F}^*)^2(s,r)$. We also denote by $\mathscr{L}^*_1:=h(s,s^*)$, $\mathscr{L}^*_2:=h(s^*,s')$ the respective distances between the points of collision, and let $\mathcal{K}:=\mathcal{K}(s)$, $\mathcal{K}':=\mathcal{K}(s')$ be the respective curvatures at $s,s'$.\footnote{ Note that the curvature at $s^*$ vanishes, as this bounce is on the flat piece $2^*$.} Set $\nu:=\sqrt{1-r^2}$,  $\nu^*:=\sqrt{1-(r^*)^2}$, and $\nu':=\sqrt{1-(r')^2}$.
It follows from \eqref{matrice sl deux} that
\begin{equation*}
D(\mathcal{F}^*)^2_x=\begin{pmatrix}
\frac{1}{ \nu'}((\mathscr{L}^*_1+\mathscr{L}^*_2) \mathcal{K}+\nu) & \frac{\mathscr{L}^*_1+\mathscr{L}^*_2}{\nu \nu'}\\
(\mathscr{L}^*_1+\mathscr{L}^*_2) \mathcal{K}\mathcal{K}'+\mathcal{K} \nu'+\mathcal{K}'\nu & \frac{1}{\nu}((\mathscr{L}^*_1+\mathscr{L}^*_2) \mathcal{K}'+\nu')
\end{pmatrix}.
\end{equation*}
Note that $\mathscr{L}:=\mathscr{L}^*_1+\mathscr{L}^*_2$ is equal to the distance between the associated bounces on the initial billiard table $\mathcal{D}$, and then, the  matrix of $D(\mathcal{F}^*)^2_x$ is equal to  the matrix of $-D\mathcal{F}_x$. Therefore, each new collision created by the introduction of the auxiliary obstacle $2^*$ at a shorter distance does not affect the differential, nor the Lyapunov exponent, as $D_x(\mathcal{F}^*)^{4n+2}=D_x \mathcal{F}^{2n+2}$.
\end{proof}

As a consequence of the previous observations, we obtain:
%The following result concludes the proof of Theorem~\ref{prop symmetr}.
\begin{prop}\label{corollary symmetric table}
  We consider a billiard table
  $\mathcal{D} \in \billiards_{\mathrm{sym}}$, with Marked Length
  Spectrum $\mathcal{MLS}(\mathcal{D})$. We let
  $\mathcal{D}^*=\mathcal{D}^*(\mathcal{D})$ be the billiard table
  defined above and denote by
  $\mathcal{F}^*=\mathcal{F}^*(\mathcal{D}^*)$ the associated billard
  map. Let
  $N^*=N^*(\mathcal{D}^*,1,2)\colon (\xi,\eta) \mapsto
  (\Delta^*(\xi\eta)\xi,\Delta^*(\xi\eta)^{-1}\eta)$ be the Birkhoff
  Normal Form of $T^*:=(\mathcal{F}^*)^2\colon (s,r)\mapsto (s'',r'')$ in a neighborhood of the
  point $(0_{1^*},0)$ in the period two orbit $(1^*2^*)$ which is on
  the first obstacle $1^*$, with
  $\Delta^*\colon z \mapsto \sum_{j=0}^{+\infty}a_j z^j$. If
  $a_1 \neq 0$, then $\mathcal{MLS}(\mathcal{D})$ determines the
  Birkhoff Normal Form $N^*$.
\end{prop}

\begin{proof}
  There is a new $2$-periodic orbit $(1^*2^*)$ for the map
  $\mathcal{F}^*$ which bounces perpendicularly at the points with
  $(\bar{x},\bar{y})$-coordinates $(-\frac 12 \ell,0)$ and
  $(0,0)$. Moreover, the image $h_\infty^*$ in $\mathcal{D}^*$ of the
  homoclinic trajectory $h_\infty$ in $\mathcal{D}$ can be defined
  following the same "folding" procedure as above. It is also
  homoclinic to $(1^*2^*)$, and similarly, it is accumulated by the
  orbits $(h_n^*)_{n}$ defined above. The point $(0_{1^*},0)$ with
  trace $(-\frac 12 \ell,0)$ is a saddle fixed point for the dynamics
  of $T^*=(\mathcal{F}^*)^2$, hence we may consider the Birkhoff
  Normal Form
  $N^*\colon (\xi,\eta) \mapsto
  (\Delta^*(\xi\eta)\xi,\Delta^*(\xi\eta)^{-1}\eta)$ of $T^*$ in a
  neighborhood of this point. The orbits $(h_{n}^*)_{n}$ are still
  palindromic, hence the analogue of Lemma~\ref{claim eta n delta n xi
    n} and Lemma~\ref{claim deuxis} remains true is this case. We also
  note that the homoclinic parameter $\xi_\infty$ is preserved by the
  unfolding construction.  By Lemma~\ref{lemma symmm}, the Lyapunov
  exponent of each orbit $h_n^*$ is
  $\mathcal{MLS}(\mathcal{D})$-invariant. Therefore, if $a_1 \neq 0$,
  then by the same method as in Lemmata~\ref{lemme exposant de
    lya},~\ref{corollar dofe},~\ref{prop structure du dev},~\ref{lemme
    utile mochomoge},~\ref{lemma exp lya detailll},
  Corollary~\ref{coroll detm inv} and Corollary~\ref{premier cororrr},
  we can recover the Birkhoff invariants of $N^*$ by considering the
  series expansion of $\mathrm{LE}(h_n^*)$ with respect to $n$. As a
  result, the Birkhoff Normal Form $N^*$ is entirely determined by the
  Marked Length Spectrum $\mathcal{MLS}(\mathcal{D})$ of the initial
  table.
\end{proof}

Let  $(0_{2^*},0)$ be the $(s,r)$-coordinates of the point in the orbit $(1^*2^*)$ whose trace is on the second  obstacle $2^*$. The Birkhoff Normal Forms of $(\mathcal{F}^*)^2$ at the two points $(0_{1^*},0)$ and $(0_{2^*},0)$ in the orbit $(1^*2^*)$ coincide:
\begin{lemma}\label{lemma seconde conk}
Let $\mathcal{D}\in \billiards_{\mathrm{sym}}$ and let $\mathcal{F}^*=\mathcal{F}^*(\mathcal{D}^*)$.  The Birkhoff  Normal Form of $T^*=(\mathcal{F}^*)^2$ in a neighborhood of the point $\mathcal{F}^*(0_{1^*},0)=(0_{2^*},0)$ % in the orbit $(1^*2^*)$ which is on the   obstacle $2'$
coincides with the map $N^*=N^*(\mathcal{D}^*,1,2)$ defined in Proposition~\ref{corollary symmetric table}.
\end{lemma}

\begin{proof}
Let $\mathcal{U}^*\subset \R^2$ be an open neighborhood of $(0_{1^*},0)$, and let $R^* \colon \mathcal{U}^*\to \R^2$ be a conjugacy map such that $R^*T^*|_{\mathcal{U}^*}=N^*R^*|_{\mathcal{U}^*}$. Then $\mathcal{F}^*(\mathcal{U}^*)$ is an open neighborhood of $(0_{2^*},0)$, since $\mathcal{F}^*(0_{1^*},0)=(0_{2^*},0)$. The map $\widetilde{R}^*:=R^* \circ (\mathcal{F}^*)^{-1}$ is symplectic, and for any $y=\mathcal{F}^*(x)$ with $x \in \mathcal{U}^*$, it holds
$$
\widetilde{R}^*\circ T^*(y)= R^*\circ (\mathcal{F}^*)^{-1} \circ  (\mathcal{F}^*)^{2}(\mathcal{F}^*(x))=R^* \circ T^*(x)=N^* \circ R^*(x)=N^* \circ \widetilde{R}^*(y),
$$
which concludes, by uniqueness of the Birkhoff Normal Form.
\end{proof}

\subsection{Recovering the geometry of a symmetric billiard table}

In the following, we use the same notation as in the last part.
Given $\mathcal{D}\in \billiards_{\mathrm{sym}}$, then by definition, after rotation by an angle of $-\frac \pi 2$, near the point $(0,\frac \ell 2)$, the first obstacle $1$ (which is the same as the obstacle $1^*$) can be represented as a graph
$$
\mathscr{C}=\left\{\Big(t,\frac \ell 2+\beta_2 t^2+\beta_4 t^4+\dots\Big): t \in I\right\},
$$
for some open interval $I \ni 0$. Indeed, this follows from our assumption that the obstacles $\obs_1,\obs_2$ have some axial symmetry with respect to the trace of the $2$-periodic orbit $(12)$, and then, there are only even coefficients in the above expansion.

\begin{figure}[H]
\begin{center}
   \includegraphics [width=13.3cm]{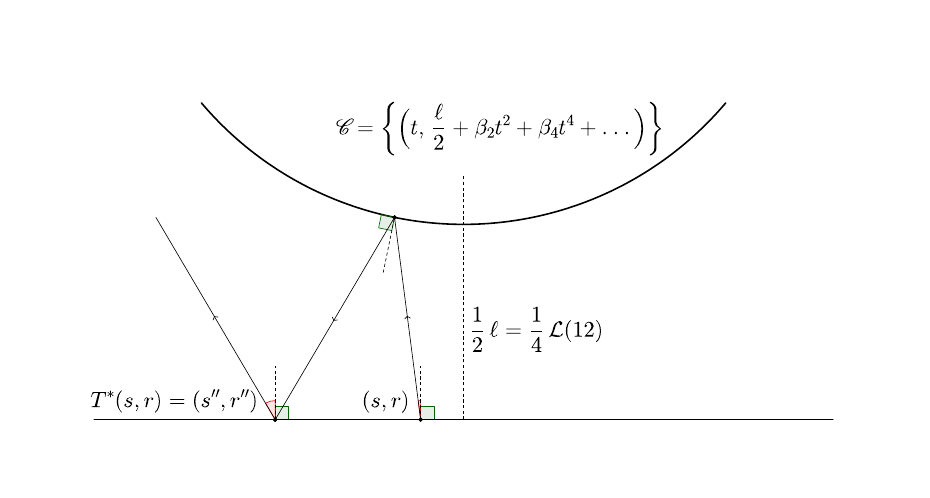}
   \caption{The map $T^*$ near  the point $(0_{2^*},0)$.}\label{image z2}
\end{center}
\end{figure}

Let us recall the following result in the paper~\cite{CdV3} of Colin de
Verdi\`ere:

\begin{lemma}[{\cite[Lemma 1]{CdV3}}]\label{leme de verdeire}
  %Let $\widetilde{T}\colon (s,r)\mapsto(s'',r'')$ be the image of the map $T$ under the conjugacy $(s,\varphi)\mapsto(s,r):=(s,\sin(\varphi))$.  Then t
  The jet of %$\widetilde{T}$
  %(thus also the jet of
  $T^*\colon (s,r)\mapsto(s'',r'')$ at $(0_{2^*},0)$ is in one-to-one
  correspondence with the coefficients $(\beta_{2 k})_{k \geq 1}$ of
  the graph $\mathscr{C}$ defined above.  Besides, the linear part of
  $D T^*_{(0_{2^*},0)}$ is associated to the hyperbolic
  matrix
  \begin{align*}
\begin{pmatrix}
A-1 & -A\\
2-A & A-1
\end{pmatrix}\in \mathrm{SL}(2,\R),\quad A:=2(2\beta_2+1)>2.
  \end{align*}
  Indeed, by the strong convexity of $\obs_1$, we have
  $\beta_2>0$. More precisely, for $k \geq 1$, and for some vector
  $v_0\in \R^2 \setminus \{(0,0)\}$, it holds
\begin{align*}
  T^*_{(2k+1)}(s,r)=T^{*,(0)}_{(2k+1)}(s,r)+(s-r)^{2k+1} \beta_{2(k+1)}v_0+O(|s|+|r|)^{2k+2},
\end{align*}
where $T_{(2k+1)}^{*,(0)}$ denotes the jet of $T^*$ of order $2k+1$ at
$(0_{2^*},0)$ for $\beta_{2(k+1)}=0$.
\end{lemma}

\begin{lemma}\label{non vanish}
  For any billiard table $\mathcal{D}\in \billiards_{\mathrm{sym}}$,
  the first Birkhoff invariant $a_1=a_1(\mathcal{D})$ of the Birkhoff
  Normal Form $N^*=N^*(\mathcal{D}^*,1,2)$ satisfies
	\begin{equation}\label{eq definition c etoile f etoile}
	a_1=c^* \mathcal{K}''+f^*(\ell,\mathcal{K}),
	\end{equation}
	for some constant $c^*\neq 0$ and some continuous function
    $f^*\colon \R^2 \to \R$, where $\mathcal{K}$, $\mathcal{K}''$
    respectively denote the curvature and its second derivative at the
    bouncing points of the $2$-periodic orbit $(12)$.

	In particular, for any $r>0$, and for an open and dense set of
    billiard tables $\mathcal{D}\in \billiards_{\mathrm{sym}}(3,r)$,
    $\mathcal{D}$ satisfies the non-degeneracy condition
	\begin{equation}\tag{$\star$}
	a_1(\mathcal{D})\neq 0.
	\end{equation}
\end{lemma}

\begin{proof}
 Fix $r>0$. For each integer $j \geq 0$, we define a map $a_j \colon \billiards_{\mathrm{sym}}(3,r)\to \R$ in such a way that for all $\mathcal{D}\in \billiards_{\mathrm{sym}}(3,r)$, the $j^{\mathrm{th}}$ Birkhoff invariant of   $N^*=N^*(\mathcal{D}^*,1,2)$ is equal to $a_j(\mathcal{D})$, i.e., $N^*\colon (\xi,\eta) \mapsto
(\Delta^*(\xi\eta)\xi,\Delta^*(\xi\eta)^{-1}\eta)$, with $\Delta^*\colon z \mapsto  \sum_{j=0}^{+\infty}a_j(\mathcal{D}) z^j$.
In particular, with the topology introduced in Definition~\ref{topology}, and by Cauchy's
integral formula, the map $a_j$ is continuous. \color{black}
%Note that we can see each coefficient as a map from $ \billiards_{\mathrm{sym}}(1,2)$ to $\R$.

As for the Birkhoff Normal Form $N(\mathcal{D},1,2)$, the coefficient
$\lambda:=a_0(\mathcal{D})$  is related to the Lyapunov exponent
$\mathrm{LE}(1^*2^*)$ and  only depends on $\mathcal{L}(12)$ and on
the curvature $\mathcal{K}$ at the bouncing points of the $2$-periodic orbit $(12)$.

Besides, by the construction of the Birkhoff Normal Form (see \eg~\cite{Bi,Ste,Mos}),
the first coefficient $a_1:=a_1(\mathcal{D})$ of $N^*$ is determined  by the jet of
order three of  $T^*$. Together with Lemma~\ref{leme de verdeire}, we thus have
$$
a_1=c_0^* \beta_4+f_0^*(\ell,\beta_2),
$$
for some constant $c_0^*\neq 0$ and some continuous function $f_0^*\colon \R^2 \to \R$.
Equivalently, $\beta_2,\beta_4$ can  be interpreted in terms of the  curvature $\mathcal{K}
$ and its second derivative $\mathcal{K}''$  at the bouncing points of the $2$-periodic orbit
$(12)$,\footnote{ The  first derivative  of the curvature vanishes due to the symmetries of the table.} as
$$
\mathcal{K}=2\beta_2,\qquad \mathcal{K}''=24(\beta_4-\beta_2^3),
$$
which gives~\eqref{eq definition c etoile f etoile}.

According to ~\eqref{eq definition c etoile f etoile}, it is therefore possible to make the first Birkhoff invariant $a_1$ non-zero by modifying the shape of the obstacles $\obs_1,\obs_2$ so as to change the value of $\mathcal{K}''$, but keeping $\ell$, $\mathcal{K}$ fixed. Let us now reformulate it in terms of the topology on $\billiards_{\mathrm{sym}}(3,r)$ introduced in Definition~\ref{topology}. After possibly applying some isometry, we have $\obs_1=\obs(f)$, with $f \in C_r^\omega(\T,\R^2)$, $\theta \mapsto \varrho(\theta)(\cos(\theta),\sin(\theta))$, for some  even\footnote{ Due to  the $\Z_2$-symmetry of $\obs_1$.} function $\varrho\in C_r^\omega(\T,\R)$,    $\theta \mapsto \sum_{j=0}^{+\infty}\hat \varrho_j e^{\mathrm{i}j \theta}$. The curvature $\mathcal{K}=\mathcal{K}(0)$ and its second derivative $\mathcal{K}''=\mathcal{K}''(0)$ satisfy
\begin{align*}
	\mathcal{K}&=\frac{1}{\varrho(0)}-\frac{\varrho''(0)}{\varrho^2(0)},\\
	%\frac{r(0)-r''(0)}{r^2(0)},\\
	\mathcal{K}''&=-\frac{\varrho''(0)}{\varrho^2(0)}+\frac{3(\varrho''(0))^2}{\varrho^3(0)}+\frac{3(\varrho''(0))^3}{\varrho^4(0)}-\frac{\varrho''''(0)}{\varrho^2(0)},
	%\frac{2r(0)r''(0)+3(r''(0))^2-r(0)r''''(0)}{r^3(0)}-\frac{3(r^2(0)-r(0)r''(0))(r(0)r''(0)+(r''(0))^2)}{r^5(0)}.
\end{align*}
with $\varrho(0)=\sum_{j=0}^{+\infty} \hat \varrho_j$,
$\varrho''(0)=-\sum_{j=0}^{+\infty} j^2\hat \varrho_j$, and
$\varrho''''(0)=\sum_{j=0}^{+\infty} j^4\hat \varrho_j$. For any table
$\mathcal{D}$ such that $a_1(\mathcal{D})$ vanishes, it is then
sufficient to perturb the first Fourier coefficients of $\varrho$ to
get a new function $\tilde \varrho \in C_r^\omega(\T,\R)$ such that
the associated table $\widetilde{\mathcal{D}}$ satisfies
$a_1(\widetilde{\mathcal{D}}) \neq 0$.  In this way, we see that for
the topology introduced in Definition~\ref{topology}, the condition
$a_1 \neq 0$ holds for a dense subset of
$\billiards_{\mathrm{sym}}(3,r)$, and clearly, this condition is also
open, as the map $a_1 \colon \billiards_{\mathrm{sym}}(3,r) \to \R$ is
continuous.
\end{proof}

By Lemma~\ref{non vanish}, in order to prove Theorem~\ref{prop symmetr
  rr}, it is sufficient to show that the Marked Length Spectrum
determines the geometry for the set of billiard tables in
$\billiards_{\mathrm{sym}}$ such that the first invariant $a_1$ is non-zero. In the
following, we fix a table
$\mathcal{D}\in \billiards_{\mathrm{sym}}$ satisfying the non-degeneracy condition ($\star$) and %
show that the geometry of $\mathcal{D}$ is determined by the Marked Length Spectrum
$\mathcal{MLS}:=\mathcal{MLS}(\mathcal{D})$.

%We are now able to conclude the proof of Theorem~\ref{prop symmetr rr}.  Let us first show the following:
%go back to the initial billiard map $\mathcal{F}=\mathcal{F}(\mathcal{D})$ and its square $T=\mathcal{F}^2$.
\begin{corollary}\label{cororororo}
The coefficients $(\beta_{2k})_{k \geq 1}$ of the graph $\mathscr{C}$ are \mlsinv{s}. Therefore, by analyticity, the geometry of $\obs_1,\obs_2$ can be reconstructed from  $\mathcal{MLS}$.
\end{corollary}

\begin{proof}
By ($\star$), Proposition~\ref{corollary symmetric table} and Lemma~\ref{lemma seconde conk}, the Birkhoff Normal Form $N^*$ of the map $T^*$ in a neighborhood of the point $(0_{2^*},0)$ is determined by the Marked Length Spectrum $\mathcal{MLS}$. By the construction of the Normal Form  given by Moser~\cite{Mos}   (see in particular the equations $(3.2)$, $(3.3)$ and $(3.4)$ on pp. 680--681), the Birkhoff invariants  are determined inductively by the jet of $T^*$. More precisely, for each $k_1 \geq 1$, the coefficients $(a_k(\mathcal{D}))_{1 \leq k \leq k_1}$ of $N^*$ are related to the $(2k_1+1)^{\text{th}}$ jet of $T^*$  by some invertible triangular system, and thus,  there is a one-to-one correspondence between the jets of $T^*$ and $N^*$ at the point $(0,0)$.  By the previous discussion, we deduce that the jet of $T^*$ at the point $(0_{2^*},0)$ is determined by $\mathcal{MLS}$. Now, by   Lemma~\ref{leme de verdeire}, the jet of $T^*$ at the point $(0_{2^*},0)$ is also in one-to-one correspondence with the   coefficients $(\beta_{2k})_{k \geq 1}$. Therefore, the coefficients $(\beta_{2k})_{k \geq 1}$ can be recovered from the Marked Length Spectrum, which concludes the proof.
\end{proof}

To conclude the proof of  Theorem~\ref{prop symmetr rr},  it remains to show that the geometry of the third scatterer can also be recovered. While the auxiliary table $\mathcal{D}^*$ and the associated Birkhoff Normal Form $N^*$ were useful to determine the geometry of $\obs_1,\obs_2$, now, we focus again on the initial billiard table  $\mathcal{D}$. We denote by $\mathcal{F}=\mathcal{F}(\mathcal{D})$ and $T:=\mathcal{F}^2$ the billiard map of $\mathcal{D}$ and its square, let $N=N(\mathcal{D},1,2)$ be the Birkhoff Normal Form of $\mathcal{F}^2$ associated to the $2$-periodic orbit $(12)$, and assume that the first Birkhoff invariant of $N$ is non-zero.

\begin{corollary}\label{cororororobibis}
	The geometry of $\obs_3$ can be reconstructed from  $\mathcal{MLS}$.
\end{corollary}

\begin{proof}
  In the following, we use the notation introduced in
  Subsections~\ref{subs prliemr}-\ref{subsection lyapunov exonent
    exp}.  By Corollary~\ref{corollary retrouver les invariants}, the
  Marked Length Spectrum $\mathcal{MLS}$ determines the function
  $\gamma$ and the differential $D\mathcal{G}_{(\xi,\eta)}$ of the
  gluing map $\mathcal{G}=R_-\circ T \circ R_+^{-1}|_{\Omega_\infty}$,
  at any point $(\xi,\eta)\in \Gamma_\infty$. Restricted to
  $\mathscr{O}_\infty:=R_+^{-1} (\Omega_\infty)$,
  resp. $T(\mathscr{O}_\infty)$, we have $R_+=R_{m_0}$,
  resp. $R_-=R_{-m_0}$, for some integer $m_0 \geq 1$, where
  $R_{\pm m_0}:=N^{\pm m_0} R_0 T^{\mp m_0}$, and $R_0$ is the
  canonical conjugacy map given by Lemma~\ref{conjug canonique}. By
  Corollary~\ref{corollary retrouver les invariants}, the map $N$ is a
  \mlsinv. Similarly, the map $R_0$ is also a \mlsinv, as it only
  depends on the obstacles $\obs_1,\obs_2$ whose geometry is also
  known, by Corollary~\ref{cororororo}. Besides, in the definition of
  $R_{\pm m_0}$, we take iterates of $T$ between the first two
  obstacles, so their expression is also known. In other words,
  restricted to $\mathscr{O}_\infty$, resp. $T(\mathscr{O}_\infty)$,
  the map $R_+=R_{m_0}$, resp. $R_-=R_{-m_0}$ is determined by
  $\mathcal{MLS}$, as it only depends on the obstacles $\obs_1,\obs_2$
  whose geometry is a \mlsinv{}. Since $\gamma$ and $R_+$ are
  \mlsinv{s}, then the arc
  $\mathscr{A}_\infty:=R_+^{-1}(\Gamma_\infty)=\{R_+^{-1}(\eta,\xi_\infty+\gamma(\eta)):|\eta|\text{
    small\,}\}$ can be recovered. Moreover, we know the differential
  $D\mathcal{G}_{(\xi,\eta)}$ at any point
  $(\xi,\eta)\in \Gamma_\infty$, with
  $\mathcal{G}=R_-\circ T \circ R_+^{-1}|_{\Omega_\infty}$, hence we
  can determine $D\mathcal{F}^2_x$, for any
  $x=(s,r)=R^{-1}(\eta,\xi_\infty+\gamma(\eta)) \in
  \mathscr{A}_\infty$, with $|\eta|$ small. By definition, given any
  such point $x$, we have $\mathcal{F}(x)=x'=(s',0)$ for some
  parameter $s' \in \R$, and $\mathcal{F}^2(x)=(s,-r)$, by
  Lemma~\ref{claim petitt}. Let us denote by
  $\mathcal{K}(s),\mathcal{K}(s')$ the respective curvatures at the
  points $x=(s,r)$, $x'=(s',0)$, by $\mathscr{L}:=h(s,s')$ the length
  of the line segment between their traces on the table, and set
  $\nu:=\sqrt{1-r^2}$.  By~\eqref{matrice sl deux}, we have
	\begin{align*}
		D\mathcal{F}^2_x=\begin{pmatrix}
			\frac{2aa'}{\nu}-1 & \frac{2a'\mathscr{L}}{\nu^2}\\
			\frac{2a}{\mathscr{L}}(aa'-\nu) & \frac{2aa'}{\nu}-1
		\end{pmatrix},
		\end{align*}
	where
    \begin{align*}
      a:=\mathscr{L} \mathcal{K}(s)+\nu,\qquad a':=\mathscr{L} \mathcal{K}(s')+1.
    \end{align*}
	The values of $\mathcal{K}(s)$ and $\nu$ are already known, thus
    by considering the first line in the expression of this
    differential, we deduce the value of $aa'$ and $a'\mathscr{L}$,
    and then, of $\mathscr{L}$ and $\mathcal{K}(s')$.  In particular,
    the geometry of $\obs_3$ is completely determined: for any such
    $x \in \mathscr{A}_\infty$, we draw a line segment of length
    $\mathscr{L}$ starting from the associated point in the table,
    such that the angle of this segment with the normal to
    $\partial \obs_2$ at this point is equal to
    $\varphi=\arcsin(r)$. Then, the endpoint belongs to
    $\partial \obs_3$, and the curvature at this point is equal to
    $\mathcal{K}(s')$. In particular, we recover the geometry of an
    arc in $\obs_3$, hence the third obstacle $\obs_3$ is entirely
    determined by $\mathcal{MLS}$, by analyticity.

\begin{figure}[H]
	\begin{center}
		\includegraphics [trim = 1cm 2cm 1cm 1cm, width=20cm]{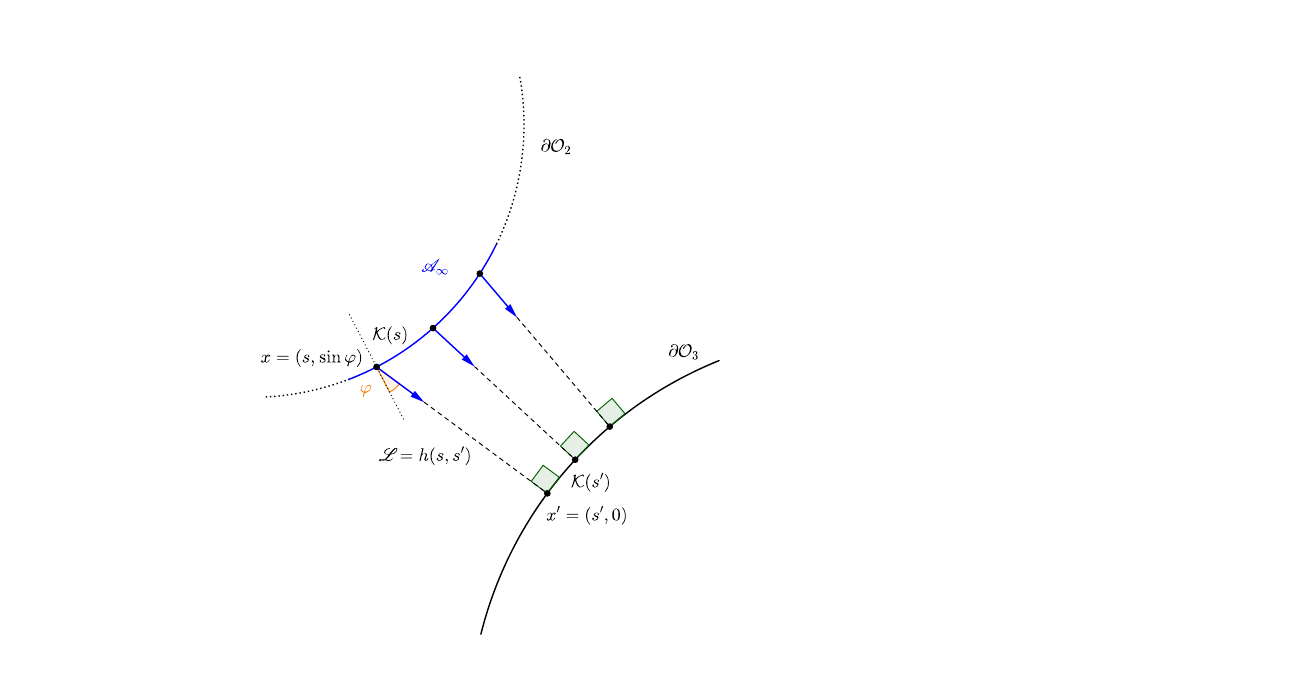}
		\caption{Reconstructing the geometry of the third obstacle $\mathcal{O}_3$.}\label{image fig fin}
	\end{center}
\end{figure}
 \end{proof}

 \begin{remark}
   We can also provide an alternative proof of the above result.
   Since $N$ and $\gamma$ are \mlsinv{s}, then for any sufficiently
   large integer $n=2m-1\geq 0$, we can compute the coordinates
   $(\xi_n,\eta_n)=R_-(x_n(1))$ of the associated point in the orbit
   $h_n$. The conjugacy map $R_-$ is entirely determined by the
   obstacles $\obs_1,\obs_2$, whose geometry is known, thus we can
   recover the coordinates of $x_n(1)=(s_n(1),r_n(1))$, and also of
   the other points in this orbit which are on
   $\partial\obs_1,\partial\obs_2$, i.e., $x_n(k)$, for
   $k=2,\cdots,2n+1$. Let $\mathscr{L}_n$ be the length of the line
   segment connecting the points of parameters $x_n(0)$ and
   $x_n(1)$. Then $\mathscr{L}_n$ is determined by $\mathcal{MLS}$:
   indeed, we know the total length $\mathcal{L}(h_n)$, as well as the
   length of the other orbit segments in $h_n$, as they are associated
   to the points $x_n(k)$, for $k=1,\cdots,2n+1$.  Therefore, starting
   from the trace of the point $x_n(1)$, then the endpoint of the
   outward line segment of length $\mathscr{L}_n$ based at this point
   and making an angle $-\varphi_n(1)=-\arcsin(r_n(1))$ with the
   normal to $\partial \obs_2$ gives a point on $\partial\obs_3$,
   associated to the parameter $x_n(0)$. For different values of $n$,
   the corresponding points are pairwise distinct (they have different
   periods and all start perpendicularly to $\partial \obs_3$), and
   since they accumulate to the trace of the homoclinic point
   $x_\infty(0)$ as $n\to +\infty$, this completely determines the
   geometry of $\obs_3$, again by analyticity.

 \end{remark}

\section{Conclusions and further questions}
In this article we showed that for a generic class of chaotic
billiards obtained by removing from the plane $m \geq 3$ convex
analytic scatterers with some symmetries, the Marked Length Spectrum
determines the
geometry of the billiard.\\

Our result leads to a number of natural questions:
\begin{question}
  Is it possible to remove the assumption about the mirror symmetry
  between $\mathcal{O}_{1}$ and $\mathcal{O}_{2}$?
\end{question}
\begin{question}
  Is it possible to remove the assumption about each of the scatterers
  $\mathcal{O}_1$ and $\mathcal{O}_{2}$ being symmetric around the
  period-two collision?
\end{question}
This would be a major step (similar to the one leading from~\cite{Z1}
to~\cite{Z2,Z3}), and the techniques involving classical Birkhoff
normal forms seem inadequate to this task.
\begin{question}
  Is it possible to obtain similar results for the unmarked Length
  Spectrum?
\end{question}
As in~\cite{Z2,Z3}, one could ask if simply marking the length of the
$2$-periodic orbit, of the associated Lyapunov exponent and possibly
requiring some non-degeneracy of the Spectrum for those values, would
suffice to recover the coefficients that describe the dynamics in the
Birkhoff coordinates.  It is unclear to us if such an approach could
be carried out successfully, since it is not easy to distinguish
between the many homoclinic orbits that accumulate on the periodic
orbit, and our strategy hinges on a very fine asymptotic analysis of a
specific family of approximating homoclinic orbits.  We ultimately
believe that such an approach should be possible, but for sure our
strategy would have to be modified to deal with all such homoclinic
orbits at the same time.

Let us also note that several quantities are length-spectral
invariants. For instance, any periodic orbit of period at least three
bounces on the three scatterers; as $2$-periodic orbits correspond to
minimizers of the distance between two scatterers, the two smallest
elements in the Length Spectrum are the lengths of two $2$-periodic
orbits.

Recall that the topological entropy of the subshift of finite type
described in the first part is equal to $\log(m-1)$, where $m \geq 3$
is the number of obstacles. By \cite{Mor2}, this number is
also %we expect $N(p)$ to grow as $(m-1)^p$, hence the number of
%obstacles to be
a length-spectral invariant.

\color{black}
\appendix
\section{The Marked Lyapunov Spectrum}\label{sec:bdkl-symmetric-case}
In this appendix we collect some rather technical results that are
needed for this paper.  Stronger versions of such results have been
stated in the paper~\cite{BDSKL}, but their proofs appear to be
incorrect.  In this work we only need results in the (much simpler)
situation in which the scatterers are symmetric: we therefore state
and prove such results here, in order for the arguments in this paper
to be complete.

\begin{theorem}\label{prp:palindromic_lyapunov}
  The Lyapunov exponent of any palindromic periodic orbit in any
  billiard table in $\billiards(m)$ is a \mlsinv{}.
\end{theorem}
\begin{lemma}\label{lem:quadratic-form}
  Let $(\bar s_{0},\cdots,\bar s_{p-1})$ denote the collision points
  of a palindromic periodic orbit of (least) period $p = 2q$, so that
  $\bar s_{i} = \bar s_{p-i}$ for any $0 < i < p$.  For $n \geq 1$,
  define
  $L_n(s_{0},s_{1},\cdots,s_{n}) = \sum_{j =
    0}^{n-1}h(s_{j},s_{j+1})$.  Then
  \begin{align*}
    L_{q}(s_{0},\cdots,s_{q}) - %\\
    L_{q}(\bar s_{0},\cdots,\bar s_{q}) &= %\\
    Q_{q}(s_{0}-\bar s_{0},\cdots, s_{q}-\bar s_{q}) + R_{q},
  \end{align*}
  where $Q_{q}$ is a positive definite quadratic form and $R_{q}$ is a
  remainder term that satisfies the estimate:
  \begin{align*}
    R_{q} = O(\|(s_{0}-\bar s_{0},\cdots, s_{q}-\bar s_{q})\|^{3}).
  \end{align*}
\end{lemma}
\begin{proof}
  Since $(\bar s_0,\cdots,\bar s_{p-1})$ is palindromic,
  $L_{q}(s_{0},\cdots,s_{q})$ has a critical point (in fact, a
  minimum) at $(\bar s_{0},\cdots,\bar s_{q})$; recall in fact that
  the periodic orbit has necessarily orthogonal collisions at the
  points $s_{0}$ and $s_{q}$.  This amounts to say
  that $\partial_{j}L_{q}(\bar s_{0},\cdots,\bar s_{q}) = 0$ for any
  $j = 0,\cdots,q$.  Hence the lemma follows from Taylor's formula,
  provided that we show that the Hessian of $L_{q}$ at
  $(\bar s_{0},\cdots,\bar s_{q})$ is positive definite.\\

  It is immediate from the definition of $L_{q}$ that
  $\partial_{ij}L_{q} = 0$ if $|i-j| > 1$; in other terms, the Hessian
  of $L_{q}$ is a tridiagonal (symmetric) matrix.  For notational
  convenience, let $h^{ij} = h(\bar{s}_{i},\bar s_{j})$; let
  $\bar\varphi_{j}\in[-\frac \pi 2,\frac \pi 2]$ denote the angle formed by the outgoing
  trajectory at the $j$-th collision point and the unit normal vector
  to the domain at $\bar s_{j}$; finally, let $\mathcal{K}_{j}$ denote
  reciprocal of the radius of curvature at the point $\bar s_{j}$.
  Recall that by convention $\mathcal{K}_{j} < 0$ for any $j$.

  A direct computation shows that the diagonal terms are given by:
\begin{align*}
  \partial_{00} L_{q} &= \frac1{h^{01}}\cos^{2}\bar\varphi_{0}-\mathcal{K}_{0}\cos\bar\varphi_{0}\\
\partial_{jj} L_{q} &= \left[\frac1{h^{j-1 j}}+\frac1{h^{j
                      j+1}}\right]\cos^{2}\bar\varphi_{j}-2\mathcal{K}_{j}\cos\bar\varphi_{j}\text{
                      for } 0 < j < q\\
\partial_{qq} L_{q} &= \frac1{h^{q-1q}}\cos^{2}\bar\varphi_{q}-\mathcal{K}_{q}\cos\bar\varphi_{q},
\intertext{while the off-diagonal terms are given by:}
  \partial_{jj+1}L_{q} &= \frac1{h^{jj+1}}\cos\bar\varphi_{j}\cos\bar\varphi_{j+1}.
\end{align*}
Using the above expressions it is simple to prove the following lemma.
\begin{lemma}\label{lem:positive-definite}
  For $0 \leq n \leq q$, let $f_{n}$ denote the determinant of the
  $(n+1)\times (n+1)$ top-left minor of $\partial_{ij}L_q$.  Then
  $f_{n} > 0$ for all $0 \leq n\leq q$.
\end{lemma}
The above lemma implies that all eigenvalues of $\partial_{ij}L_{q}$
are positive, which completes the proof of our result.
\end{proof}
\begin{proof}[Proof of Lemma~\ref{lem:positive-definite}]
  We will in fact prove a slightly stronger statement which holds for
  $0 \leq n < q$, namely:
  \begin{align}\label{eq:f_n-f_n-1frac}
    f_{n} > f_{n-1}\frac{\cos^{2}\bar\varphi_{n}}{h^{n n+1}};
  \end{align}
  The proof will follow by induction.  Since the matrix is
  tridiagonal, it is known that the determinants $f_{n}$ can be
  expressed by the following recursive relation:
  \begin{align*}
    f_{n} = \partial_{nn}L_{q}f_{n-1}-(\partial_{n-1n}L_{q})^{2}f_{n-2},
  \end{align*}
  with the convention $f_{-1} = 1$ and $f_{-2} = 0$.\\
  The base case is $f_{0} = \partial_{00}L_{q}$, for
  which~\eqref{eq:f_n-f_n-1frac} immediately holds.  Assuming by
  inductive hypothesis that~\eqref{eq:f_n-f_n-1frac} holds for $n-1$,
  let us prove it for $n$.  The recursive relation yields, for
  $n < q$:
  \begin{align*}
    f_{n}-f_{n-1}\frac {\cos^{2}\bar\varphi_{n}}{h^{nn+1}} &= f_{n-1}\left(\partial_{nn}L_{q}-\frac{\cos^{2}\bar\varphi_{n}}{h^{nn+1}}\right) -(\partial_{n-1n}L_{q})^{2}f_{n-2} = \\
                                                       &=
                                                         f_{n-1}\left(\frac{\cos^{2}\bar\varphi_{n}}{h^{n-1n}}-2\mathcal{K}_{n}\cos\bar\varphi_{n}\right)
                                                         -\left(\frac{\cos\bar\varphi_{n-1}\cos\bar\varphi_{n}}{h^{n-1
                                                         n}}\right)^{2}f_{n-2} >  \\
                                                       &>
                                                         f_{n-2}\frac{\cos^{2}\bar\varphi_{n-1}}{h^{n-1n}}\frac{\cos^{2}\bar\varphi_{n}}{h^{n-1n}}
                                                         -2f_{n-1}\mathcal{K}_{n}\cos\bar\varphi_{n}-\left(\frac{\cos\bar\varphi_{n-1}\cos\bar\varphi_{n}}{h^{n-1n}}\right)^{2}f_{n-2} > \\
&> -2f_{n-1}\mathcal{K}_{n}\cos{\bar\varphi_{n}} > 0.
  \end{align*}
This shows the statement for up to $n = q-1$; an analogous computation then
shows that also $f_{q} > 0$, which concludes the proof of our lemma.
\end{proof}

\begin{lemma}\label{lem:length-asymptotics}
  Let $\sigma = (\sigma_{0}\cdots\sigma_{q-1})$ with
  $\sigma_{i} = \sigma_{q-i}$ encode a palindromic periodic orbit of
  period $q$ and Lyapunov exponent
  $\mathrm{LE}(\sigma)=-\frac{1}{q}\log \lambda$, $\lambda=\lambda(\sigma)$ being the contracting eigenvalue of $D\mathcal{F}^q$ at $\sigma$; %
  let $\tau$ (of length $p$) be so that both $\sigma\tau$ and
  $\tau\sigma$ are admissible.  %
  Let $h_{n}(\sigma,\tau)$ denote the periodic orbit encoded by
  $(\tau\sigma^{n})$ of period $p+nq$.  There exists
  $C_{0}(\sigma,\tau)$ and $C_{1}(\sigma,\tau)$ so that:
  \begin{align}\label{eq:length-asymptotics}
    \mathcal L(h_{n}(\sigma,\tau)) - n\mathcal L(\sigma)%\\
    &= C_{0}(\sigma,\tau) +%
       C_{1}(\sigma,\tau)\lambda^{n} +%
       o(\lambda^{n}).
  \end{align}
\end{lemma}
\begin{proof}
  The existence of $C_{0}(\sigma,\tau)$ follows by the same arguments
  used in the definition of $\mathcal{L}^{\infty}$ (see also
  Proposition~\ref{prop bdskl}).  Obtaining the remaining terms
  follows step-by-step by the proof of Proposition~\ref{prop lourde
    un}, where it is proved in the special case $\sigma = (12)$ and
  $\tau = (32)$, and obtains an explicit expression for the coefficient
  $C_{1}((12),(32))$, which we do not need in this lemma.  It is
  omitted in the interest of keeping the length of this paper under
  control.
\end{proof}
\begin{proof}[Proof of Proposition~\ref{prp:palindromic_lyapunov}]
  Using Lemma~\ref{lem:length-asymptotics}, and taking the limit
  of~\eqref{eq:length-asymptotics} for $n\to\infty$ , we conclude that
  $C_{0}(\sigma,\tau)$ is a \mlsinv{}, since the left hand side is
  spectrally determined.  Hence, again
  by~\eqref{eq:length-asymptotics}, we gather:
  \begin{align*}
    \lambda= \lim_{n\to\infty}
    \frac1{n}\log |\mathcal L(h_{n}(\sigma,\tau))%
    - n\mathcal L(\sigma) - C_{0}(\sigma,\tau)|.
  \end{align*}
Since the right hand side is spectrally determined, we conclude that
$\lambda$ is \mlsinv{}.  Similar considerations show that
$C_{1}(\sigma,\tau)$ is also \mlsinv{}.
\end{proof}
Proposition~\ref{prp:palindromic_lyapunov} has the following immediate
corollary:
\begin{theorem} \label{rayon courbure 12} Let
  $\mathcal{D}\in\bsym(m)$; consider the $2$-periodic orbit encoded by
  the word $\sigma = (12)$ and let $R=\mathcal{K}^{-1}$ be the
  (common) curvature radius at the collision points of the orbit.
  Then $R$ is a \mlsinv{}.
\end{theorem}
\begin{proof}
  By Theorem~\ref{prp:palindromic_lyapunov}, the Lyapunov exponent
  $\mathrm{LE}(\sigma)$ of the $2$-periodic orbit is a \mlsinv{}.  Hence,
  $\mls(D)$ determines the eigenvalues of the linearization of the
  square of the billiard map $D\mathcal F^{2}$ at the collision points
  of the $2$-periodic orbit; by~\eqref{matrice sl deux} we have
  \begin{align*}
    D\mathcal{F}^{2}_{(0,0)} = \left(
    \begin{array}{cc}
      \mathscr L\mathcal K+1& \mathscr L\\ \mathscr L\mathcal K^{2} +
      2\mathcal K& \mathscr L \mathcal K+1
    \end{array}
    \right)^{2},
  \end{align*}
where $\mathscr L:=\frac 12 \mathcal L((12))$.
In particular, we have that $\lambda^{1/2}+\lambda^{-1/2} = 2(\mathscr
L\mathcal K+1)$; since it is of course
spectrally determined, we conclude that $R=\mathcal K^{-1}$ is \mlsinv{}.
\end{proof}
%\begin{remark}\label{prp:quadratic-form-mls-invariance}
%  If $(\bar s_{0},\bar s_{1})$ corresponds to the $(12)$ periodic
%  orbit, then the calculations in Lemma~\ref{lem:quadratic-form} yield
%  the expression
%  \begin{align*}
%    \partial_{ii}L_{1}(\bar s_{0},\bar s_{1}) &=  \frac1{\mathcal L}+\frac1{R},
%  \end{align*}
%  where $\mathcal L = \mathcal L((12))/2$ is the distance between the
%  collision points (a \mlsinv{}) and $R$ is the (common) curvature
%  radius at the collision points, also a \mlsinv{} by the above
%  lemma.  We conclude that $\mathrm{tr}\,\partial^{2}L_1$ is a \mlsinv{}.
%\end{remark}

\begin{remark}
  As a matter of fact, Theorem~\ref{prp:palindromic_lyapunov} holds
  also for non-palindromic orbits.  It must be noted that in the
  general case there is no analog of
  Lemma~\ref{lem:positive-definite}, since there is no sub-orbit
  contained in a periodic orbit that is a length minimizer (otherwise
  we would have an orthogonal collision, which would force the orbit
  to be palindromic).  Because of this, one needs to show an
  additional cancellation for $\Sigma^{1}_{n}$ and $\Sigma^{2}_{n}$;
  the proofs are marginally more involved but, in the interest of
  keeping this paper short but self-contained, we have only stated the
  result in the palindromic case.
\end{remark}

\end{document}